\documentclass{amsart}
\usepackage[utf8]{inputenc}
\usepackage{amsfonts}
\usepackage{fontenc}

\pdfoutput=1
\usepackage[headings]{fullpage}
\usepackage{amssymb,epic,eepic,epsfig,amsbsy,amsmath,amscd,color,amsthm}
\usepackage[all,cmtip]{xy}
\usepackage{amssymb,amsmath,amscd,color}
\usepackage{graphicx}
\usepackage{texdraw}
\usepackage{url}
\usepackage{mathrsfs}
\usepackage[cal=boondox,scr=boondoxo]{mathalfa}

\usepackage{hyperref, bookmark}

\newtheorem{theorem}{Theorem}[section]
\newtheorem{thm}{Theorem}
\newtheorem{lemma}[theorem]{Lemma}
\newtheorem{coro}[theorem]{Corollary}
\newtheorem{prop}[theorem]{Proposition}

\newtheorem{defn}[theorem]{Definition}
\newtheorem{rmk}[theorem]{Remark}
\newtheorem{example}[theorem]{Example}

\newtheorem{Not}[theorem]{Notation}

\usepackage{tikz}
\tikzstyle cross=[preaction={draw=white, -, line width=6pt}]
\tikzstyle normal=[thick]
\usepackage{braids}
\usetikzlibrary{arrows,decorations.markings}
\usepackage{relsize,exscale}
\usetikzlibrary{cd}

\usetikzlibrary{patterns,arrows,decorations.pathreplacing}

\newcommand{\hyellow}{yellow!80!black}

\makeatletter
\newcommand*{\encircled}[1]{\relax\ifmmode\mathpalette\@encircled@math{#1}\else\@encircled{#1}\fi}
\newcommand*{\@encircled@math}[2]{\@encircled{$\m@th#1#2$}}
\newcommand*{\@encircled}[1]{%
  \tikz[baseline,anchor=base]{\node[draw,circle,outer sep=0pt,inner sep=.2ex] {#1};}}
\makeatother

\newcommand{\BC}{\mathbb{C}}
\newcommand{\BN}{\mathbb{N}}

\newcommand{\BR}{\mathbb{R}}
\newcommand{\BS}{\mathbb{S}}
\newcommand{\BZ}{\mathbb{Z}}

\newcommand{\CB}{\mathcal{B}}
\newcommand{\CE}{\mathcal{E}}
\newcommand{\CF}{\mathcal{F}}
\newcommand{\CH}{\mathcal{H}}
\newcommand{\CK}{\mathcal{K}}
\newcommand{\CP}{\mathcal{P}}
\newcommand{\CR}{\mathcal{R}}
\newcommand{\CW}{\mathcal{W}}

\newcommand{\fg}{\mathfrak{g}}

\newcommand{\mq}{\operatorname{\mathtt{q}}}

\newcommand{\br}{\mathbf{r}}
\newcommand{\bz}{\mathbf{z}}

\newcommand{\Bn}{\mathcal{B}_n}
\newcommand{\PBn}{\mathcal{PB}_n}

\newcommand{\Sk}{\mathfrak{S}}
\newcommand{\sign}{\operatorname{sign}} 

\newcommand{\Mod}{\operatorname{Mod}}

\newcommand{\bapp}{\left. \begin{array}{ccc}}
\newcommand{\eapp}{\end{array} \right.}
\newcommand{\bfct}{\left\lbrace \begin{array}{ccc}}
\newcommand{\efct}{\end{array} \right.}

\newcommand{\Conf}{\operatorname{Conf}}
\newcommand{\Coloring}{\operatorname{Col}}
\newcommand{\flip}{\operatorname{flip}}

\newcommand{\Hlf}{\operatorname{H} ^{\mathrm{BM}}}
\newcommand{\Hnot}{\operatorname{H}}

\newcommand{\Clf}{\operatorname{C}^{\mathrm{BM}}}

\newcommand{\CHbm}{\CH^{\mathrm{BM}}}
\newcommand{\CHe}{\CH^{\mathrm{e}}}

\newcommand{\Laurent}{\CR}

\newcommand{\slt}{{\mathfrak{sl}(2)}}

\newcommand{\Uq}{{U_q\slt}}
\newcommand{\Uqg}{{U_q\fg}}
\newcommand{\Uqgm}{{U_q\fg^{<0}}}
\newcommand{\qSerre}{\operatorname{qSerre}}
\newcommand{\T}{\operatorname{\mathcal{T}}}

\newcommand{\qbin}[2]{\left[\begin{array}{c}
      #1 \\
      #2 \end{array}\right]}   
      
\newcommand{\Aut}{\operatorname{Aut}}
\newcommand{\Hom}{\operatorname{Hom}}

\def\Id{\operatorname{Id}}

\newcommand{\splits}{\mathrm{split}}
\newcommand{\del}{\operatorname{del}}
\newcommand{\rot}{\operatorname{rot}}
\newcommand{\im}{\operatorname{Im}}
\newcommand{\str}{\mathrm{stretch}}
\newcommand{\cut}{\mathrm{cut}}

\newcommand{\AM}{\mathrm{AW}}
\newcommand{\diag}{\mathrm{diag}}

\title{Quantum groups from homologies of configuration spaces}
\author{Stephen Bigelow \& Jules Martel}

\begin{document}

\begin{abstract}
We reconstruct a quantum group associated with any Lie algebra together with its representation theory from twisted homologies of generalized configuration spaces of disks. Along the way it brings new combinatorics to the theory, but our diagrams represent true submanifolds of configuration spaces and combinatorial relations between them translate actual twisted homological relations. 
\end{abstract}

\maketitle


\section{Introduction}

\subsection{Quantum groups}

A \textit{quantum group} is a Hopf algebra derived from a one parameter quantized deformation of the enveloping algebra of a given Lie algebra. If the Lie algebra is $\mathfrak{g}$ then the corresponding quantum group is denoted $\Uqg$, the notation emphasizes that it depends on a parameter $q$ which can also be treated as a formal variable. They were invented by Drinfel'd and Jimbo while Lusztig has developed the integral versions of the theory. These algebras are also known as Drinfel'd--Jimbo algebras, or Lusztig quantum groups depending on the version one uses. Their representation theory is very rich, and was applied in many contexts (see e.g. \cite{Maj,CP}). In the present paper we find quantum groups and their modules in nature, namely in twisted homologies of configuration spaces of disks. Recovering these structures from topology of smooth manifolds sheds light on the topological content of these algebras that were extensively used so to built \textit{quantum topological invariants} but using their algebraic characterizations.  

We fix notation from now on and for the whole paper. Let $\fg$ be a semisimple Lie algebra (see e.g. \cite{Hum}) with which is associated a root system denoted $\Phi$. The root system is completely described by a Cartan matrix (see \cite[11.1]{Hum}) or equivalently by a set of simple roots denoted $\Pi:=\lbrace \alpha_1,\ldots, \alpha_l \rbrace$, plus the entries of the Cartan matrix $\langle \alpha_i, \alpha_j \rangle := \frac{2(\alpha_i,\alpha_j)}{(\alpha_j,\alpha_j)}$ following notation from \cite{Hum,Jan}, where $(\cdot,\cdot)$ is the inner product. In \cite{Lus90}, things are defined equivalently with the matrix of $\langle \cdot , \cdot \rangle$ being given by $(a_{i,j})_{1\le i,j \le l }$ and the existence of a vector $(d_1 , \ldots, d_l)$ s.t. the matrix $(d_i a_{i,j})_{i,j}$ is symmetric, positive definite, $a_{i,i} = 2 $ and $a_{i,j} \le 0$. We consider these data fixed in all what follows and we denote $(\cdot, \cdot )$ the inner product given by the symmetric matrix. \textbf{These data are fixed for the the whole paper, hence all the results hold for any such Lie algebra $\mathfrak{g}$.}

With any such Lie algebra $\fg$ is thus associated the {\em quantum universal enveloping algebra} that we denote $\Uqg$. The Drinfel'd--Jimbo version is the $\mathbb{Q}(q)$-algebra generated by $E_{\alpha}, F_{\alpha}, K_{\alpha}^{\pm 1}$ for all $\alpha \in \Pi$, which are subject to the relations:

\begin{gather*}
 K_\alpha K_\alpha^{-1} = K_\alpha^{-1} K_\alpha = 1, \\*
 K_{\alpha_i} E_{\alpha_j} K_{\alpha_i}^{-1} = q_{\alpha_i}^{-\frac{a_{i,j}}{2}}E_{\alpha_j}, \qquad K_{\alpha_i} F_{\alpha_j}^{(1)} K_{\alpha_i}^{-1} = q_{\alpha_i}^{\frac{a_{i,j}}{2}} F_{\alpha_j}, \qquad
 [E_\alpha,F_\alpha] = \frac{K_\alpha - K_\alpha^{-1}}{2},
\end{gather*}
and the so called \textit{quantum Serre relations}, for $\beta\neq \alpha$:
\begin{align*}
& \sum_{l=0}^{1-a_{i,j}} (-1)^l \binom{1-a_{i,j}}{l}_{q_\alpha} F_{\alpha}^{l} F_{\beta} F_{\alpha}^{1-a_{i,j}-l} = 0, \\
& \sum_{l=0}^{1-a_{i,j}} (-1)^l \binom{1-a_{i,j}}{l}_{q_\alpha} E_{\alpha}^{l} E_{\beta} E_{\alpha}^{1-a_{i,j}-l} = 0.
\end{align*}
where $q_\alpha:=q^{d_\alpha}$ (notice that we will use slightly different conventions in the paper but we will display in detail the definitions we use). We refer the reader to \cite[Sec.~6.1.2]{KS} for this definition. Notice that some \textit{quantum binomials} appear in the definition, namely elements $\binom{1-a_{i,j}}{l}_{q_\alpha}$ are polynomials in $q$ that will be also precisely defined. In the $\slt$-case there is only one simple root, and so $\Uq$ has only three generators and no quantum Serre relations. We also denote by $\Uqgm$ the subalgebra generated only by the $F$ generators. The Lusztig quantum group is defined on the ring $\BZ[q^{\pm 1}]$ rather than a field, and (infinitely many) more generators are necessary for its definition, the so called \textit{divided powers} of $F$'s and $E$'s encoding a particular exponentiation of generators. We will recover them too and will define them precisely together with the relation they involve. 

These algebras are the main and major example of Hopf algebras meaning bialgebras with an antipode that are neither commutative nor cocommutative. It makes their representation theory very rich and all set to build topological invariants of manifolds. Their categories of modules are naturally braided monoidal which has opened the way to \textit{quantum topology} that studies the relations between low dimensional topology and monoidal categories in a broad sense. We refer to (among many others) \cite{Kas,O,CP} for the applications in topology. Notice that the constructions use mostly the semisimple part of the representation theory, while more recently new constructions arose taking advantage also of the non-semisimple part of the theory providing invariants of a new kind but still utilizing inherent tools from quantum groups to actually perform the construction. A first book on the topic is \cite{KL}. 

Notice that Lusztig has introduced a sheaf theoretic approach to quantum groups (see e.g. \cite{Lusztig_geom,LusztigBook}). One purpose was to find canonical bases (\cite{Lus90b}). On the other hand quantum groups were categorified in a series of papers (\cite{Lau_sl2,KhLa}, \cite{Rouquier1,Rouquier2}) as the decategorification of a module category on the so called Khovanov--Lauda--Rouquier algebras. It has rejoined the Lusztig geometric approach to canonical bases in \cite{Michela}. We hope that our new topological construction will fit with this whole picture and to find relations with these models. Categorifications are in close relation with link homologies, and this motivates a topological construction. 

For what concerns this topological approach, in \cite{Stephen_Hecke} the first author has recovered Iwahori--Hecke algebras representations (closely relatives of $\Uq$-algebras) from twisted homology of configuration spaces of punctured disks following ideas of Lawrence. In \cite{Jules_Verma}, the second author has rebuilt part of the representation theory of $\Uq$ using the same kind of setup, and it was inspired by a topological model by Felder and Wieczerkowski \cite{FW}. In \cite{JulesMarco} they recover non semisimple modules of $\Uq$ at roots of unity from configuration spaces on higher genus surfaces, and they use them to rebuild Kerler-Lyubashenko's topological quantum field theories \cite{KL}, inspired by a topological model by Crivelli--Felder--Wieczerkowski for the torus. In \cite{Pierre}, P. Godfard uses the evaluation of modules from \cite{Jules_Verma} at roots of unity so to study homologically the representation theory of $\Uq$ at roots of unity involved in Reshetikhin--Turaev constructions and to relate them with invariants from Hodge structures on configuration spaces. The present work generalizes these constructions for the $\slt$-case in several directions and above all to all semisimple Lie algebras. Notice that twisted homologies of configuration spaces of a surface more or less systematically carry an action of the mapping class group of the surface. On the quantum topology side, the converse has been done every time: building mapping class group representations from quantum groups modules. We find quite natural homological operators endowing homologies with a quantum group action commuting with the existing natural action of mapping class groups. 

\subsection{Content of the paper}

The first section, Section~\ref{BS:homology_etc} is devoted to defining the setup of twisted homologies on configuration spaces of disks. Let $D$ be the unit disk, and $c \in \BN[\Pi]$ that we call a \textit{coloring by simple roots}. We associate with this $c$ a \textit{colored configuration space} denoted by $\Conf_c$ which shall be thought as a space of sets of points in $D$ that are decorated by simple roots. It is a smooth manifold precisely defined in Section~\ref{S:conf_space}. We define in Section~\ref{S:local_system_empty_disk} a \textit{local system} on $\Conf_c$ that is denoted $\Laurent_c$ and takes values in the ring of Laurent polynomials $\Laurent=\BZ[q^{\pm 1}]$. It is defined by monodromies of some loops, but getting rid of using just one base point which happens to be useful in computations. Now Section~\ref{S:twisted_homology_bases_etc} defines the homology modules of interest. Let:
\[
\CHbm_c := \Hlf_{m_c}(\Conf_c,S;\Laurent_c),
\]
that is the Borel--Moore homology of $\Conf_c$ relative to part of its boundary denoted $S$, twisted by the local system $\Laurent_c$ and where $m_c:= \sum_{\alpha\in \Pi} c(\alpha)$. We also define $\CH_c$ to have the same definition but with the standard singular homology rather than the Borel--Moore one. Both are modules over $\BZ[q^{\pm 1}]$. We let $\CP_c$ be the set of \textit{partitions of $c$} that is lists of integers decorated by $c$. For any $\br \in \CP_c$ we define a \textit{pearl necklace} diagram denoted $\CF_\br$ that naturally defines a class in $\CHbm_c$, and similarly a class $\CF^{[\br]}$ that naturally defines one in $\CH_c$. Both families of homology classes happen to be dual for a perfect pairing introduced in the Appendix~\ref{A:the_pairing}. This bilinear form is a natural intersection form at the twisted homology that computes twisted algebraic intersection between manifolds. We use it to prove that both $\CHbm_c$ and $\CH_c$ are free modules of which the two mentioned dual families define bases, in Propositions~\ref{structure_result} and \ref{T:structure_standard_homology}. The Section~\ref{S:computation rules} displays a set of relations between diagrams defining homology classes describing relations in homology, they will be extensively used to do computations in $\CHbm_c$ and $\CH_c$. They give the very homological flavor to our constructions. 

In Section~\ref{BS:homological_algebra} we put an algebra structure on both:
\[
\CH:= \bigoplus_{c\in \BN[\Pi]} \CH_c \hspace{1cm} \CHbm:= \bigoplus_{c\in \BN[\Pi]} \CHbm_c.
\]
The product is given by stacking disks hence by a geometrical operation. While defining Borel--Moore classes denoted $\CF_\alpha^{(k)}$ for a given root $\alpha$ and a given integer $k$, we prove that they satisfy the divided power relation holding in integral versions of quantum groups (Prop.~\ref{prop_DividedPowers}). In Section~\ref{S:quantum_Serre} we first introduce a particular homology class that we denote $\qSerre_{\alpha,\beta}^{(k)}$ associated with two simple roots $\alpha,\beta$ and an integer $k$. Expressing it in two ways gives the following result.
\begin{theorem}[Prop.~\ref{prop_QuantumSerre}, Theorem~\ref{T:homological_version_for_Uqgm}]\label{T:Uqgm_homological_intro}
The elements $\{\CF^{(k)}_\alpha, \alpha \in \Pi, k\in \BN \}$ satisfy the quantum Serre relations, and the divided power relation. Thus there is a homomorphism of $\BZ[q^{\pm 1}]$-algeras:
\[
\Uqgm \to \CHbm := \bigoplus_{c\in \Coloring_\Pi} \CHbm_c. 
\]
\end{theorem}
There is a canonical map that sends $\CH:=\bigoplus_{c\in\Coloring_\Pi} \CH_c$ to $\CHbm$ induced by the inclusion of chain complexes. We denote by $\overline{\CH}$ its image, and the next section (\ref{S:homology_to_Borel_Moore}) is devoted to proving the result:
\begin{theorem}[Theorems~\ref{T:free_algebra} and \ref{T:Uqgm_is_homological}]
$\CH$ is a free algebra generated by elements $\{\CF_\alpha^{[\alpha]}, \alpha \in \Pi \}$ and $\overline{\CH}$ is isomorphic to $\Uqgm$. 
\end{theorem}
To prove this, along the way we recover homologically a coproduct on $\Uqgm$ introduced by Lusztig and show it satisfies the good compatibility axioms with a natural intersection pairing at homology defined on $\CH$ (Prop.~\ref{P:recovering_r} and \ref{P:recovering_coproduct}). Again, we compare this coproduct and pairing introduced by Lusztig in \cite{LusztigBook} with the perfect bilinear intersection form at homology. It is well known that Lusztig has also provided Poincaré--Birkhoff--Witt type bases for quantum groups (see e.g. \cite{Lus90}). They now give bases to $\overline{\CH}$ thanks to the above theorem. They are parametrized by the action of a braid group on $\Uqgm$ that we also recover by purely geometrical means in Section~\ref{S:PBW}, see Corollary~\ref{C:Ti_is_Lusztig}. 

Last section, Section~\ref{BS:rep_theory}, aims at rebuilding the representation theory of $\Uqg$. We extend definitions of colored configuration spaces and of twisted homologies to the case of an $n$-time punctured disk denoted $D_n$. Since we study the weight modules of $\Uqg$ they are characterized by a choice of character that is a function on simple roots. We choose one formal variable for the image of each simple root by such a character, and we choose one set of formal variables per puncture. We denote by $\boldsymbol{L}$ the whole set of formal variables and by $\Laurent_{\boldsymbol{L}}$ the ring of Laurent polynomials in variables from $\boldsymbol{L}$ and $q$. We put a local system on $\Conf_c(D_n)$ with coefficients in $\Laurent_{\boldsymbol{L}}$ and we let then $\CH(D_n)$ resp. $\CHbm(D_n)$ be the analogs of $\CH$ resp $\CHbm$ but built from $D_n$ and its local system encoding the character. We prove that both are free modules, identified by a perfect intersection pairing, and we provide them with diagrammtic bases in Propositions~\ref{P:structure_result_punctures} and \ref{P:structure_result_punctures_standard}. By $\Uqg^E$ we mean the subalgebra from the integral version of $\Uqg$ obtained by removing the quantum Serre relations for $E$ generators (but with divided powers for $F$'s). 
\begin{theorem}[Theorem~\ref{T:half_integral_modules}]
The module $\CHbm(D_n)$ is endowed with an action of $\Uqg^E$. 
\end{theorem}
In the above, we first put an action of $\CHbm$ that defines one of $\Uqgm$ thanks to Theorem~\ref{T:Uqgm_homological_intro} and that is given by a geometric operation on disks. We add an action of generators $E_\alpha$'s by a natural homological operator provided by the boundary map involved in the long exact sequence of a triple (Def.~\ref{D:homological_E_half_integral}). Let $\Uqg^F$ be the subalgebra of $\Uqg$ obtained by removing quantum Serre relations for $F$'s (but keeping divided powers for $E$'s). 
\begin{theorem}[Theorem~\ref{T:adjoint_modules_to_half_integral}]
There is an action of $\Uqg^F$ on $\CH(D_n)$ that is left adjoint to that of $\Uqg^E$ on $\CHbm$ regarding the intersection form.
\end{theorem}  
In the above again, actions for divided powers of $E$'s are defined by natural cohomological operators. We also introduce $\CH(D_n^\circ)$ that is built from the disk with holes rather than punctures, so that $\CHbm(D_n^\circ)$ injects in $\CHbm(D_n)$. These last two modules are called half integral Verma and coVerma respectively. Now if $\overline{\CH}(D_n)$ denotes the image of $\CH(D_n)$ in $\CHbm(D_n)$ induced by the natural inclusion of chain complexes, we have the following result. 

\begin{theorem}[Theorem~\ref{T:full_Uqg_module}]
The module $\overline{\CH}(D_n)$ is a module on the full Drinfel'd--Jimbo $\mathbb{Q}(q)$-algebra $\Uqg$. $\overline{\CH}(D_1)$ is the Verma module.
\end{theorem}

Notice that any simple module of $\Uqg$ can be recovered from the universal Verma modules discussed above, and it will be the subject of some future notes. One main feature of $\Uqg$-modules is that they form a monoidal category. It turns out that adding punctures corresponds on the algebraic side to making tensor products, so that we recover this monoidal structure again by simple geometrical operation translated at homology. 

\begin{theorem}[Theorem~\ref{T:monoidality}]
If $D_n$ is the connected sum of two punctured disks $D$ and $D'$, then there is an isomorphism:
\[
\CHbm(D_n) \simeq \CHbm(D') \otimes \CHbm(D)
\]
that is equivariant for $\Uqg$-module structures, given by the coproduct on the right term. 
\end{theorem}

In the last Section~\ref{S:braiding} we discuss the existence of an action of the braid group on $n$ strands denoted $\CB_n$ on $\CHbm(D_n)$ that comes from an $R$-matrix of $\Uqg$. Again it is recovered by using an action of the braid group by mapping classes at homology. We will further investigate its relation with knot theory, but thus starting with a total topological point of view. 

There is an Appendix \ref{A:the_pairing} that presents recalls and references on definitions of various homologies involved in our constructions, Poincaré duality and its realisation by a twisted intersection pairing. 

\subsection*{Acknowledgments}

J.M. is very grateful to P. Godfard for all the discussions on homology theories (and the shortest proof of Prop.~\ref{P:standard_homology_is_free}), and to R. Maksimau for discussions on quantum groups that motivated him on the subject.

\section{Generalized configuration spaces and twisted homology}\label{BS:homology_etc}

\subsection{Colored configuration spaces}\label{S:conf_space}
In the paper, $D$ is the unit disk. It will be considered topologically, as sometimes drawings will show squared disks for clarity. 

\begin{Not}
We suppose the set of simple roots to be fixed, $\Pi:=\lbrace \alpha_1,\ldots, \alpha_l \rbrace$. Then:
\begin{itemize}
\item A {\em coloring} by the roots $c$ is a function:
\[
c: \Pi \to \BN.
\]
The reader should think of such a coloring by $c(\alpha)$ indicating the number of points ``colored'' by the root $\alpha$ in configuration spaces below. Let $\Coloring_{\Pi}$ be the set of all colorings (by $\Pi$). Notice that $\Coloring_{\Pi} = \BN[\Pi]$ which is the notation used in \cite{LusztigBook}, using the assimilation $$c=\sum_{\alpha\in \Pi} c(\alpha)\alpha$$. 
\item The \textit{weight of a coloring} $c\in \Coloring_{\Pi}$, denoted $m_c$, is defined to be:
\[
m_c := \sum_{i= 1, \ldots, l} c(\alpha_i) 
\]
\end{itemize}
\end{Not}

 From a fixed coloring we define an associated configuration space.

\begin{defn}[Colored configuration space, {\cite[Definition~3.1.1]{Jules_phd}}]
Let $c \in \Coloring_{\Pi}$, we denote by $\Conf_c$ the $c$-colored configuration space, defined as follows:
\[
\Conf_c = \left( D^{m_c} \setminus \bigcup_{i<j} \{z_i=z_j\} \right) \Big/ \Sk_{c(\alpha_1)} \times \cdots \times \Sk_{c(\alpha_l)}
\]
\end{defn}

\begin{rmk}
An element of $\Conf_c$ is denoted $\left(\lbrace z^{\alpha_1}_1 , \ldots , z^{\alpha_1}_{c(\alpha_1)} \rbrace, \lbrace z^{\alpha_2}_1 , \ldots , z^{\alpha_2}_{c(\alpha_2)} \rbrace, \ldots , \lbrace z^{\alpha_l}_1 , \ldots , z^{\alpha_l}_{c(\alpha_l)} \rbrace \right)$ and can be seen as $l$ packages of undistinguishable points, one of $c(\alpha_1)$ points colored by $\alpha_1$, then $c(\alpha_2)$ points colored by $\alpha_2$ and so on. The exponent of a coordinate indicates its color. An even simpler notation could be employed for the above configuration: $(\boldsymbol{z^{\alpha_l}}, \ldots, \boldsymbol{z^{\alpha_l}})$ where every $\boldsymbol{z^{\alpha_i}}=\lbrace z^{\alpha_i}_1 , \ldots , z^{\alpha_i}_{c(\alpha_i)} \rbrace$ is a set of undistinguished points in $D$ colored by $\alpha_i$. 
\end{rmk}

\subsection{Local system}\label{S:local_system_empty_disk}

\begin{Not}\label{N:writhe_etc}
We define the following function $f$ on any list $\lbrace z_1, \ldots z_n \rbrace \in \left( D^n \setminus \bigcup_{i<j} \{z_i=z_j\} \Big/ \Sk_{n} \right)$:
\begin{equation}
f \left( \lbrace z_1 , \ldots , z_n \rbrace \right) := \prod_{1<i<j<n} \left( \frac{z_j-z_i}{|z_j-z_i|} \right)^2 = \prod_{1<i,j<n} \left( \frac{z_j-z_i}{|z_j-z_i|} \right) \in \BS^1,
\end{equation}
it is the {\em writhe} of the list $\lbrace z_1 , \ldots , z_n \rbrace$. We also define the writhe of one list relatively to the other:
\begin{equation}
g \left( \lbrace z_1 , \ldots , z_n \rbrace, \lbrace w_1 , \ldots , w_k \rbrace \right) := \prod_{ 1<i<n , 1 < j < k } \left( \frac{w_j-z_i}{|w_j-z_i|} \right)^2 \in \BS^1,
\end{equation}
where $\BS^1$ is the unit circle. 
\end{Not}

\begin{rmk}
In the previous definition, the square in the definition of $f$ (first expression) is for making it well defined (on an unordered list). The square in $g$ is not required for it to be well defined. We define it like that for some reasons, e.g. the following remarks: 
\begin{itemize}
\item One can apply any permutation to the concatenated list $\lbrace z_1 , \ldots , z_n , w_1 , \ldots , w_k \rbrace$ leaving $g$ invariant. It includes the transposition of the two lists. 
\item A consequence is that horizontally aligned configurations are sent to $1$ (it is also the case for $f$).
\item It makes quantum computations work, as the present work shows.  
\end{itemize}
\end{rmk}

\begin{defn}
Associated with a coloring $c$, we consider the map $\Phi_c$ defined as follows:
\[
\bapp
\Conf_c & \to & (\BS^1)^{l} \times (\BS^1)^{\frac{l(l-1)}{2}} \\
(\bz^{\alpha_1}, \ldots , \bz^{\alpha_l}) & \mapsto & \left( f(\bz^{\alpha_1}), \ldots , f(\bz^{\alpha_l}) , g(\bz^{\alpha_1},\bz^{\alpha_2}), g(\bz^{\alpha_1}, \bz^{\alpha_3}), \ldots , g(\bz^{\alpha_2},\bz^{\alpha_3}) , \ldots , g(\bz^{\alpha_{l-1}},\bz^{\alpha_l}) \right)
\eapp
\]
We slightly modify it, fixing $\CW_c:= \Phi_c \circ r_{\frac{\pi}{2}}$, where $r_{\frac{\pi}{2}}$ is the rotation of $\frac{\pi}{2}$ applied (by abuse of notation) to all coordinates of elements of $\Conf_c$. It is the first ingredient for the definition of a local system on $\Conf_c$.
\end{defn}

\begin{rmk}
Some remarks about $\CW_c$. 
\begin{itemize}
\item The letter $\CW$ stands for ``writhe'' as $\CW_c$ is the writhe function associated with the coloring $c$. 
\item We make the remark that it can be expressed using other particular functions in two interesting ways:
\begin{align*}
f \left( \lbrace z_1 , \ldots , z_n \rbrace \right) & = \prod_{i<j<n}  \frac{z_j-z_i}{\left(\overline{z_j-z_i}\right)} 
\\ & = \frac{V(z_1,\ldots, z_m)}{{\overline{V(z_1,\ldots, z_m)}}},
\end{align*}
where $V$ designates the Vandermonde function (i.e. the determinant of the Vandermonde matrix associated with the list). The fraction of the second line computes twice the argument of $V(z_1,\ldots, z_m)$. 
\item The first $l$ coordinates of $\CW_c$ are writhes of packages of points colored by the same root.
\item Last $\frac{l(l-1)}{2}$ coordinates are writhes of a package colored by a root relatively to a different one. The reader would have guessed the classical order on coordinates applied here. 
\item $\CW_c$ sends vertically aligned configurations to $(1,\ldots,1) \in (\BS^1)^{\times l} \times (\BS^1)^{\frac{l(l-1)}{2}}$. (The application of the $\frac{\pi}{2}$-rotation is considered only because we prefer vertically aligned configurations to have this property, rather than horizontally aligned ones). 
\end{itemize}
\end{rmk}

\begin{Not}
\begin{itemize}
\item We denote $B_c := (\BS^1)^{l} \times (\BS^1)^{\frac{l(l-1)}{2}}$ and we choose a base point $b_c := (1, \ldots, 1)$ of it. We denote the fundamental group of $B_c$ based at $b_c$ by $\pi_c$ and we recall that:
\[
\pi_c \simeq \BZ^l \times \BZ^{\frac{l(l-1)}{2}} := \BZ \langle k_1 \rangle \oplus \ldots \BZ \langle k_l \rangle \oplus \BZ \langle k_{(1,2)} \rangle \oplus \langle k_{(1,3)} \rangle \oplus \cdots \oplus \BZ \langle k_{(l-1,l)} \rangle . 
\]
where we have chosen a generator for each copy of $\BS^1$ involved corresponding to a counterclockwise loop. 
\item We recall that the fundamental group of $\Conf_c$ (up to a chosen base point) is made of braids (colored by roots) having induced permutation in $\Sk_{c(\alpha_1)} \times \cdots \times \Sk_{c(\alpha_l)}$. Let's consider the map $\CW_c$ at the level of fundamental groups. Then the coordinate of the image of a loop in $\Conf_c$ along $k_i$ is the writhe of strands of the loops colored by $\alpha_i$, while the coordinate along $k_{(i,j)}$ is the writhe of strands colored by $\alpha_i$ relatively to those colored by $\alpha_j$. Here we recover the usual definition of ``writhe'' of braids. 
\end{itemize}
\end{Not}

We define a representation of $\pi_c$ that will be the local ring of coefficients with which we will twist the homology.

\begin{defn}[Local system]\label{localsystem}
We define the local ring of interest to be $\Laurent := \BZ \left[ q^{\pm 1} \right]$, and we endow it with a $\BZ \left[ \pi_c \right]$- module structure as follows:
\[
\rho_c : \bfct
\BZ \left[ \pi_c \right]  &  \to & \Laurent \\
k_i & \mapsto & -q^{(\alpha_i,\alpha_i)/2} = -q^{d_i} \\
k_{(i,j)} & \mapsto & q^{(\alpha_i, \alpha_j)/2}.
\efct
\]
We will denote $\Laurent_c$ the entire data set $(\Laurent, \CW_c, \rho'_c)$ reunited under the name of local system (or local ring of coefficients). It is used in next section to endow the homology with a local ring of coefficients isomorphic to $\Laurent$.
\end{defn}

\subsection{Twisted homology: definition, basis and structure.}\label{S:twisted_homology_bases_etc}

We define the twisted homology associated with the local system defined in Definition \ref{localsystem}. First we draw a unit disk in a certain way (as a rectangle with vertical sides in red):
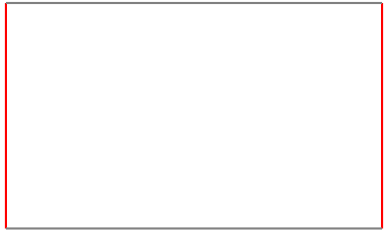
\[
\vcenter{\hbox{\begin{tikzpicture}[scale=0.5]
\draw[red, thick] (-5,-3) -- (-5,3);
\draw[red, thick] (5,-3) -- (5,3);
\draw[gray, thick] (-5,-3) -- (5,-3);
\draw[gray, thick] (-5,3) -- (5,3);
\end{tikzpicture}}}
\]

\begin{defn}[Homologies with twisted coefficients]\label{twisted_homology}
Let $\partial^-D$ be the union of both lateral sides of $\partial D$, marked in red in the figure. We define $S \in \partial \Conf_c$ to be made of configurations having at least a coordinate in $\partial^-D$. Then we denote:
\[
\CHbm_c := \Hlf_{m_c} (\Conf_c, S; \Laurent_c),
\]
where $\Laurent_c$ is the local ring of coefficients from Def. \ref{localsystem}. We emphasize that here we use the twisted \emph{Borel--Moore} homology (or Borel--Moore homology with \textit{local coefficients}) for which we refer the reader to the detailed Appendix in \cite[Appendix~A]{JulesMarco} since we use exactly the same setup. It is here isomorphic to the \textit{locally finite} version of the singular homology if one restricts to \textit{$\pi$-finite chains} for this twisted version, precise references and definitions to these facts are made in the appendix of \cite{Jules_Lefschetz}. 
We recall, what is detailed in the above mentioned Appendix, that these homologies carry an action of $\Laurent$, by means of Deck transformations of a certain regular cover. They are $\Laurent$-modules. Notice also that we deal with the middle homology, in terms of dimension of $\Conf_c$ as a manifold. We define the non Borel--Moore analog:
\[
\CH_c:= \Hnot_{m_c} (\Conf_c, S; \Laurent_c)
\]
and it also carries a $\Laurent$-module structure. In what follows the symbol $\CHe_c$ will stand to designate either $\CHbm_c$ or $\CH_c$. 
\end{defn}

\begin{rmk}[Twisted homology from regular covers]\label{twisted_homology_from_cover}
We recall that:
\[
\CHe_c \simeq \Hnot^{e}_{m_c} (\widehat{\Conf_c}, p_c^{-1}(S); \BZ),
\]
where:
\[
p_c: \widehat{\Conf_c} \to \Conf_c
\]
is the cover associated with the morphism: $\rho_c \circ \CW_c : \BZ\left[\pi_1 (\Conf_c) \right] \to \Laurent$ (up to a choice of base point for $\Conf_c$). As the latter morphism is onto, the deck transformation ring of $\widehat{\Conf_c}$ is isomorphic to $\Laurent$ and its action on $\CHe_c$ recovers the $\Laurent$-module structure of the twisted homology.  
\end{rmk}

We now define particular twisted chains happening to be cycles in the Borel--Moore setup, hence to define classes in $\CHbm_c$. 

\begin{defn}[$\CF_{\alpha}$'s classes]\label{F_classes_1dim}
Let $c_{\alpha}^1$ be the coloring with a single root $\alpha$ coloring a single point ($c(\alpha) = m_c = 1$). We define:
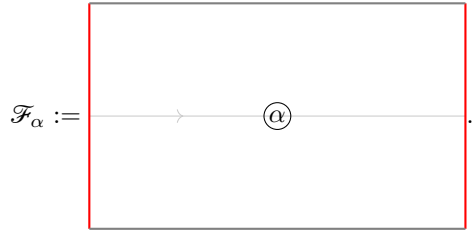
\[
\CF_{\alpha} := 
\vcenter{\hbox{\begin{tikzpicture}[scale=0.5, every node/.style={scale=1},decoration={
    markings,
    mark=at position 0.25 with {\arrow{>}}}
    ]
\coordinate (w0g) at (-5,0) {};
\coordinate (w0d) at (5,0) {};

\draw[gray!40!white, thin,postaction={decorate}] (w0g) -- (w0d) node[pos=0.5, black] (a1) {$\encircled{\alpha}$};

\draw[red, thick] (-5,-3) -- (-5,3);
\draw[red, thick] (5,-3) -- (5,3);
\draw[gray, thick] (-5,-3) -- (5,-3);
\draw[gray, thick] (-5,3) -- (5,3);
\end{tikzpicture}}}. 
\]
We explain how the above diagram defines naturally a class in $\CHbm_{c^1_{\alpha}}$: the middle gray horizontal line is an embedded interval (oriented and marked by $\alpha$). The fact that its ends reach $\partial^-D$ makes it a cycle relatively to $S$ and hence makes it define naturally a class of $\CHbm_{c^1_{\alpha}}$ (which is just the regular homology, trivially twisted for this particular coloring). 
\end{defn}

Then we define a generalization of these classes for higher homologies (i.e. colorings with weight greater than $1$) but it requires more explanations. 

\begin{Not}[Pearl necklace]\label{N:pearl_necklace}
Let $c$ be a coloring by roots of cardinal $m:=m_c$ with $l$ roots. Let $(r_1, \ldots , r_m) \in \lbrace 1, \ldots , l \rbrace^m$ be a partition of the coloring $c$, i.e. be such that the cardinal of $i$ in $(r_1, \ldots , r_m)$ is exactly $c(\alpha_i)$. We denote the set of partitions of $c$ by $\CP_c$. Then we draw:
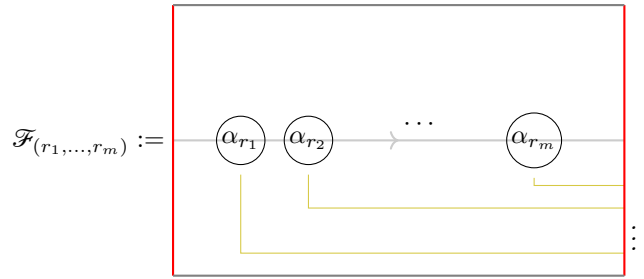
\[
\CF_{(r_1,\ldots , r_m)}:= \vcenter{\hbox{\begin{tikzpicture}[scale=0.6, every node/.style={scale=1},decoration={
    markings,
    mark=at position 0.5 with {\arrow{>}}}
    ]    
    
\coordinate (w0g) at (-5,0) {};
\coordinate (w0d) at (5,0) {};

\coordinate (h1) at (5,-2.5);
\coordinate (h2) at (5,-1.5);
\coordinate (h3) at (5,-1);

\node at (5.2,-2) {$\vdots$};

\draw[gray!40!white, thick, postaction={decorate}] (w0g) -- (w0d) node[pos=0.15,black] (a1) {$\encircled{\alpha_{r_1}}$} node[pos=0.3,black] (a2) {$\encircled{\alpha_{r_2}}$} node[pos=0.55,above,black] {$\cdots$} node[pos=0.8, black] (a3) {$\encircled{\alpha_{r_m}}$};

\draw[yellow!80!black] (h1)--(h1-|a1)--(a1);
\draw[yellow!80!black] (h2)--(h2-|a2)--(a2);
\draw[yellow!80!black] (h3)--(h3-|a3)--(a3);

\draw[red, thick] (-5,-3) -- (-5,3);
\draw[red, thick] (5,-3) -- (5,3);
\draw[gray, thick] (-5,-3) -- (5,-3);
\draw[gray, thick] (-5,3) -- (5,3);
\end{tikzpicture}}}.
\]
The above diagram is made of several elements:
\begin{itemize}
\item A \textit{necklace} $N(\CF_{(r_1,\ldots,r_m)})$ which is the light gray horizontal interval with ends in $\partial^-S$. 
\item The necklace is oriented. When no orientation is indicated, a left to right natural orientation should be implicit. 
\item \textit{Pearls}, namely the marking by roots $\alpha_{r_1}, \ldots, \alpha_{r_m}$ (their order is what matters). 
\item The handle $H(\CF_{(r_1,\ldots,r_m)})$ consisting in the union of ($m$) segments relating $\partial^-D$ to the pearls. It naturally corresponds to a path in $\Conf_c$, also denoted $H(\CF_{(r_1,\ldots,r_m)}) : I \to \Conf_c$ where $I$ is the unit interval. This path starts at a vertically aligned configuration (on $\partial^-D$) and ends at a configuration lying on the necklace. Note that the definition of $\partial^-D$ forces configurations lying on it to be vertically aligned. 
\end{itemize}
\end{Not}

We describe how a pearl necklace naturally defines a twisted homology class. Let $\CF_{(r_1, \ldots, r_m)}$ be a pearl necklace as just defined. 
\begin{itemize}
\item Let $(k^{\alpha_1}_1, \ldots , k^{\alpha_1}_{c(\alpha_1)})$ be the sub ordered list of $(1, \ldots , m)$ corresponding to labels by $\alpha_1$ (i.e. $r_{k^{\alpha_1}_j} = 1$ for any such index) ; then $(k^{\alpha_2}_1, \ldots , k^{\alpha_2}_{c(\alpha_2)})$ those corresponding to the pearls labeled by $2$ and so on. The union of the pearls and the necklace defines the embedding of an $m$-dimensional open simplex:
\begin{equation}\label{embedding_necklace}
\Delta_{\CF_{(r_1, \ldots, r_m)}}: \bfct
\Delta^m = \lbrace 0 < t_1 < \cdots < t_m < 1 \rbrace & \to & \Conf_c \\
(t_1 , \ldots , t_m) & \mapsto & \left( \lbrace N(t_{k^{\alpha_1}_1}) , \ldots , N(t_{k^{\alpha_1}_{c(\alpha_1)}}) \rbrace , \ldots , \lbrace N(t_{k^{\alpha_l}_1}) , \ldots , N(t_{k^{\alpha_l}_{c(\alpha_l)}}) \rbrace\right)
\efct
\end{equation}
where $N$ is here a shorter notation for the embedded interval $N(\CF_{(r_1,\ldots,r_m)})$. Another way to think of this embedding is first to notice that $\Delta^m$ is also the space of configurations of $m$ ordered points in the unit interval $I$. Then by sending $I$ to the necklace $N$ and choosing an order on colors (by roots) it defines the embedding in $\Conf_c$ we were seeking. The pearls are used as a marking choosing the order on colors. 
\item One can check that the embedding from Eq. \eqref{embedding_necklace} is a Borel--Moore cycle relatively to $S$. Faces of the simplex correspond either to configurations in the relative part $S$, either to collisions between points, namely points removed at infinity for defining $\Conf_c$, in Borel--Moore homology closed manifolds going to infinity define cycles.
\item As the simplex is a cycle, the union of the pearls and the necklace naturally defines a class in $\Hlf_{m_c} (\Conf_c, S; \BZ)$.
\item It misses a path from $b_c$ (base point of $B_c$) to $\CW_c(\Delta_{\CF_{(r_1, \ldots, r_m)}})$ to define a twisted homology class out of this regular class. The path $\CW_c(H(\CF_{(r_1,\ldots,r_m)}))$ plays that role. It fixes a lift of the simplex to the regular cover
\item {\bf As a summary:} the union of pearls and necklace naturally defines a Borel--Moore homology class of $\Conf_c$ while the handle is used to define a twisted class. One can think of the handle as a choice of lift of $\Delta_{\CF_{(r_1, \ldots, r_m)}}$ to the cover $\widehat{\Conf_c}$ (see Rem. \ref{twisted_homology_from_cover}), by unique lift property of this path. Or one can also think about the handle as a way to choose a $1$ in the local ring $\Laurent_c$. 
\end{itemize}

We use the word handle, as it was used in \cite{Stephen_linearity,Jules_Verma} and many more papers dealing with that kind of cover of configuration spaces. 

\begin{defn}[Pearl necklaces]\label{homology_generators}
For any $\br \in \CP_c$, the process just described naturally defines a homology class in $\CHbm_c$ from the diagram $\CF_{\br}$. We still denote $\CF_{\br}$ the corresponding class by abuse of notation, as we will always use the diagram as a homology class. 
\end{defn}

\begin{rmk}\label{rmk_Def_diagrams}
We make some remarks on notation and conventions.
\begin{enumerate}
\item One can generalize the definition of a pearl necklace to a multi necklace. For instance, a natural homology class could be assigned with the following figure:
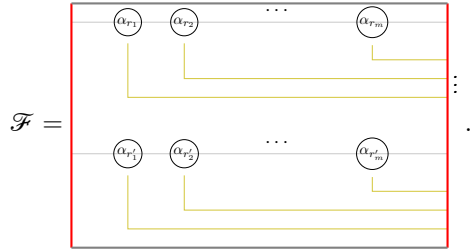
\[
\CF = \vcenter{\hbox{\begin{tikzpicture}[scale=0.5, every node/.style={scale=0.75},decoration={
    markings,
    mark=at position 0.5 with {\arrow{>}}}
    ]    
    
\coordinate (w0g) at (-5,2.5) {};
\coordinate (w0d) at (5,2.5) {};

\coordinate (w1g) at (-5,-1) {};
\coordinate (w1d) at (5,-1) {};

\coordinate (h1) at (5,0.5);
\coordinate (h2) at (5,1);
\coordinate (h3) at (5,1.5);

\node at (5.2,1) {$\vdots$};

\coordinate (h1p) at (5,-3);
\coordinate (h2p) at (5,-2.5);
\coordinate (h3p) at (5,-2);

\node at (5.2,1) {$\vdots$};

\draw[gray!40!white, thin] (w0g) -- (w0d) node[pos=0.15,black] (a1) {$\encircled{\alpha_{r_1}}$} node[pos=0.3,black] (a2) {$\encircled{\alpha_{r_2}}$} node[pos=0.55,above,black] {$\cdots$} node[pos=0.8, black] (a3) {$\encircled{\alpha_{r_m}}$};

\draw[gray!40!white, thin] (w1g) -- (w1d) node[pos=0.15,black] (a1p) {$\encircled{\alpha_{r'_1}}$} node[pos=0.3,black] (a2p) {$\encircled{\alpha_{r'_2}}$} node[pos=0.55,above,black] {$\cdots$} node[pos=0.8, black] (a3p) {$\encircled{\alpha_{r'_m}}$};

\draw[yellow!80!black] (h1)--(h1-|a1)--(a1);
\draw[yellow!80!black] (h2)--(h2-|a2)--(a2);
\draw[yellow!80!black] (h3)--(h3-|a3)--(a3);

\draw[yellow!80!black] (h1p)--(h1p-|a1p)--(a1p);
\draw[yellow!80!black] (h2p)--(h2p-|a2p)--(a2p);
\draw[yellow!80!black] (h3p)--(h3p-|a3p)--(a3p);

\draw[red, thick] (-5,-3.5) -- (-5,3);
\draw[red, thick] (5,-3.5) -- (5,3);
\draw[gray, thick] (-5,-3.5) -- (5,-3.5);
\draw[gray, thick] (-5,3) -- (5,3);
\end{tikzpicture}}}.
\]
More precisely, we can assign with this figure, the embedding $\Delta_{\CF_{(r_1,\ldots,r_i)}} \times \Delta_{\CF_{r'_1, \ldots , r'_j)}}$ of $\Delta^i \times \Delta^j$ in $\Conf_c$ where $c$ is the appropriate coloring (of which the whole set $\lbrace \alpha_{r_1}, \ldots , \alpha_{r_i}, \alpha_{r'_1} , \ldots , \alpha_{r'_j} \rbrace$ is a partition). This embedding is a cycle, and we use the handles for choosing a twised homology class. This precise class will be denoted:
\[
\CF_{(r_1, \ldots , r_i)} \cdot \CF_{(r'_1,\ldots , r'_j)} := \CF.
\]
We will make this notation meaningful in Sec.~\ref{sec_product}. 
\item When there is only one color involved, we can avoid pearls and use the simpler figure:
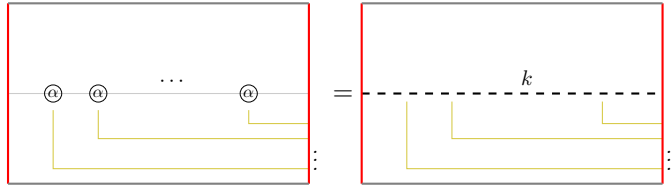
\[
\vcenter{\hbox{\begin{tikzpicture}[scale=0.4, every node/.style={scale=0.8},decoration={
    markings,
    mark=at position 0.5 with {\arrow{>}}}
    ]    
    
\coordinate (w0g) at (-5,0) {};
\coordinate (w0d) at (5,0) {};

\coordinate (h1) at (5,-2.5);
\coordinate (h2) at (5,-1.5);
\coordinate (h3) at (5,-1);

\node at (5.2,-2) {$\vdots$};

\draw[gray!40!white, thin] (w0g) -- (w0d) node[pos=0.15,black] (a1) {$\encircled{\alpha}$} node[pos=0.3,black] (a2) {$\encircled{\alpha}$} node[pos=0.55,above,black] {$\cdots$} node[pos=0.8, black] (a3) {$\encircled{\alpha}$};

\draw[yellow!80!black] (h1)--(h1-|a1)--(a1);
\draw[yellow!80!black] (h2)--(h2-|a2)--(a2);
\draw[yellow!80!black] (h3)--(h3-|a3)--(a3);

\draw[red, thick] (-5,-3) -- (-5,3);
\draw[red, thick] (5,-3) -- (5,3);
\draw[gray, thick] (-5,-3) -- (5,-3);
\draw[gray, thick] (-5,3) -- (5,3);
\end{tikzpicture}}} =
\vcenter{\hbox{\begin{tikzpicture}[scale=0.4, every node/.style={scale=0.8},decoration={
    markings,
    mark=at position 0.5 with {\arrow{>}}}
    ]    
    
\coordinate (w0g) at (-5,0) {};
\coordinate (w0d) at (5,0) {};

\coordinate (h1) at (5,-2.5);
\coordinate (h2) at (5,-1.5);
\coordinate (h3) at (5,-1);

\node at (5.2,-2) {$\vdots$};

\draw[dashed, thick] (w0g) -- (w0d) node[pos=0.15,black] (a1) {} node[pos=0.3,black] (a2) {} node[pos=0.55,above,black] {$k$} node[pos=0.8, black] (a3) {};

\draw[yellow!80!black] (h1)--(h1-|a1)--(a1);
\draw[yellow!80!black] (h2)--(h2-|a2)--(a2);
\draw[yellow!80!black] (h3)--(h3-|a3)--(a3);

\draw[red, thick] (-5,-3) -- (-5,3);
\draw[red, thick] (5,-3) -- (5,3);
\draw[gray, thick] (-5,-3) -- (5,-3);
\draw[gray, thick] (-5,3) -- (5,3);
\end{tikzpicture}}}
\]
where there is $k \in \BN$ pearls marked $\alpha$ on the left. This homology class is in $\CHbm_c$ where $c$ is he coloring with $k$ points colored by $\alpha$. 
\item When the number of colors is fixed, we can use real colors instead of marked pearls. For instance if $\alpha, \beta, \delta$ are three simple roots, we associate respectively colors red, blue and green to them, so that the following figures are related:
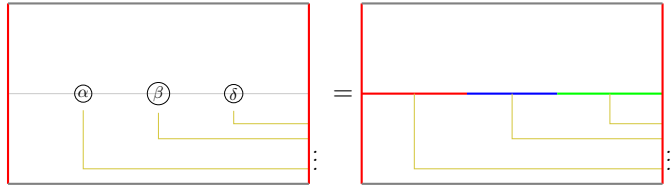
\[
\vcenter{\hbox{\begin{tikzpicture}[scale=0.4, every node/.style={scale=0.8},decoration={
    markings,
    mark=at position 0.5 with {\arrow{>}}}
    ]    
    
\coordinate (w0g) at (-5,0) {};
\coordinate (w0d) at (5,0) {};

\coordinate (h1) at (5,-2.5);
\coordinate (h2) at (5,-1.5);
\coordinate (h3) at (5,-1);

\node at (5.2,-2) {$\vdots$};

\draw[gray!40!white, thin] (w0g) -- (w0d) node[pos=0.25,black] (a1) {$\encircled{\alpha}$} node[pos=0.5,black] (a2) {$\encircled{\beta}$} node[pos=0.75, black] (a3) {$\encircled{\delta}$};

\draw[yellow!80!black] (h1)--(h1-|a1)--(a1);
\draw[yellow!80!black] (h2)--(h2-|a2)--(a2);
\draw[yellow!80!black] (h3)--(h3-|a3)--(a3);

\draw[red, thick] (-5,-3) -- (-5,3);
\draw[red, thick] (5,-3) -- (5,3);
\draw[gray, thick] (-5,-3) -- (5,-3);
\draw[gray, thick] (-5,3) -- (5,3);
\end{tikzpicture}}} =
\vcenter{\hbox{\begin{tikzpicture}[scale=0.4, every node/.style={scale=0.8},decoration={
    markings,
    mark=at position 0.5 with {\arrow{>}}}
    ]    
    
\coordinate (w0g) at (-5,0) {};
\coordinate (w0d) at (5,0) {};

\coordinate (h1) at (5,-2.5);
\coordinate (h2) at (5,-1.5);
\coordinate (h3) at (5,-1);

\node at (5.2,-2) {$\vdots$};

\draw[red,thick] (-5,0)--(-1.5,0) node[midway,above] (a1) {};
\draw[blue,thick] (-1.5,0)--(1.5,0) node[midway,above] (a2) {};
\draw[green,thick] (1.5,0)--(5,0) node[midway,above] (a3) {};

\draw[yellow!80!black] (h1)--(h1-|a1)--(a1);
\draw[yellow!80!black] (h2)--(h2-|a2)--(a2);
\draw[yellow!80!black] (h3)--(h3-|a3)--(a3);

\draw[red, thick] (-5,-3) -- (-5,3);
\draw[red, thick] (5,-3) -- (5,3);
\draw[gray, thick] (-5,-3) -- (5,-3);
\draw[gray, thick] (-5,3) -- (5,3);
\end{tikzpicture}}}
\]
Right hand notation is abusive as it seems to be the embedding of three intervals while it is still the simplex from the left (it can be thought as intervals with moving ends). Another example:
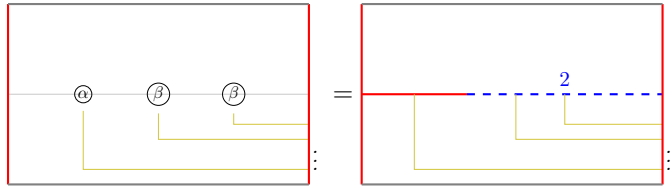
\[
\vcenter{\hbox{\begin{tikzpicture}[scale=0.4, every node/.style={scale=0.8},decoration={
    markings,
    mark=at position 0.5 with {\arrow{>}}}
    ]    
    
\coordinate (w0g) at (-5,0) {};
\coordinate (w0d) at (5,0) {};

\coordinate (h1) at (5,-2.5);
\coordinate (h2) at (5,-1.5);
\coordinate (h3) at (5,-1);

\node at (5.2,-2) {$\vdots$};

\draw[gray!40!white, thin] (w0g) -- (w0d) node[pos=0.25,black] (a1) {$\encircled{\alpha}$} node[pos=0.5,black] (a2) {$\encircled{\beta}$} node[pos=0.75, black] (a3) {$\encircled{\beta}$};

\draw[yellow!80!black] (h1)--(h1-|a1)--(a1);
\draw[yellow!80!black] (h2)--(h2-|a2)--(a2);
\draw[yellow!80!black] (h3)--(h3-|a3)--(a3);

\draw[red, thick] (-5,-3) -- (-5,3);
\draw[red, thick] (5,-3) -- (5,3);
\draw[gray, thick] (-5,-3) -- (5,-3);
\draw[gray, thick] (-5,3) -- (5,3);
\end{tikzpicture}}} =
\vcenter{\hbox{\begin{tikzpicture}[scale=0.4, every node/.style={scale=0.8},decoration={
    markings,
    mark=at position 0.5 with {\arrow{>}}}
    ]    
    
\coordinate (w0g) at (-5,0) {};
\coordinate (w0d) at (5,0) {};

\coordinate (h1) at (5,-2.5);
\coordinate (h2) at (5,-1.5);
\coordinate (h3) at (5,-1);

\node at (5.2,-2) {$\vdots$};

\draw[red,thick] (-5,0)--(-1.5,0) node[midway,above] (a1) {};
\draw[blue,dashed,thick] (-1.5,0)--(5,0) node[pos=0.25,above] (a2) {} node[midway,above] (a3) {$2$};

\draw[yellow!80!black] (h1)--(h1-|a1)--(a1);
\draw[yellow!80!black] (h2)--(h2-|a2)--(a2);
\draw[yellow!80!black] (h3)--(h3-|a3)--(a3);

\draw[red, thick] (-5,-3) -- (-5,3);
\draw[red, thick] (5,-3) -- (5,3);
\draw[gray, thick] (-5,-3) -- (5,-3);
\draw[gray, thick] (-5,3) -- (5,3);
\end{tikzpicture}}}
\]
\item the yellow color will be dedicated to handles. 
\end{enumerate}
\end{rmk}

We state the proposition describing the structure of homologies $\CHbm_c$ as $\Laurent$-modules. 

\begin{prop}[Structural proposition]\label{structure_result}
Let $c \in \Coloring_\Pi$, then the module $\CHbm_c$ is:
\begin{itemize}
\item a free $\Laurent$-module,
\item for which the set $\CB_{\CHbm_c} := \left\lbrace \CF_\br \text{ s.t. } \br \in \CP_c \right\rbrace$ is a basis.
\item It is the only non-vanishing module from the sequence $\Hlf_{\bullet} (\Conf_c, S; \Laurent_c)$ (the homology is concentrated in the middle dimension).
\end{itemize}
\end{prop}
\begin{proof}
The proof is almost word by word the same as that of \cite[Prop.~3.6]{Jules_Verma}, adapted to colored configuration but this is transparent in the proof. It relies on a trick from the first author, see \cite[Lemma~3.1]{Stephen_Hecke}. This trick only depends on the fact that the unit disk retracts on a tubular neighborhood of the unit interval on the real line, that this retraction can be extended to configurations. As the homology theory is not sensitive to retraction, and by excision theorem, then one can easily show the following isomorphism of modules:
\[
\CHbm_c \simeq \Hlf_{m_c} (\Conf_c^{\BR}, S^{\BR}; \Laurent_c^\BR)
\]
where $\Conf_c^{\BR}$ means the configurations of $\Conf_c$ lying on the real line, $S^{\BR}$ and $\Laurent_c^{\BR}$ the respective restrictions of $S$ and $\Laurent_c$ to $\Conf_c^{\BR}$. Then $\Conf_c^{\BR}$ is the disjoint union of all $\CF_\br$ for $\br \in \CP_c$. This proves the three lines of the proposition. (Recall that Borel--Moore homology of an open ball is one dimensional and concentrated in the dimension of the ball, see the Appendix \cite{Jules_Verma}.)
\end{proof}

Let $\alpha_{r_1},\ldots, \alpha_{r_k} \in \Pi$, we let:
\[
\CF^{[r_1,\ldots,r_k]} :=  \vcenter{\hbox{\begin{tikzpicture}[scale=0.5, every node/.style={scale=1},decoration={
    markings,
    mark=at position 0.5 with {\arrow{>}}}
    ]
\coordinate (w0g) at (-5,1.5) {};
\coordinate (w0d) at (5,1.5) {};

\coordinate (w1g) at (-5,-1.5) {};
\coordinate (w1d) at (5,-1.5) {};

\draw[gray!40!white, thin] (w0g) -- (w0d) node[pos=0.5, black] (a1) {$\encircled{\alpha_{r_1}}$};
\node at (0,0.5) {$\vdots$};
\draw[gray!40!white, thin] (w1g) -- (w1d) node[pos=0.5, black] (a1) {$\encircled{\alpha_{r_k}}$};

\draw[red, thick] (-5,-3) -- (-5,3);
\draw[red, thick] (5,-3) -- (5,3);
\draw[gray, thick] (-5,-3) -- (5,-3);
\draw[gray, thick] (-5,3) -- (5,3);
\end{tikzpicture}}}.
\]
Notice that following the algorithm defining homology class from diagrams, we have $$\CF^{[r_1,\ldots,r_k]} \in \CH_{\alpha_{r_1} + \cdots + \alpha_{r_k}},$$
since it corresponds to the embedding of a hypercube $I^k$ in $\Conf_{\alpha_{r_1} + \cdots + \alpha_{r_k}}$ with faces in $S$. In this case where there is only one pearl per necklace the diagram can express either a class in $\CHbm_c$ or one in $\CH_c$: we let $\CF^{[r_1,\ldots,r_k]}$ be the non Borel--Moore one. Thanks to the duality between Borel--Moore and standard homology, we deduce the structure of the standard homology. 

\begin{prop}[Structure of the standard homology]\label{T:structure_standard_homology}
Let $c \in \Coloring$, then the module $\CH_c$ is:
\begin{itemize}
\item a free $\Laurent$-module,
\item for which the set $\CB_{\CH_c} := \left\lbrace \CF^{[r_1,\ldots,r_k]} \text{ s.t. } r_1+\cdots + r_k \in \CP_c \right\rbrace$ is a basis.
\end{itemize}
\end{prop}
\begin{proof}
There is a perfect pairing between $\CH$ and $\CHbm$, the one described in Appendix~\ref{A:the_pairing}) to which we first apply a $\pi/2$-rotation to right term elements so that they belong to $\CHbm$. Now it is clear and also in the Appendix~\ref{A:the_pairing} that $\CB_{\CH_c} := \left\lbrace \CF^{[r_1,\ldots,r_k]} \text{ s.t. } r_1+\cdots + r_k \in \CP_c \right\rbrace$ forms a dual family to the basis elements of $\CHbm$. 
\end{proof}

\subsection{Computation rules}\label{S:computation rules}

We develop diagram rules translating homological calculus in terms of relations between pearl necklaces types of classes. They are all directly deduced from \cite[Sec.~4.1,~4.2]{Jules_Verma}, they were re-established in even more detailed \cite[Prop.~2.3, detailed~in~App.~B]{JulesMarco}. 

Before displaying the rules, we recall notation for combinatorial objects as elements of $\BZ[v^{\pm 1}]$ arising from quantum group calculus. 

\begin{Not}[Quantum variables, numbers, factorials and binomials]\label{N:quantum_numbers}
Let $v$ be a variable (the reader could think of it as an invertible element in a ring of Laurent polynomials):
\begin{enumerate}
\item {\em (asymmetric)} We define the following quantum numbers, factorials and binomials, for $k>l \in \BN$:
\[
(k)_v = 1 + v + \cdots + v^{k-1} = \frac{1-v^k}{1-v}, (k)_v! = (k)_v (k-1)_v \cdots (1)_v, \binom{k}{l}_v = \frac{(k)_v!}{(l)_v! (k-l)_v!}
\]
\item {\em (symmetric)} We define the following quantum numbers, factorials and binomials, for $k>l \in \BN$:
\[
\left[k \right]_v = \frac{v^k-v^{-k}}{v-v^{-1}}, \left[k\right]_v! = \left[k\right]_v \left[k-1\right]_v \cdots \left[1\right]_v, \qbin{k}{l}_v = \frac{\left[k\right]_v!}{\left[l\right]_v!\left[k-l\right]_v!}
\]
and also $\lbrace k \rbrace_v := v^k - v^{-k}$ s.t. $\left[ k \right]_v = \frac{\lbrace k \rbrace_v }{\lbrace 1 \rbrace_v}$ and $\lbrace k \rbrace_v ! := \lbrace k \rbrace_v \lbrace k-1 \rbrace_v \cdots \lbrace 1 \rbrace_v$. 
\item {\em (relations asymmetric vs symmetric)} We have (by easy computations):
\[
(k)_{v^{-2}} = v^{1-k} [k]_v , (k)_{v^{-2}} ! = v^{-k(k-1)/2} [k]_v !, \binom{k+l}{l}_{v^{-2}} = v^{-kl} \qbin{k+l}{l}_v
\]
\item {\em (variables)} Let $\alpha,\beta \in \Pi$ be two different roots, we fix the following not. for elements in $\Laurent$:
\[
q_{\alpha} := q^{d_{\alpha}} \text{ and } q_{\alpha,\beta} := q^{(\alpha, \beta)/2}
\]
\end{enumerate}

\end{Not}

The following rule works in either standard or Borel--Moore homology since it only deals with change of lift in the cover, and hence to local coefficients change.

\begin{rmk}[Handle rule, {\cite[Rem.~4.1]{Jules_Verma},\cite[Prop.~2.13.(B)]{JulesMarco}}]\label{rmk_handle_rule}
Let $M$ be a sub-manifold of $\Conf_c$ such that it represents a class in $\Hlf \left( \Conf_c , S; \BZ \right)$. To define a class in $\CHe_c$ one can use a {\em handle} as in previous section. It is an embedded path of $\Conf_c$ relating a vertically aligned configuration (usually on the right side of the disk) and a configuration $\lbrace z_1 , \ldots , z_{m_c} \rbrace \in M$. By the unique lift property, it defines a unique lift of $M$ in $\widehat{\Conf_c}$ hence eventually a class in $\CHe_c$ (Rem. \ref{twisted_homology_from_cover}). These handles are choosing lifts, and they satisfy rules:
\begin{itemize}
\item The lift is unchanged under ambient isotopy of the handle fixing its ends.
\item Let $\widehat{M}^\alpha$ and $\widehat{M}^\beta$ be two lifts of $M$ obtained from two handles $\alpha$, $\beta$. Then, in homology:
\begin{equation*}
\left[\widehat{M}^\alpha \right] =  \CW_c\circ\rho_c\left(\alpha \beta^{-1}\right) \left[ \widehat{M}^\beta \right]
\end{equation*}
with $\alpha \beta^{-1}$ being the concatenation of paths giving a loop and assimilated here with its homotopy class (hence, the first rule is recovered). 
\end{itemize}
\end{rmk}

\begin{rmk}[Permutation and embeddings, {\cite[Rem.~4.2]{Jules_Verma}, \cite[Prop.~2.13.(O)]{JulesMarco}}]\label{rmk_sign_change}
While working with diagrams such as those defined in previous section, the reader must be careful when changing a handle (yellow path) in diagrams. Namely, a change of yellow path could imply a change of path, hence an application of the rule from previous remark. But such a change could also permute the order on variables in the embedding of the simplex (in Eq. \eqref{embedding_necklace}), as in diagrams the order with which the yellow paths are reaching the necklaces is indicating this parameter. Changing them possibly requires multiplication by the sign of the induced permutation. We do one example in $\CHbm_{c_1^{\alpha}}$: 
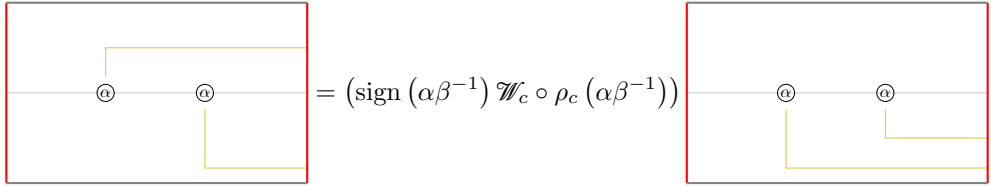
\[
\vcenter{\hbox{\begin{tikzpicture}[scale=0.4, every node/.style={scale=0.8},decoration={
    markings,
    mark=at position 0.5 with {\arrow{>}}}
    ]    
    
\coordinate (w0g) at (-5,0) {};
\coordinate (w0d) at (5,0) {};

\coordinate (h1) at (5,-2.5);
\coordinate (h2) at (5,-1.5);
\coordinate (h3) at (5,1.5);


\draw[gray!40!white, thin] (w0g) -- (w0d) node[pos=0.33,black] (a1) {$\encircled{\alpha}$} node[pos=0.66,black] (a2) {$\encircled{\alpha}$}; 

\draw[yellow!80!black] (h3)--(h3-|a1)--(a1);
\draw[yellow!80!black] (h1)--(h1-|a2)--(a2);

\draw[red, thick] (-5,-3) -- (-5,3);
\draw[red, thick] (5,-3) -- (5,3);
\draw[gray, thick] (-5,-3) -- (5,-3);
\draw[gray, thick] (-5,3) -- (5,3);
\end{tikzpicture}}} = 
\left( \sign\left( \alpha \beta^{-1} \right) \CW_c\circ\rho_c\left(\alpha \beta^{-1}\right) \right)
\vcenter{\hbox{\begin{tikzpicture}[scale=0.4, every node/.style={scale=0.8},decoration={
    markings,
    mark=at position 0.5 with {\arrow{>}}}
    ]    
    
\coordinate (w0g) at (-5,0) {};
\coordinate (w0d) at (5,0) {};

\coordinate (h1) at (5,-2.5);
\coordinate (h2) at (5,-1.5);
\coordinate (h3) at (5,1.5);


\draw[gray!40!white, thin] (w0g) -- (w0d) node[pos=0.33,black] (a1) {$\encircled{\alpha}$} node[pos=0.66,black] (a2) {$\encircled{\alpha}$}; 

\draw[yellow!80!black] (h1)--(h1-|a1)--(a1);
\draw[yellow!80!black] (h2)--(h2-|a2)--(a2);

\draw[red, thick] (-5,-3) -- (-5,3);
\draw[red, thick] (5,-3) -- (5,3);
\draw[gray, thick] (-5,-3) -- (5,-3);
\draw[gray, thick] (-5,3) -- (5,3);
\end{tikzpicture}}}
\]
where $\alpha$ and $\beta$ are respectively paths in $\Conf_{c_1^{\alpha}}$ corresponding to yellow paths from diagrams of right hand side resp. left hand side. Here $\CW_c\circ\rho_c\left(\alpha \beta^{-1}\right)  = -q^{d_{\alpha}}$ and $\sign\left( \alpha \beta^{-1} \right)=-1$. While $d_{\alpha}$ is odd, one has:
\[
\left( \sign\left( \alpha \beta^{-1} \right) \CW_c\circ\rho_c\left(\alpha \beta^{-1}\right) \right) = \iota_q \circ \CW_c\circ\rho_c\left(\alpha \beta^{-1}\right) 
\]
where $\iota_q$ is the involution $q \mapsto -q$. The latter is general whenever $d_{\alpha}$ is odd for all roots $\alpha$ involved:
\[
\sign \times \CW_c\circ\rho_c = \iota_q \circ \CW_c\circ\rho_c,
\]
where $\times$ is the product in $\Laurent_c$. For a fixed coloring $c$, the function $\sign$ is the product of signs of permutations:
\[
\sign : \Sk_{c(\alpha_1)} \times \cdots \times \Sk_{c(\alpha_l)} \to \lbrace \pm 1 \rbrace
\]
\end{rmk}

Now the following rule only holds in the Borel--Moore homology, since its right term involves diagrams that only defines Borel--Moore classes (i.e. with more than one pearl per necklace). It is the one making quantum factorials appearing, which partly explains why the quantum groups will live on the Borel--Moore side as we will see later. 

\begin{prop}[Fusion rules, {\cite[Prop.~4.8-11]{Jules_Verma},\cite[Prop.~2.13.(F)]{JulesMarco}}]\label{prop_fusion_rules}
Following are the \textit{fusion rules} for diagrams adapted to colored configuration spaces:
\begin{enumerate}
\item The following fusion rules hold in $\CHbm_{k \alpha}$ resp. $\CHbm_{(k+l) \alpha}$:
\begin{equation}
\vcenter{\hbox{\begin{tikzpicture}[scale=0.4, every node/.style={scale=0.8},decoration={
    markings,
    mark=at position 0.5 with {\arrow{>}}}
    ]    
    
\coordinate (w0g) at (-5,1) {};
\coordinate (w0d) at (5,1) {};

\coordinate (w1g) at (-5,-1) {};
\coordinate (w1d) at (5,-1) {};

\node[scale=3] at (-5.5,0) {$\lbrace$};
\node at (-6.5,0) {$k$};

\node at (0,0.25) {$\vdots$};

\draw[thick] (w0g)--(w0d);
\draw[thick] (w1g)--(w1d);

\draw[red, thick] (-5,-3) -- (-5,3);
\draw[red, thick] (5,-3) -- (5,3);
\draw[gray, thick] (-5,-3) -- (5,-3);
\draw[gray, thick] (-5,3) -- (5,3);

\node[yellow!80!black] at (w0d) {$\bullet$};
\node[yellow!80!black] at (w1d) {$\bullet$};

\end{tikzpicture}}} =
(k)_{q^{d_{\alpha}}}! \vcenter{\hbox{\begin{tikzpicture}[scale=0.4, every node/.style={scale=0.8},decoration={
    markings,
    mark=at position 0.5 with {\arrow{>}}}
    ]    
    
\coordinate (w0g) at (-5,0) {};
\coordinate (w0d) at (5,0) {};

\coordinate (h1) at (5,-2.5);
\coordinate (h2) at (5,-1.5);
\coordinate (h3) at (5,-1);

\node at (5.2,-2) {$\vdots$};

\draw[dashed, thick] (w0g) -- (w0d) node[pos=0.15,black] (a1) {} node[pos=0.3,black] (a2) {} node[pos=0.55,above,black] {$k$} node[pos=0.8, black] (a3) {};

\draw[yellow!80!black] (h1)--(h1-|a1)--(a1);
\draw[yellow!80!black] (h2)--(h2-|a2)--(a2);
\draw[yellow!80!black] (h3)--(h3-|a3)--(a3);

\draw[red, thick] (-5,-3) -- (-5,3);
\draw[red, thick] (5,-3) -- (5,3);
\draw[gray, thick] (-5,-3) -- (5,-3);
\draw[gray, thick] (-5,3) -- (5,3);
\end{tikzpicture}}},
\end{equation}

\begin{equation}
\vcenter{\hbox{\begin{tikzpicture}[scale=0.4, every node/.style={scale=0.8},decoration={
    markings,
    mark=at position 0.5 with {\arrow{>}}}
    ]    
    
\coordinate (w0g) at (-5,1) {};
\coordinate (w0d) at (5,1) {};

\coordinate (w1g) at (-5,-1) {};
\coordinate (w1d) at (5,-1) {};

\draw[thick,dashed] (w0g)--(w0d) node[midway,above] {$k$} node[pos=0.7,black] (a1) {} node[pos=0.8,black] (a1b) {};
\draw[thick,dashed] (w1g)--(w1d) node[midway,above] {$l$} node[pos=0.2,black] (a2) {} node[pos=0.4,black] (a2b) {};

\draw[red, thick] (-5,-3) -- (-5,3);
\draw[red, thick] (5,-3) -- (5,3);
\draw[gray, thick] (-5,-3) -- (5,-3);
\draw[gray, thick] (-5,3) -- (5,3);

\coordinate (h1) at (5,-1.5);
\coordinate (h1b) at (5,-2);
\coordinate (h2) at (5,0.5);
\coordinate (h2b) at (5,-0.5);
\coordinate (h3) at (5,-1);

\draw[yellow!80!black] (h2b)--(h2b-|a1)--(a1);
\draw[yellow!80!black] (h2)--(h2-|a1b)--(a1b);
\draw[yellow!80!black] (h1)--(h1-|a2b)--(a2b);
\draw[yellow!80!black] (h1b)--(h1b-|a2)--(a2);

\end{tikzpicture}}} =
\binom{k+l}{k}_{q^{d_{\alpha}}} \vcenter{\hbox{\begin{tikzpicture}[scale=0.4, every node/.style={scale=0.8},decoration={
    markings,
    mark=at position 0.5 with {\arrow{>}}}
    ]    
    
\coordinate (w0g) at (-5,0) {};
\coordinate (w0d) at (5,0) {};

\coordinate (h1) at (5,-2.5);
\coordinate (h2) at (5,-1.5);
\coordinate (h3) at (5,-1);

\node at (5.2,-2) {$\vdots$};

\draw[dashed, thick] (w0g) -- (w0d) node[pos=0.15,black] (a1) {} node[pos=0.3,black] (a2) {} node[pos=0.55,above,black] {$k+l$} node[pos=0.8, black] (a3) {};

\draw[yellow!80!black] (h1)--(h1-|a1)--(a1);
\draw[yellow!80!black] (h2)--(h2-|a2)--(a2);
\draw[yellow!80!black] (h3)--(h3-|a3)--(a3);

\draw[red, thick] (-5,-3) -- (-5,3);
\draw[red, thick] (5,-3) -- (5,3);
\draw[gray, thick] (-5,-3) -- (5,-3);
\draw[gray, thick] (-5,3) -- (5,3);
\end{tikzpicture}}},
\end{equation}
where the black color is used for $\alpha$. When no handle are drawn, there is a canonical vertically aligned configuration on the right side of the square that is materialized by fixed yellow dots, it is the constant path. 
\item The following holds in $\CH_c$, for the appropriate coloring $c=\alpha_{r_0}+ \cdots + \alpha_{r_m} \in \BN[\Pi]$ roots, we have:
\begin{equation}
\vcenter{\hbox{\begin{tikzpicture}[scale=0.5, every node/.style={scale=0.8},decoration={
    markings,
    mark=at position 0.5 with {\arrow{>}}}
    ]    
    
\coordinate (w0g) at (-5,0) {};
\coordinate (w0d) at (5,0) {};

\coordinate (w1g) at (-5,2) {};
\coordinate (w1d) at (5,2) {};

\coordinate (h1) at (5,-2.5);
\coordinate (h2) at (5,-1.5);
\coordinate (h3) at (5,-1);
\coordinate (h) at (5,1);

\node at (5.2,-2) {$\vdots$};

\draw[gray!40!white, thick, postaction={decorate}] (w0g) -- (w0d) node[pos=0.15,black] (a1) {$\encircled{\alpha_{r_1}}$} node[pos=0.3,black] (a2) {$\encircled{\alpha_{r_2}}$} node[pos=0.55,above,black] {$\cdots$} node[pos=0.8, black] (a3) {$\encircled{\alpha_{r_m}}$};

\draw[gray!40!white, thick, postaction={decorate}] (w1g) -- (w1d) node[pos=0.6,black] (a) {$\encircled{\alpha_{r_0}}$};

\draw[yellow!80!black] (h1)--(h1-|a1)--(a1);
\draw[yellow!80!black] (h2)--(h2-|a2)--(a2);
\draw[yellow!80!black] (h3)--(h3-|a3)--(a3);
\draw[yellow!80!black] (h)--(h-|a)--(a);

\draw[red, thick] (-5,-3) -- (-5,3);
\draw[red, thick] (5,-3) -- (5,3);
\draw[gray, thick] (-5,-3) -- (5,-3);
\draw[gray, thick] (-5,3) -- (5,3);
\end{tikzpicture}}} = \sum_{i=0}^m q^{\left( \alpha_{r_0}, \sum_{j=i+1}^m \alpha_{r_j} \right)} F_{(r_1,\ldots,r_i,r_0,r_{i+1},\ldots,r_m)} .
\end{equation}
\end{enumerate}
\end{prop}

\begin{prop}[Cutting rule, {\cite[Ex.~4.5]{Jules_Verma}, \cite[Prop.~2.13.(C)]{JulesMarco}}]\label{prop_breaking_rule}
In a diagram, a single colored arc can be broken against a puncture or part of the boundary, following bellow's rules.
\begin{enumerate}
\item Let $p \in D$ and $\Conf_c(p) := \lbrace \bz \in \Conf_c \text{ s.t. } p \in \bz \rbrace$. We define $\CHbm_c(p) := \Hlf_{m_c} \left( \Conf_c, \Conf_c(p) \cup S; \Laurent_c \right)$.

Then in $\CHbm_c(p)$ (thus with $c(\alpha) \ge k$), it holds:
\begin{equation*}
\left(\vcenter{\hbox{
\begin{tikzpicture}[every node/.style={scale=0.8},decoration={
    markings,
    mark=at position 0.5 with {\arrow{>}}}
    ]
\node (w0g) at (-1.5,0) {};
\node (w0d) at (1.5,0) {};
\node (wi) at (0,1) {};
\coordinate (a) at (-1,-1);
\draw[dashed] (w0g) to node[above=0.1pt,pos=0.5] (k0) {$k$} (w0d);

\node[gray] at (wi)[left=3pt] {$p$};
\foreach \n in {wi}
  \node at (\n)[gray,circle,fill,inner sep=2pt]{};

\draw[double,\hyellow,thick] (k0) -- (k0|-a);

\end{tikzpicture}
}}\right)
=
\sum_{l=0}^k
\left(\vcenter{\hbox{
\begin{tikzpicture}[every node/.style={scale=0.8},decoration={
    markings,
    mark=at position 0.5 with {\arrow{>}}}
    ]
\node (w0g) at (-1.5,0) {};
\node (w0d) at (1.5,0) {};
\node (wi) at (0,1) {};
\coordinate (a) at (-1,-1);
\draw[dashed] (w0g) to[bend right=50] node[above=0.1pt,pos=0.5] (k0) {$l$} (wi);
\draw[dashed] (wi) to[bend right=50] node[above=0.1pt,pos=0.5] (k1) {$k-l$} (w0d);

\node[gray] at (wi)[left=3pt] {$p$};
\foreach \n in {wi}
  \node at (\n)[gray,circle,fill,inner sep=2pt]{};

\draw[double,\hyellow,thick] (k0) -- (k0|-a);
\draw[double,\hyellow,thick] (k1) -- (k1|-a);

\end{tikzpicture}
}}\right),
\end{equation*}
with following conventions: the parentheses mean that we are considering part of a diagram defining a homology class in $\CHbm_c$ as those from Rem. \ref{rmk_Def_diagrams}, it is locally drawn as above. The yellow handles from the left and from the right are following same parallel paths outside parentheses, and the rest of the diagram remains the same outside parentheses. 
\item In $\CHbm_c$, we have the following equality:
\begin{equation*}
\left(\vcenter{\hbox{
\begin{tikzpicture}[every node/.style={scale=0.8},decoration={
    markings,
    mark=at position 0.5 with {\arrow{>}}}
    ]
\node (w0g) at (0,1.5) {};
\node (w0d) at (0,-1.5) {};
\node (w1u) at (1,1) {};
\node (w1d) at (1,-1) {};
\coordinate (a) at (-1,-1.5);
\draw[dashed] (w0g) to[bend left=90] node[above=0.1pt,pos=0.7] (k0) {} node[left=0.1pt,pos=0.5] {$k$} (w0d);

\draw[double,\hyellow,thick] (k0) -- (k0|-a);

\draw[thick,red] (1.25,1.5)--(1.25,-1.5);

\end{tikzpicture}
}}\right)
=
\sum_{l=0}^k
\left(\vcenter{\hbox{
\begin{tikzpicture}[every node/.style={scale=0.8},decoration={
    markings,
    mark=at position 0.5 with {\arrow{>}}}
    ]
\node (w0g) at (0,1.5) {};
\node (w0d) at (0,-1.5) {};
\node (w1u) at (1.25,0.25) {};
\node (w1d) at (1.25,-0.25) {};
\coordinate (a) at (-1,-1.5);
\draw[dashed] (w0g) to[bend left=20] coordinate[above=0.1pt,pos=0.6] (k0) coordinate[above=3pt,pos=0.8] (k0r) node[above,pos=0.5] {$l$} (w1u);
\draw[dashed] (w0d) to[bend right=20] node[above=0.1pt,pos=0.4] (k1) {} node[left,pos=0.5] {$k-l$} (w1d);

\draw[double,\hyellow,thick] (k0) to[bend left=60] (k0r)--(k0r|-a);
\draw[double,\hyellow,thick] (k1) -- (k1|-a);

\draw[thick,red] (1.25,1.5)--(1.25,-1.5);

\end{tikzpicture}
}}\right),
\end{equation*}
where we assume that diagrams are the same outside parentheses drawn in the equality. There might be other components of the diagram elsewhere outside the parentheses, with other root involved, so that $c$ is the appropriate coloring. In the dashed one that is represented in the relation, points are uniquely colored by the same root. Yellow handles are following same paths outside the parentheses in both terms, and the red line is part of $\partial^-D$. 
\end{enumerate}
Furthermore, when $k=1$ in the above statements, the equalities also hold in the standard homology. 
\end{prop}
\begin{proof}
We refer the reader to {\cite[Ex.~4.5]{Jules_Verma}, \cite[Prop.~2.13.(C)]{JulesMarco}} for detailed proof, we give another point of view. We simply mention that if $\overline{\Conf}_c$ denotes the manifold with boundary that is a compactification homotopy equivalent to $\Conf_c$ (it is due to D. Sinha \cite{sinha}, see the Appendix~\ref{A:the_pairing}), one has:
\[
\CHbm_c = \Hnot_{m_c} (\overline{\Conf}_c, \overline{\Conf}_c \setminus \Conf_c \cup S; \Laurent_c),
\]
where $\overline{\Conf}_c \setminus \Conf_c$ is part of the boundary, and $\Laurent_c$ is well defined due to homotopy equivalence. It is then easy to check that cutting rules are isotopies in $\overline{\Conf}_c$ relating part of the diagram to the relative part in the boundary, just like the case $k=1$ where the part of the boundary corresponding to the compactification is unnecessary. It serves for intuition in general, and for instance for Remark~\ref{R:cutting_in_chains}. 
\end{proof}

In the above two rules, one diagram is divided into the linear combination of diagrams where one single colored arc (necklace) has been divided into two differently indexed ones. In both of them one can check that as the diagram on the left was defining a homology class, those on the right are still defining ones (thanks to the fact that they reach $\partial^-D$, or the puncture used to define $\Conf_c(p)$). 

\begin{rmk}\label{R:cutting_in_chains}
In the previous proposition, the first rule also holds in $\CHbm_c$ (without the puncture $p$) but the terms on the right hand side are not homology class but well define embeddings of (products of) simplices, still the linear combination is a cycle. In $\CHbm_c(p)$ each term is a cycle and well defines a homology class. 
\end{rmk}

\section{A homological algebra structure}\label{BS:homological_algebra}

Let $\Uqgm$ be the strictly negative part of $\Uqg$. It is the subalgebra of $\Uqg$ generated by $\lbrace F_{\alpha}, \alpha \in \Pi \rbrace$. This section aims at endowing the spaces $\CHe := \bigoplus_{c \in \Coloring} \CHe_c$ with an algebra structure. It recovers homologically the algebra $\Uqgm$. This algebra $\Uqgm$ is a free algebra with one generator per simple root, modded out by the \textit{quantum Serre relations}. First we prove that the quantum Serre relations hold at Borel--Moore homology, in $\CHbm$. Then we notice that the standard homology $\CH$ yields a free algebra. Lastly we use Lusztig approach to $\Uqgm$ in \cite{LusztigBook} as a quotient of a free algebra to show that $\Uqgm$ is the image of the homology in the Borel--Moore one, hence the image of the canonical map $\CH \to \CHbm$ induced by inclusion of chain complexes. We discuss homological interpretation of Poincaré--Birkhoff--Witt bases of $\Uqgm$.  


\subsection{Algebra structure}\label{S:the_algebra_structure}

\subsubsection{Product}\label{sec_product}

Let $c_1$ and $c_2$ be two colorings by roots of $\Pi$. We define the coloring $c_1 + c_2$ as follows:
\[
\forall \alpha \in \Pi ,  (c_1 + c_2) (\alpha) = c_1(\alpha) + c_2(\alpha),
\]
it corresponds to the sum in $\BN[\Pi]$. 

We want to define a module morphism:
\[
\CHe_{c_1} \otimes \CHe_{c_2} \to \CHe_{c_1 + c_2}
\]
that satisfies axioms of a graded algebra product (graded by colorings).

Shortly speaking, the product is given by \emph{stacking} the unit disk used for $\Conf_{c_1}$ \emph{above} that used for $\Conf_{c_2}$, but it requires a more rigorous definition that we describe in what follows. 

Namely, stacking disks (drawn as rectangles) one above the other gives topologically the same unit disk with same $\partial^-D$, namely same red sides, forgetting the gluing interval:
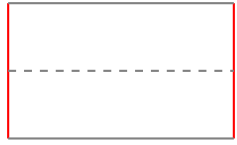
\[
\vcenter{\hbox{\begin{tikzpicture}[scale=0.3]
\draw[red, thick] (-5,-3) -- (-5,3);
\draw[red, thick] (5,-3) -- (5,3);
\draw[gray, thick] (-5,-3) -- (5,-3);
\draw[gray, thick] (-5,3) -- (5,3);
\draw[gray,dashed, thick] (-5,0) -- (5,0);
\end{tikzpicture}}}
\]
The dashed line is the attaching interval of both disks, that we forget after the gluing.

Hence, this operation naturally defines a map:
\[
\Conf_{c_1} \times \Conf_{c_2} \to \Conf_{c_1 + c_2}. 
\] 
This operation defines a product of simplices, which gives a product of regular simplicial homology classes of $\Conf_{c_1}$ and $\Conf_{c_2}$. 

Let $\Delta_{c_1}$ (and $\Delta_{c_2}$) be an $m_{c_1}$ (resp. $m_{c_2}$) dimensional simplex of $\Conf_{c_1}$ (resp. $\Conf_{c_2}$). Let $H_1$ (resp. $H_2$) be a path in $B_{c_1}$ (resp. $B_{c_2}$, notice that they are the same space as it only depends on $\Pi$) joining $(1, \ldots ,1)$ and $\CW_{c_1} (\Delta_{c_1})$ (resp. $\CW_{c_1} (\Delta_{c_1})$). The term by term product of paths $H_1$ and $H_2$ yields a path denoted $H_1 \cdot H_2$ in $B_{c_1 + c_2}$ (isomorphic to both $B_{c_1}$ and $B_{c_2}$) joining $(1,\ldots, 1)$ and $\CW_{c_1 + c_2} (\Delta_{c_1}\cdot \Delta_{c_2})$. Then the morphism of chain complexes:
\[
\bapp
\Clf_{m_{c_1}} (\Conf_{c_1}, S; \Laurent_{c_1}) \times \Clf_{m_{c_2}} (\Conf_{c_2}, S; \Laurent_{c_2}) & \to & \Clf_{m_{c_1+ c_2}} (\Conf_{c_1+ c_2}, S; \Laurent_{c_1+ c_2}) \\
\left( (\Delta_{c_1}, H_1) , (\Delta_{c_2} , H_2) \right) & \mapsto & \left( \Delta_{c_1} \cdot \Delta_{c_2} , H_1 \cdot H_2 \right)
\eapp
\]
naturally provides the product:
\[
\CHe_{c_1} \otimes \CHe_{c_2} \to \CHe_{c_1 + c_2}
\]
we were seeking. 

\begin{prop}[(homological) Algebra structure]\label{strictlynegative_algebrastructure}
The space $\CHe := \bigoplus_{c \in \Coloring_{\Pi}} \CHe_c$ endowed with the product $\cdot$ just described is a graded algebra (graded by $\Coloring_{\Pi}$, itself endowed with $+$ on colorings). The unit is defined to be the empty set, unique element of $\Conf_0$ (where $0$ is the zero coloring, namely zero points considered in the disk). 
\end{prop}
\begin{proof}
The stacking from above operation is associative. The associativity axiom follows. Stacking the empty set does not change any homology class, so that the unit axiom is also satisfied. 
\end{proof}

The set $\Coloring_{\Pi}$ endowed with $+$ is isomorphic to $\BN^l$ with generators $c_{\alpha_i}^1$ for $i=1,\ldots , l$. In other words, the algebra $\CHe$ is graded over $\BN^l$. 

\begin{example}
\begin{enumerate}
\item We can write the product using the bases described in Def. \ref{homology_generators}:
\begin{equation}
\bfct
\CHbm_{c_1} \otimes \CHbm_{c_2} & \to & \CHbm_{c_1 + c_2} \\
\CF_{\br} \otimes \CF_{\br'} & \mapsto & \CF_{\br} \cdot \CF_{\br'}:= \vcenter{\hbox{\begin{tikzpicture}[scale=0.35, every node/.style={scale=0.8},decoration={
    markings,
    mark=at position 0.5 with {\arrow{>}}}
    ]    
   
\coordinate (w0g) at (-5,0) {};
\coordinate (w0d) at (5,0) {};

\coordinate (h1) at (5,-2.5);
\coordinate (h2) at (5,-1.5);
\coordinate (h3) at (5,-1);

\node at (5.2,-2) {$\vdots$};

\draw[gray!40!white, thin] (w0g) -- (w0d) node[pos=0.15,black] (a1) {$\encircled{\alpha_{r_1}}$} node[pos=0.3,black] (a2) {$\encircled{\alpha_{r_2}}$} node[pos=0.55,above,black] {$\cdots$} node[pos=0.8, black] (a3) {$\encircled{\alpha_{r_m}}$};

\draw[yellow!80!black] (h1)--(h1-|a1)--(a1);
\draw[yellow!80!black] (h2)--(h2-|a2)--(a2);
\draw[yellow!80!black] (h3)--(h3-|a3)--(a3);

\draw[red, thick] (-5,-3) -- (-5,3);
\draw[red, thick] (5,-3) -- (5,3);
\draw[gray, dashed] (-5,-3) -- (5,-3);
\draw[gray, thick] (-5,3) -- (5,3);


\coordinate (w0gD) at (-5,-6) {};
\coordinate (w0dD) at (5,-6) {};

\coordinate (h1D) at (5,-8.5);
\coordinate (h2D) at (5,-7.5);
\coordinate (h3D) at (5,-7);

\node at (5.2,-8) {$\vdots$};

\draw[gray!40!white, thin] (w0gD) -- (w0dD) node[pos=0.15,black] (a1D) {$\encircled{\alpha_{r'_1}}$} node[pos=0.3,black] (a2D) {$\encircled{\alpha_{r'_2}}$} node[pos=0.55,above,black] {$\cdots$} node[pos=0.8, black] (a3D) {$\encircled{\alpha_{r'_m}}$};

\draw[yellow!80!black] (h1D)--(h1D-|a1D)--(a1D);
\draw[yellow!80!black] (h2D)--(h2D-|a2D)--(a2D);
\draw[yellow!80!black] (h3D)--(h3D-|a3D)--(a3D);

\draw[red, thick] (-5,-3) -- (-5,-9);
\draw[red, thick] (5,-3) -- (5,-9);
\draw[gray, thick] (-5,-9) -- (5,-9);

\end{tikzpicture}}}
\efct
\end{equation}
for any $(\br,\br') \in \CP_{c_1} \times \CP_{c_2}$. The diagram describing the element $\CF_{\br} \cdot \CF_{\br'} \in \CHbm_{c_1 + c_2}$ is the stacking of two necklaces to which is naturally assigned a homology class: the union of the two necklaces (gray lines marked by indexed pearls) naturally defines a product of simplices, while the union of handles naturally defines a path from a vertically aligned configuration to a point in the product of simplices just defined. 

This expression of the product could be considered as a definition (as pearl Necklaces were shown to be a homological basis in Prop. \ref{structure_result}). 
\item We do a special case of the one just introduced. Let $\CF_{\alpha}$ be a class as in Def. \ref{F_classes_1dim}. The product $\CF_{\alpha} \cdot \CF_{\alpha}$ is given by the following diagram:
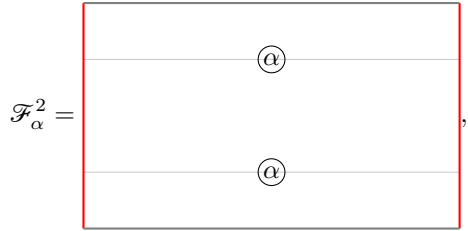
\[
\CF_{\alpha}^2 = \vcenter{\hbox{\begin{tikzpicture}[scale=0.5, every node/.style={scale=1},decoration={
    markings,
    mark=at position 0.5 with {\arrow{>}}}
    ]
\coordinate (w0g) at (-5,1.5) {};
\coordinate (w0d) at (5,1.5) {};

\coordinate (w1g) at (-5,-1.5) {};
\coordinate (w1d) at (5,-1.5) {};

\draw[gray!40!white, thin] (w0g) -- (w0d) node[pos=0.5, black] (a1) {$\encircled{\alpha}$};
\draw[gray!40!white, thin] (w1g) -- (w1d) node[pos=0.5, black] (a1) {$\encircled{\alpha}$};

\draw[red, thick] (-5,-3) -- (-5,3);
\draw[red, thick] (5,-3) -- (5,3);
\draw[gray, thick] (-5,-3) -- (5,-3);
\draw[gray, thick] (-5,3) -- (5,3);
\end{tikzpicture}}},
\]
i.e. the stacking from above of two diagrams associated with $\CF_{\alpha}$. This new diagram naturally comes with a natural way of reading it as a class in $\CHbm_{c_{\alpha}^2}$ (where $c_{\alpha}^2 := c_{\alpha}^1 +c_{\alpha}^1$ means two points colored by $\alpha$ and no other):
\begin{itemize}
\item The union of the two light gray necklaces defines an embedding of a product of interval in $\Conf_{c_{\alpha}^2}$. This embedding of a square is used as the chain. 
\item There is a unique configuration of two points in the embedding just described which is vertically aligned on $\partial^-D$. This is sent to $(1,1)$ in $B_{c_{\alpha}^2}$ by the writhe application. The constant path (constant in $(1,1) \in B_{c_{\alpha}^2}$) is then used as a handle. (In fact any vertically aligned configuration gives the same class, we choose it to lie on $\partial^-D$). 
\end{itemize}
A direct consequence of the proof of \cite[Lemma~4.8]{Jules_Verma}, one has:
\[
\CF_{\alpha_i}^2 = (1- (-q)^{d_i} ) \CF_{(\alpha, \alpha)}
\]
where $\CF_{(\alpha, \alpha)}$ is defined as in Def. \ref{homology_generators}.
It is easily generalized to the following:
\[
\CF_{\alpha_i} \CF_{\alpha_j} = \CF_{(\alpha_j, \alpha_i)} - q^{(\alpha_i, \alpha_j)} \CF_{(\alpha_i, \alpha_j)}
\]
In terms of diagrams, the above equality can be seen as the following one:
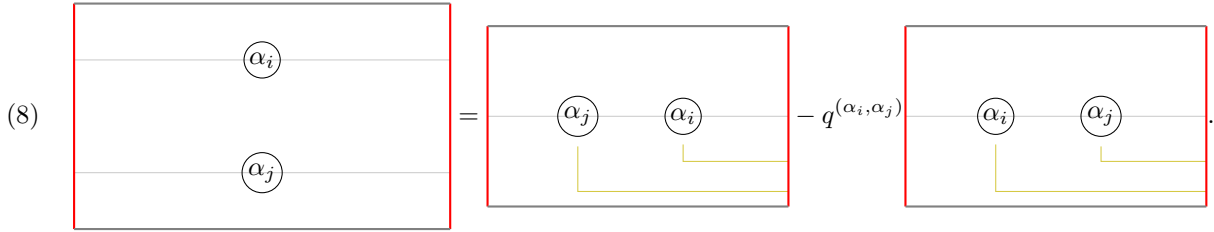
\begin{equation}
\vcenter{\hbox{\begin{tikzpicture}[scale=0.5, every node/.style={scale=1},decoration={
    markings,
    mark=at position 0.5 with {\arrow{>}}}
    ]
\coordinate (w0g) at (-5,1.5) {};
\coordinate (w0d) at (5,1.5) {};

\coordinate (w1g) at (-5,-1.5) {};
\coordinate (w1d) at (5,-1.5) {};

\draw[gray!40!white, thin] (w0g) -- (w0d) node[pos=0.5, black] (a1) {$\encircled{\alpha_i}$};
\draw[gray!40!white, thin] (w1g) -- (w1d) node[pos=0.5, black] (a1) {$\encircled{\alpha_j}$};

\draw[red, thick] (-5,-3) -- (-5,3);
\draw[red, thick] (5,-3) -- (5,3);
\draw[gray, thick] (-5,-3) -- (5,-3);
\draw[gray, thick] (-5,3) -- (5,3);
\end{tikzpicture}}}
 =  \vcenter{\hbox{\begin{tikzpicture}[scale=0.4, every node/.style={scale=1},decoration={
    markings,
    mark=at position 0.5 with {\arrow{>}}}
    ]    
    
\coordinate (w0g) at (-5,0) {};
\coordinate (w0d) at (5,0) {};

\coordinate (h1) at (5,-2.5);
\coordinate (h2) at (5,-1.5);
\coordinate (h3) at (5,-1);

\draw[gray!40!white, thin] (w0g) -- (w0d) node[pos=0.3,black] (a1) {$\encircled{\alpha_j}$} node[pos=0.65,black] (a2) {$\encircled{\alpha_i}$}; 

\draw[yellow!80!black] (h1)--(h1-|a1)--(a1);
\draw[yellow!80!black] (h2)--(h2-|a2)--(a2);

\draw[red, thick] (-5,-3) -- (-5,3);
\draw[red, thick] (5,-3) -- (5,3);
\draw[gray, thick] (-5,-3) -- (5,-3);
\draw[gray, thick] (-5,3) -- (5,3);
\end{tikzpicture}}} - q^{(\alpha_i, \alpha_j)}\vcenter{\hbox{\begin{tikzpicture}[scale=0.4, every node/.style={scale=1},decoration={
    markings,
    mark=at position 0.5 with {\arrow{>}}}
    ]    
    
\coordinate (w0g) at (-5,0) {};
\coordinate (w0d) at (5,0) {};

\coordinate (h1) at (5,-2.5);
\coordinate (h2) at (5,-1.5);
\coordinate (h3) at (5,-1);

\draw[gray!40!white, thin] (w0g) -- (w0d) node[pos=0.3,black] (a1) {$\encircled{\alpha_i}$} node[pos=0.65,black] (a2) {$\encircled{\alpha_j}$}; 

\draw[yellow!80!black] (h1)--(h1-|a1)--(a1);
\draw[yellow!80!black] (h2)--(h2-|a2)--(a2);

\draw[red, thick] (-5,-3) -- (-5,3);
\draw[red, thick] (5,-3) -- (5,3);
\draw[gray, thick] (-5,-3) -- (5,-3);
\draw[gray, thick] (-5,3) -- (5,3);
\end{tikzpicture}}}.
\end{equation}
\end{enumerate}
\end{example}

\begin{rmk}
Product is read from left to right while stacking is read from top to bottom. 
\end{rmk}

\subsubsection{Divided powers generators in Borel--Moore homologies}\label{S:divided_powers}

\begin{defn}[Divided powers]\label{def_div_powers_Fs}
Let $\alpha$ be a root, $k$ be a positive integer, and $c_{\alpha}^k:= k \alpha$ the coloring by $\alpha$ of $k$ points. We define the {\em $k$-th divided power of $\CF_{\alpha}^{(1)}$} as follows:
\[
\CF_{\alpha}^{(k)} :=  q^{-d_{\alpha} \frac{k(k-1)}{4}} 
\vcenter{\hbox{\begin{tikzpicture}[scale=0.4, every node/.style={scale=0.8},decoration={
    markings,
    mark=at position 0.5 with {\arrow{>}}}
    ]    
    
\coordinate (w0g) at (-5,0) {};
\coordinate (w0d) at (5,0) {};

\coordinate (h1) at (5,-2.5);
\coordinate (h2) at (5,-1.5);
\coordinate (h3) at (5,-1);

\node at (5.2,-2) {$\vdots$};

\draw[dashed, thick] (w0g) -- (w0d) node[pos=0.15,black] (a1) {} node[pos=0.3,black] (a2) {} node[pos=0.55,above,black] {$k$} coordinate[pos=0.8] (a3) {};


\coordinate (a33) at (5,2);
\draw[yellow!80!black,double,thick] (a33)--(a33-|a3)--(a3);

\draw[red, thick] (-5,-3) -- (-5,3);
\draw[red, thick] (5,-3) -- (5,3);
\draw[gray, thick] (-5,-3) -- (5,-3);
\draw[gray, thick] (-5,3) -- (5,3);
\end{tikzpicture}}} \in \CHbm_{k \alpha}
\]
(where $F_{\alpha}^{(1)}$ is also defined). 
\end{defn}

\begin{prop}[Divided power property]\label{prop_DividedPowers}
In $\CHbm_{c_{\alpha}^k}$ the divided powers satisfy:
\[
\left(\CF_{\alpha}^{(1)} \right)^k = \left[ k \right]_{q_{\alpha}^{\frac{1}{2}}} ! \CF_{\alpha}^{(k)}
\]
\end{prop}
\begin{proof}
It is a straightforward consequence of the fusion rule Prop.~\ref{prop_fusion_rules}~(6) and of the relation between asymmetric and symmetric quantum factorials. 
\end{proof}

\subsection{A free algebra on standard homologies}\label{S:the_free_algebra}

Notice that the diagrams of $\CF^{(1)}_\alpha$ has one single pearl, and hence can alternatively define a homology class in $\CH_\alpha$. We denote by $\CF^{[1]}_\alpha$ the class in $\CH_\alpha$ defined by the same diagram as $\CF^{(1)}_\alpha$.

\begin{theorem}\label{T:free_algebra}
The algebra $\CH:= \bigoplus_{c\in\Coloring_{\Pi}} \CH_c$ is the free algebra generated by the set $\{\CF^{[1]}_{\alpha}, \alpha \in \Pi \}$. 
\end{theorem}
\begin{proof}
Theorem~\ref{T:structure_standard_homology} shows that as a module, $\CH_c$ is generated by $$\{\CF^{[1]}_{\alpha_{r_1}}\cdots\CF^{[1]}_{\alpha_{r_k}}, \sum \alpha_{r_i} = c \}$$. It proves the result. 
\end{proof}

\subsection{Quantum-Serre relations in Borel--Moore homologies}\label{S:quantum_Serre}

\subsubsection{Squash vs partition}\label{S:squash_vs_partition}

The aim of this section is to express a particular homology class in two ways: \textit{squashing it} or \textit{partitioning it} will provide an important relation. First, we re-draw the unit disk as follows:

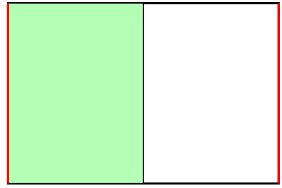
\[
\vcenter{\hbox{
\begin{tikzpicture}[scale=0.3]
\draw[thick] (0,0) rectangle ++(12,8);
\draw[fill = green!30!white] (0,0) rectangle ++(6,8);
\draw[red,thick] (0,0)--(0,8);
\draw[red,thick] (12,0)--(12,8);
\end{tikzpicture}}}
\]

We denote by $S' \subset \Conf_c$ the configurations having at least one coordinate in the part of the disk filled in green union the red intervals. This thickening of $\partial^-D$ is denoted $\underline{\partial^-D}$.  
\begin{prop}
For any coloring $c$ we have:
\[
\CHbm_c \simeq \Hlf_{m_c} \left( \Conf_c, S'; \Laurent_c\right),
\]
namely we can replace $S$ by $S'$ as a relative part for the homology, the natural inclusion of spaces becomes an isomorphism in homology.
\end{prop}
\begin{proof}
In the Appendix~\ref{A:the_pairing} we mention a definition of the Borel--Moore homology as that of a compactification relative to what is added to compactify, requiring a homotopy equivalent compactification so for it to work in the twisted homology. 
In the present proof, we use the definition of the Borel--Moore homology detailed for instance in \cite[A.2]{JulesMarco} (and in this twisted version), it is namely:
\[
\Hlf_\bullet(X,Y;\Laurent) := \varprojlim_{K \in \boldsymbol{K}} \Hnot_\bullet(X,Y \cup (X \smallsetminus K),\Laurent)
\]
where $\Laurent$ is a ring, and $\boldsymbol{K}$ is the set of all compacts with the inclusion as a partial order. Our (squared) disk is endowed with the obvious euclidean distance. Then, for $\varepsilon>0$, let:
\[
A_\varepsilon := \{ \boldsymbol{z} \in \Conf_c(D),  |z_i-z_j|\ge \varepsilon,\forall z_i, z_j \in \boldsymbol{z} \}
\]
which defines a projective family of compact sets. Hence:
\[
\CHbm_c := \lim_{\varepsilon \to 0} \Hnot_c (\Conf_c, \Conf_c \setminus A_\varepsilon \cup S; \Laurent_c). 
\]
It exists a sufficiently small collar $\partial^-_\varepsilon D$ of $\partial^-D$ such that:
\[
( z_i, z_j \in \partial^-_\varepsilon D \text{ and } \Im(z_i) = \Im(z_j) ) \implies |z_i - z_j| < \varepsilon. 
\]
Let $S_\varepsilon$ be the space of configurations with one coordinate in $\partial^-_\varepsilon D$. Now without loss of generality, we can assume:
\[
\Hlf_{m_c} \left( \Conf_c, S'; \Laurent_c\right) \simeq \Hlf_{m_c} \left( \Conf_c, S'_\varepsilon; \Laurent_c\right).
\]
Notice:
\[
(\{ z_i , z_j \} \in \boldsymbol{z} \in S'_\varepsilon \text{ and } \Im(z_i) = \Im(z_j)) \implies \boldsymbol{z} \in \Conf_c \setminus A_\varepsilon. 
\]
Now let:
\[
S^{\neq}_{\varepsilon} := \{ \boldsymbol{z} \in S'_\varepsilon, \Im(z_i) \neq \Im(z_j) , \forall \{z_i, z_j\} \in \boldsymbol{z} \}. 
\]
We have:
\[
\Conf_c \setminus A_\varepsilon \cup S'_\varepsilon = \Conf_c \setminus A_\varepsilon \cup S^{\neq}_\varepsilon,
\]
so that:
\[
\Hnot_c (\Conf_c, \Conf_c \setminus A_\varepsilon \cup S'_\varepsilon; \Laurent_c) = \Hnot_c (\Conf_c, \Conf_c \setminus A_\varepsilon \cup S^{\neq}_\varepsilon; \Laurent_c).
\]
Now the obvious horizontal line retraction of $\partial^-_\varepsilon D$ on $\partial^-D$ can be extended to $S^{\neq}_\varepsilon$, and sends it to $S$. It provides an isomorphism:
\[
\Hnot_c (\Conf_c, \Conf_c \setminus A_\varepsilon \cup S; \Laurent_c) \simeq \Hnot_c (\Conf_c, \Conf_c \setminus A_\varepsilon \cup S^{\neq}_\varepsilon; \Laurent_c).
\] 
It proves the result. Notice that all retractions used in this proof are contractions, so that they preserve $\Conf_c \setminus A_\varepsilon$, and are well defined retractions of pairs. 
\end{proof}

We define the following homology class for a pair of simple roots $\alpha, \beta \in \Pi$ (represented by blue resp. red coloring), and an integer $k$.  

\begin{equation}\label{qSerre_k}
\qSerre_{\alpha,\beta}^k := \vcenter{\hbox{
\begin{tikzpicture}[scale=0.4,decoration={
    markings,
    mark=at position 0.3 with {\arrow{>}}}]
\draw[thick] (0,0) rectangle ++(12,8);
\draw[fill = green!30!white] (0,0) rectangle ++(6,8);
\draw[red,thick] (0,0)--(0,8);
\draw[red,thick] (12,0)--(12,8);
\draw[red,postaction={decorate}]  (4,4) -- (12,4);
\node[red,above right] at (4,4) {$P$};
\node[red] at (4,4) {$\times$};
\draw[blue,postaction={decorate}] (12,3) -- (6,3) .. controls (2,3) and (2,5) .. (6,5) -- (12,5);
\draw[blue,postaction={decorate}] (12,2) -- (6,2) .. controls (0,2) and (0,6) .. (6,6) -- (12,6);
\node[right,blue] at (12,3) {$p_1$};
\node[right,blue] at (12,2) {$p_k$};
\node[right,blue,thick] at (12,5.5) {$\Big\rbrace k$};
\node[left,blue] at (12,5.6) {$\vdots$};
\node[right,red] at (12,4) {$q$};
\end{tikzpicture}}} \in \CHbm_{k \alpha + \beta}. 
\end{equation}

There are $k$ blue arcs (i.e. colored by $\alpha$) nested into each other and around a red arc (colored by $\beta$). Also some points denoted $p_1,\ldots ,p_k ,q$ are vertically aligned on $\partial^-D$. The point $P$ will be considered as a \textit{puncture} rather than a point. The reason is that, a simple application of the excision theorem shows:
\[
\Hlf_{m_c} \left( \Conf_c (P), S'(P); \Laurent_c\right) \to \Hlf_{m_c} \left( \Conf_c, S'; \Laurent_c\right)
\]
given by inclusion is an isomorphism, where $S'(P)$ is the suitable subspace of $\Conf_c(P)$ for which no configuration contains $P$. We denote by ${\CHbm_c}' :=  \Hlf_{m_c} \left( \Conf_c (P), S'(P); \Laurent_c\right)$, and when $c$ is the coloring with $k$ points colored by $\alpha$ and $1$ by $\beta$ ($c=k\alpha + \beta$) we explain briefly how $qSerre_{\alpha, \beta}^k$ naturally defines a homology class in ${\CHbm_c}'$. There is an embedding:
\[
I^{\times k} \times I \to \Conf_c
\]
obtained by sending the first $k$ copies of the interval $I$ to the blue arcs (first copy to the outer most one), and sending the last copy of $I$ to the red interval, considering the orientation prescribed by the ones marked on the diagram. The latter is a cycle relative to $S'$ (it is a hypercube with faces in $S'(P)$). The configuration $(\lbrace p_1 , \ldots, p_k \rbrace , q) \in \Conf_c$ is vertically aligned, we use it as a constant handle. Thus a class in ${\CHbm_c}'$ is naturally assigned with the diagram (\ref{qSerre_k}). We will find two expressions for $\qSerre_{\alpha, \beta}^k$ in terms of other standard elements obtained while one squashes it or partitions it, meaning two different computation tricks. We start with notation:
\[
q_{\alpha} := q^{(\alpha,\alpha)/2} , q_{\alpha, \beta} := q^{(\alpha , \beta)/2}
\]
and the following lemma.
\begin{lemma}\label{lemma_qSerre_divpower_setup}
In ${\CHbm_c}'$ we have the following equality:
\begin{equation}\label{qSerre_dividedpower_setup}
\qSerre_{\alpha,\beta}^k := (k)_{q_{\alpha}}! \vcenter{\hbox{
\begin{tikzpicture}[scale=0.45,decoration={
    markings,
    mark=at position 0.3 with {\arrow{>}}}]
\draw[thick] (0,0) rectangle ++(12,8);
\draw[fill = green!30!white] (0,0) rectangle ++(6,8);
\draw[red,thick] (0,0)--(0,8);
\draw[red,thick] (12,0)--(12,8);
\draw[red,postaction={decorate}]  (4,4) -- (12,4);
\node[red,above right] at (4,4) {$P$};
\node[red] at (4,4) {$\times$};
\coordinate (kk) at (12,3);
\draw[blue,dashed,postaction={decorate}] (12,2) to coordinate[midway] (k) (6,2) .. controls (0,2) and (0,6) .. (6,6) to node[midway,above] {$k$} (12,6);
\draw[yellow!80!black,double,thick] (k)--(k|-kk)--(kk);
\node[right,red] at (12,4) {$q$};
\end{tikzpicture}}} .
\end{equation}
The RHS has an indexed $k$ dashed arc, representing an arc on which $k$ blue points are embedded, and a yellow (multi)-handle is chosen: a path joining a vertical $k$-configuration on the boundary to a $k$-configuration on the dashed arc. We denote the right hand diagram $\qSerre_{\alpha,\beta}^{(k)}$ so that:
\[
\qSerre_{\alpha,\beta}^k  = (k)_{q_{\alpha}}! \qSerre_{\alpha,\beta}^{(k)}
\] 
\end{lemma}
\begin{proof}
It is a direct application of a fusion rule Prop.~\ref{prop_fusion_rules}.   
\end{proof}

We now state the two propositions yielding two different expressions of $\qSerre_{\alpha,\beta}^k$, the first one is obtained by \textit{squashing} it.

\begin{prop}[Squashing operation]\label{prop_squashing}
In ${\CHbm_c}'$ we have the following equality:
\[
\qSerre_{\alpha,\beta}^{(k)} = q_{\alpha}^{\frac{-k(k-1)}{2}}\prod_{l=0}^{k-1} \left( q_{\alpha}^{-l} q_{\alpha,\beta}^{-1} - q_{\alpha,\beta} \right)  \vcenter{\hbox{
\begin{tikzpicture}[scale=0.35,decoration={
    markings,
    mark=at position 0.5 with {\arrow{>}}}]
\draw[thick] (0,0) rectangle ++(12,8);
\draw[fill = green!30!white] (0,0) rectangle ++(6,8);
\draw[red,thick] (0,0)--(0,8);
\draw[red,thick] (12,0)--(12,8);
\draw[red,postaction={decorate}]  (4,4) to coordinate[pos=0.6] (kr) (8,4);
\draw[blue,dashed,postaction={decorate}]  (8,4) to node[midway,above] (k) {$k$} (12,4);
\node[red] at (4,4) {$\times$};
\coordinate (kk) at (12,2);
\coordinate (kkr) at (12,1);
\draw[yellow!80!black,double,thick] (k)--(k|-kk)--(kk);
\draw[yellow!80!black] (kr)--(kr|-kkr)--(kkr);
\end{tikzpicture}}}
\]
\end{prop}
\begin{proof}
First we express the following homology class (defined for any $k \in \BN$), the proof will be a direct recursion involving it. 
\begin{align*}
\vcenter{\hbox{ 
\begin{tikzpicture}[scale=0.22,every node/.style={scale=0.5},decoration={
    markings,
    mark=at position 0.3 with {\arrow{>}}}]
\draw[thick] (0,0) rectangle ++(12,8);
\draw[fill = green!30!white] (0,0) rectangle ++(6,8);
\draw[red,thick] (0,0)--(0,8);
\draw[red,thick] (12,0)--(12,8);
\draw[red,postaction={decorate}] (4,4) to coordinate[pos=0.6] (kr) (8,4);
\draw[blue,dashed,postaction={decorate}] (8,4) to node[midway,above] (k) {$k$} (12,4);
\draw[blue,postaction={decorate}] (12,2) -- (6,2) .. controls (0,2) and (0,6) .. (6,6) -- (12,6);
\node[yellow!80!black] at (12,2) {$\bullet$};
\coordinate (kk) at (12,3);
\coordinate (kkr) at (12,2.5);
\node[red] at (4,4) {$\times$};
\draw[yellow!80!black,double,thick] (k)--(k|-kk)--(kk);
\draw[yellow!80!black] (kr)--(kr|-kkr)--(kkr);
\end{tikzpicture}}}
& = \vcenter{\hbox{ 
\begin{tikzpicture}[scale=0.22,every node/.style={scale=0.5},decoration={
    markings,
    mark=at position 0.4 with {\arrow{>}}}]
\draw[thick] (0,0) rectangle ++(12,8);
\draw[fill = green!30!white] (0,0) rectangle ++(6,8);
\draw[red,thick] (0,0)--(0,8);
\draw[red,thick] (12,0)--(12,8);
\draw[red,postaction={decorate}] (3,4) to coordinate[pos=0.7] (kr) (8,4);
\draw[blue,dashed,postaction={decorate}] (8,4) to node[midway,above] (k) {$k$} (12,4);
\draw[blue,postaction={decorate}] (3,4) .. controls (3,6) and (4,6) .. (6,6) -- (12,6);
\draw[yellow!80!black] (12,1.5)--(2.5,1.5)--(2.5,4.5)--(3,4.5);
\coordinate (kk) at (12,3);
\coordinate (kkr) at (12,2.5);
\node[red] at (3,4) {$\times$};
\draw[yellow!80!black,double,thick] (k)--(k|-kk)--(kk);
\draw[yellow!80!black] (kr)--(kr|-kkr)--(kkr);
\end{tikzpicture}}} +
\vcenter{\hbox{ 
\begin{tikzpicture}[scale=0.22,every node/.style={scale=0.5},decoration={
    markings,
    mark=at position 0.25 with {\arrow{>}}}]
\draw[thick] (0,0) rectangle ++(12,8);
\draw[fill = green!30!white] (0,0) rectangle ++(6,8);
\draw[red,thick] (0,0)--(0,8);
\draw[red,thick] (12,0)--(12,8);
\draw[red,postaction={decorate}] (3,4) to coordinate[pos=0.7] (kr) (8,4);
\draw[blue,dashed,postaction={decorate}] (8,4) to node[midway,above] (k) {$k$} (12,4);
\draw[blue,postaction={decorate}] (12,2) -- (6,2) .. controls (4,2) and (3,2) .. (3,4);
\coordinate (kk) at (12,3);
\coordinate (kkr) at (12,2.5);
\node[red] at (3,4) {$\times$};
\draw[yellow!80!black,double,thick] (k)--(k|-kk)--(kk);
\draw[yellow!80!black] (kr)--(kr|-kkr)--(kkr);
\node[yellow!80!black] at (12,2) {$\bullet$};
\end{tikzpicture}}}\\ 
 & = \vcenter{\hbox{
\begin{tikzpicture}[scale=0.22,every node/.style={scale=0.5},decoration={
    markings,
    mark=at position 0.5 with {\arrow{>}}}]
\draw[thick] (0,0) rectangle ++(12,8);
\draw[fill = green!30!white] (0,0) rectangle ++(6,8);
\draw[red,thick] (0,0)--(0,8);
\draw[red,thick] (12,0)--(12,8);
\draw[red,postaction={decorate}] (5,4) -- (8,4);
\draw[blue,postaction={decorate}] (3,4) -- (5,4);
\draw[blue,dashed,postaction={decorate}] (8,4) to node[midway,above] (k) {$k$} (12,4);
\draw[yellow!80!black] (12,1)--(2.5,1)--(2.5,5)--(4,5)--(4,4);
\coordinate (kk) at (12,3);
\coordinate (kkr) at (12,2.5);
\node[red] at (3,4) {$\times$};
\draw[yellow!80!black,double,thick] (k)--(k|-kk)--(kk);
\draw[yellow!80!black] (kr)--(kr|-kkr)--(kkr);
\end{tikzpicture}}}+ 
\vcenter{\hbox{
\begin{tikzpicture}[scale=0.22,every node/.style={scale=0.5},decoration={
    markings,
    mark=at position 0.5 with {\arrow{>}}}]
\draw[thick] (0,0) rectangle ++(12,8);
\draw[fill = green!30!white] (0,0) rectangle ++(6,8);
\draw[red,thick] (0,0)--(0,8);
\draw[red,thick] (12,0)--(12,8);
\draw[blue,dashed,postaction={decorate}] (8,4) to node[midway,above] (k) {$k$} (12,4);
\draw[red,postaction={decorate}] (3,4) -- (8,4);
\draw[blue,postaction={decorate}] (8,4) to[bend left] (12,6);
\draw[yellow!80!black] (12,2)--(2.5,2)--(2.5,5)--(9,5);
\draw[yellow!80!black] (12,2.5)--(4.5,2.5)--(4.5,4);
\coordinate (kk) at (12,3);
\coordinate (kkr) at (12,2.5);
\node[red] at (3,4) {$\times$};
\draw[yellow!80!black,double,thick] (k)--(k|-kk)--(kk);
\end{tikzpicture}}}- 
\vcenter{\hbox{
\begin{tikzpicture}[scale=0.22,every node/.style={scale=0.5},decoration={
    markings,
    mark=at position 0.5 with {\arrow{>}}}]
\draw[thick] (0,0) rectangle ++(12,8);
\draw[fill = green!30!white] (0,0) rectangle ++(6,8);
\draw[red,thick] (0,0)--(0,8);
\draw[red,thick] (12,0)--(12,8);
\draw[red,postaction={decorate}] (5,4) -- (8,4);
\draw[blue,postaction={decorate}] (3,4) -- (5,4);
\draw[blue,dashed,postaction={decorate}] (8,4) to node[midway,above] (k) {$k$} (12,4);
\draw[yellow!80!black] (12,2)--(6.5,2)--(6.5,4);
\draw[yellow!80!black] (12,1)--(4,1)--(4,4);
\coordinate (kk) at (12,3);
\coordinate (kkr) at (12,2.5);
\node[red] at (3,4) {$\times$};
\draw[yellow!80!black,double,thick] (k)--(k|-kk)--(kk);
\end{tikzpicture}}}-
\vcenter{\hbox{
\begin{tikzpicture}[scale=0.22,every node/.style={scale=0.5},decoration={
    markings,
    mark=at position 0.5 with {\arrow{>}}}]
\draw[thick] (0,0) rectangle ++(12,8);
\draw[fill = green!30!white] (0,0) rectangle ++(6,8);
\draw[red,thick] (0,0)--(0,8);
\draw[red,thick] (12,0)--(12,8);
\draw[red,postaction={decorate}] (3,4) to coordinate[pos=0.7] (kr) (8,4);
\draw[blue,dashed,postaction={decorate}] (8,4) to node[midway,above] (k) {$k$} (12,4);
\draw[blue,postaction={decorate}]  (8,4) .. controls (8,2) and (9,2) .. (10,2) -- (12,2);
\coordinate (kk) at (12,3);
\coordinate (kkr) at (12,2.5);
\node[red] at (3,4) {$\times$};
\draw[yellow!80!black,double,thick] (k)--(k|-kk)--(kk);
\draw[yellow!80!black] (kr)--(kr|-kkr)--(kkr);
\node[yellow!80!black] at (12,2) {$\bullet$};
\end{tikzpicture}}}
\end{align*}
Here is a step by step explanation of the above computation:
\begin{itemize}
\item The left hand side is a product of a class defined as in Rem. \ref{rmk_Def_diagrams} (3) (for points on the horizontal axis) with the interval (in blue, that runs once around the horizontal axis) on which one blue point is embedded. The yellow handles are used to choose a lift, the yellow point means a constant handle (fixed on this point). 
\item The first equality is an application of the cutting rule (Prop.~\ref{prop_breaking_rule}) to the blue arc. It splits it on the puncture $P$ in two pieces. 
\item Second equality is again two applications of the cutting rule to the blue arcs, expressing each of former terms into the sum of two. Notice that this time it is not cut at a puncture (see Rem.~\ref{R:cutting_in_chains}).
\end{itemize}
As puncture $P$ is transparent regarding the local system, loops only winding around it are trivial under $\rho_c$. One notes that first and last term of last four are the same: yellow handles are isotopic (up to $\rho_c$, i.e. their composition is trivial under $\rho_c$) so that by the handle rule (Rem.~\ref{rmk_handle_rule}) they are equal. We hence have:
\begin{align*} 
\vcenter{\hbox{ 
\begin{tikzpicture}[scale=0.22,every node/.style={scale=0.5},decoration={
    markings,
    mark=at position 0.3 with {\arrow{>}}}]
\draw[thick] (0,0) rectangle ++(12,8);
\draw[fill = green!30!white] (0,0) rectangle ++(6,8);
\draw[red,thick] (0,0)--(0,8);
\draw[red,thick] (12,0)--(12,8);
\draw[red,postaction={decorate}] (4,4) to coordinate[pos=0.6] (kr) (8,4);
\draw[blue,dashed,postaction={decorate}] (8,4) to node[midway,above] (k) {$k$} (12,4);
\draw[blue,postaction={decorate}] (12,2) -- (6,2) .. controls (0,2) and (0,6) .. (6,6) -- (12,6);
\node[yellow!80!black] at (12,2) {$\bullet$};
\coordinate (kk) at (12,3);
\coordinate (kkr) at (12,2.5);
\node[red] at (4,4) {$\times$};
\draw[yellow!80!black,double,thick] (k)--(k|-kk)--(kk);
\draw[yellow!80!black] (kr)--(kr|-kkr)--(kkr);
\end{tikzpicture}}}
& = 
\vcenter{\hbox{
\begin{tikzpicture}[scale=0.22,every node/.style={scale=0.5},decoration={
    markings,
    mark=at position 0.5 with {\arrow{>}}}]
\draw[thick] (0,0) rectangle ++(12,8);
\draw[fill = green!30!white] (0,0) rectangle ++(6,8);
\draw[red,thick] (0,0)--(0,8);
\draw[red,thick] (12,0)--(12,8);
\draw[blue,dashed,postaction={decorate}] (8,4) to node[midway,above] (k) {$k$} (12,4);
\draw[red,postaction={decorate}] (3,4) -- (8,4);
\draw[blue,postaction={decorate}] (8,4) to[bend left] (12,6);
\draw[yellow!80!black] (12,2)--(2.5,2)--(2.5,5)--(9,5);
\draw[yellow!80!black] (12,2.5)--(4.5,2.5)--(4.5,4);
\coordinate (kk) at (12,3);
\coordinate (kkr) at (12,2.5);
\node[red] at (3,4) {$\times$};
\draw[yellow!80!black,double,thick] (k)--(k|-kk)--(kk);
\end{tikzpicture}}}-
\vcenter{\hbox{
\begin{tikzpicture}[scale=0.22,every node/.style={scale=0.5},decoration={
    markings,
    mark=at position 0.5 with {\arrow{>}}}]
\draw[thick] (0,0) rectangle ++(12,8);
\draw[fill = green!30!white] (0,0) rectangle ++(6,8);
\draw[red,thick] (0,0)--(0,8);
\draw[red,thick] (12,0)--(12,8);
\draw[red,postaction={decorate}] (3,4) to coordinate[pos=0.7] (kr) (8,4);
\draw[blue,dashed,postaction={decorate}] (8,4) to node[midway,above] (k) {$k$} (12,4);
\draw[blue,postaction={decorate}]  (8,4) .. controls (8,2) and (9,2) .. (10,2) -- (12,2);
\coordinate (kk) at (12,3);
\coordinate (kkr) at (12,2.5);
\node[red] at (3,4) {$\times$};
\draw[yellow!80!black,double,thick] (k)--(k|-kk)--(kk);
\draw[yellow!80!black] (kr)--(kr|-kkr)--(kkr);
\node[yellow!80!black] at (12,2) {$\bullet$};
\end{tikzpicture}}}\\
& = 
q_{\alpha}^{-k} q_{\alpha,\beta}^{-1} \vcenter{\hbox{
\begin{tikzpicture}[scale=0.22,every node/.style={scale=0.5},decoration={
    markings,
    mark=at position 0.5 with {\arrow{>}}}]
\draw[thick] (0,0) rectangle ++(12,8);
\draw[fill = green!30!white] (0,0) rectangle ++(6,8);
\draw[red,thick] (0,0)--(0,8);
\draw[red,thick] (12,0)--(12,8);
\draw[blue,dashed,postaction={decorate}] (8,4) to node[midway,above] (k) {$k$} (12,4);
\draw[red,postaction={decorate}] (3,4) -- (8,4);
\draw[blue,postaction={decorate}] (8,4) to[bend left] (12,6);
\draw[yellow!80!black] (12,3.5)--(11,3.5)--(11,6);
\draw[yellow!80!black] (12,2.5)--(4.5,2.5)--(4.5,4);
\coordinate (kk) at (12,3);
\coordinate (kkr) at (12,2.5);
\node[red] at (3,4) {$\times$};
\draw[yellow!80!black,double,thick] (k)--(k|-kk)--(kk);
\end{tikzpicture}}}-
q_{\alpha,\beta} \vcenter{\hbox{
\begin{tikzpicture}[scale=0.22,every node/.style={scale=0.5},decoration={
    markings,
    mark=at position 0.5 with {\arrow{>}}}]
\draw[thick] (0,0) rectangle ++(12,8);
\draw[fill = green!30!white] (0,0) rectangle ++(6,8);
\draw[red,thick] (0,0)--(0,8);
\draw[red,thick] (12,0)--(12,8);
\draw[red,postaction={decorate}] (3,4) to coordinate[pos=0.7] (kr) (8,4);
\draw[blue,dashed,postaction={decorate}] (8,4) to node[pos=0.7,above] (k) {$k$} (12,4);
\draw[blue,postaction={decorate}]  (8,4) .. controls (8,2) and (9,2) .. (10,2) -- (12,2);
\coordinate (kk) at (12,1.5);
\coordinate (kkr) at (12,0.5);
\node[red] at (3,4) {$\times$};
\draw[yellow!80!black,double,thick] (k)--(k|-kk)--(kk);
\draw[yellow!80!black] (kr)--(kr|-kkr)--(kkr);
\draw[yellow!80!black] (10,2)--(10,1)--(12,1);
\end{tikzpicture}}}\\
& = 
(k+1)_{q_{\alpha}^{-1}} \left(q_{\alpha}^{-k} q_{\alpha,\beta}^{-1} - q_{\alpha,\beta}\right) \vcenter{\hbox{
\begin{tikzpicture}[scale=0.22,every node/.style={scale=0.5},decoration={
    markings,
    mark=at position 0.5 with {\arrow{>}}}]
\draw[thick] (0,0) rectangle ++(12,8);
\draw[fill = green!30!white] (0,0) rectangle ++(6,8);
\draw[red,thick] (0,0)--(0,8);
\draw[red,thick] (12,0)--(12,8);
\draw[blue,dashed,postaction={decorate}] (8,4) to node[midway,above] (k) {$k$} (12,4);
\draw[red,postaction={decorate}] (3,4) -- (8,4);
\draw[yellow!80!black] (12,2.5)--(4.5,2.5)--(4.5,4);
\coordinate (kk) at (12,3);
\coordinate (kkr) at (12,2.5);
\node[red] at (3,4) {$\times$};
\draw[yellow!80!black,double,thick] (k)--(k|-kk)--(kk);
\end{tikzpicture}}}
\end{align*}
In chronological order of operations:
\begin{itemize}
\item First line is the starting point from previous observations,
\item second line is a reorganization of handles using handle rule (Rem.~\ref{rmk_handle_rule}),
\item last equality comes from a fusion rule (Prop.~\ref{prop_fusion_rules}).
\end{itemize}
Finally, the result is the following, for all $k \in \BN$:
\begin{equation}\label{proof_squashing_recursion}
\vcenter{\hbox{ 
\begin{tikzpicture}[scale=0.22,every node/.style={scale=0.5},decoration={
    markings,
    mark=at position 0.3 with {\arrow{>}}}]
\draw[thick] (0,0) rectangle ++(12,8);
\draw[fill = green!30!white] (0,0) rectangle ++(6,8);
\draw[red,thick] (0,0)--(0,8);
\draw[red,thick] (12,0)--(12,8);
\draw[red,postaction={decorate}] (4,4) to coordinate[pos=0.6] (kr) (8,4);
\draw[blue,dashed,postaction={decorate}] (8,4) to node[midway,above] (k) {$k$} (12,4);
\draw[blue,postaction={decorate}] (12,2) -- (6,2) .. controls (0,2) and (0,6) .. (6,6) -- (12,6);
\node[yellow!80!black] at (12,2) {$\bullet$};
\coordinate (kk) at (12,3);
\coordinate (kkr) at (12,2.5);
\node[red] at (4,4) {$\times$};
\draw[yellow!80!black,double,thick] (k)--(k|-kk)--(kk);
\draw[yellow!80!black] (kr)--(kr|-kkr)--(kkr);
\end{tikzpicture}}}
=
(k+1)_{q_{\alpha}^{-1}} \left(q_{\alpha}^{-k} q_{\alpha,\beta}^{-1} - q_{\alpha,\beta}\right) \vcenter{\hbox{
\begin{tikzpicture}[scale=0.22,every node/.style={scale=0.5},decoration={
    markings,
    mark=at position 0.5 with {\arrow{>}}}]
\draw[thick] (0,0) rectangle ++(12,8);
\draw[fill = green!30!white] (0,0) rectangle ++(6,8);
\draw[red,thick] (0,0)--(0,8);
\draw[red,thick] (12,0)--(12,8);
\draw[blue,dashed,postaction={decorate}] (8,4) to node[midway,above] (k) {$k$} (12,4);
\draw[red,postaction={decorate}] (3,4) -- (8,4);
\draw[yellow!80!black] (12,2.5)--(4.5,2.5)--(4.5,4);
\coordinate (kk) at (12,3);
\coordinate (kkr) at (12,2.5);
\node[red] at (3,4) {$\times$};
\draw[yellow!80!black,double,thick] (k)--(k|-kk)--(kk);
\end{tikzpicture}}}
\end{equation}
By a direct recursion, one finds:
\[
\qSerre_{\alpha,\beta}^{k} = (k)_{q_{\alpha}^{-1}}! \prod_{l=0}^{k-1} \left(q_{\alpha}^{-l} q_{\alpha,\beta}^{-1} - q_{\alpha,\beta}\right) \vcenter{\hbox{
\begin{tikzpicture}[scale=0.22,every node/.style={scale=0.5},decoration={
    markings,
    mark=at position 0.5 with {\arrow{>}}}]
\draw[thick] (0,0) rectangle ++(12,8);
\draw[fill = green!30!white] (0,0) rectangle ++(6,8);
\draw[red,thick] (0,0)--(0,8);
\draw[red,thick] (12,0)--(12,8);
\draw[blue,dashed,postaction={decorate}] (8,4) to node[midway,above] (k) {$k$} (12,4);
\draw[red,postaction={decorate}] (3,4) -- (8,4);
\draw[yellow!80!black] (12,2.5)--(4.5,2.5)--(4.5,4);
\coordinate (kk) at (12,3);
\coordinate (kkr) at (12,2.5);
\node[red] at (3,4) {$\times$};
\draw[yellow!80!black,double,thick] (k)--(k|-kk)--(kk);
\end{tikzpicture}}} .
\]
Using Lemma \ref{lemma_qSerre_divpower_setup}, we must simplify by $(k)_{q_{\alpha}}!$ to obtain $\qSerre_{\alpha,\beta}^{(k)}$ which concludes the proof. 
\end{proof}

The second expression is obtained by \textit{partitioning} the involved homology class. 

\begin{prop}[Partitioning]\label{prop_partition}
In ${\CHbm_c}'$ we have the following equality:
\[
\qSerre_{\alpha,\beta}^{(k)} = \sum_{l=0}^k (-1)^{k-l}  q_{\alpha, \beta}^{-l} q_{\alpha}^{-l(l-1)/2} 
\vcenter{\hbox{
\begin{tikzpicture}[scale=0.3,every node/.style={scale=0.6},decoration={
    markings,
    mark=at position 0.5 with {\arrow{>}}}]
\draw[thick] (0,0) rectangle ++(12,8);
\draw[fill = green!30!white] (0,0) rectangle ++(6,8);
\draw[red,thick] (0,0)--(0,8);
\draw[red,thick] (12,0)--(12,8);
\draw[red,postaction={decorate}]  (4,4) -- (12,4);
\node[red,above right] at (4,4) {$P$};
\node[red] at (4,4) {$\times$};
\coordinate (ll) at (12,3);
\coordinate (kll) at (12,7);
\draw[blue,dashed,postaction={decorate}] (0,6) to node[pos=0.7,below] (kl) {$l$} (12,6);
\draw[blue,dashed,postaction={decorate}] (0,2) to node[pos=0.7,below] (l) {$k-l$} (12,2);
\draw[yellow!80!black,double,thick] (l)--(l|-ll)--(ll);
\draw[yellow!80!black,double,thick] (kl)--(kl|-kll)--(kll);
\node[right,yellow!80!black] at (12,4) {$q$};
\node[yellow!80!black] at (12,4) {$\bullet$};
\end{tikzpicture}}}
\]
\end{prop}
\begin{proof}
In ${\CHbm_c}'$ we have:
\begin{align*}
\vcenter{\hbox{
\begin{tikzpicture}[scale=0.22,every node/.style={scale=0.6},decoration={
    markings,
    mark=at position 0.5 with {\arrow{>}}}]
\draw[thick] (0,0) rectangle ++(12,8);
\draw[fill = green!30!white] (0,0) rectangle ++(6,8);
\draw[red,thick] (0,0)--(0,8);
\draw[red,thick] (12,0)--(12,8);
\draw[red,postaction={decorate}]  (4,4) -- (12,4);
\node[red,above right] at (4,4) {$P$};
\node[red] at (4,4) {$\times$};
\coordinate (kk) at (12,3);
\draw[blue,dashed,postaction={decorate}] (12,2) to coordinate[midway] (k) (6,2) .. controls (0,2) and (0,6) .. (6,6) to node[midway,above] {$k$} (12,6);
\draw[yellow!80!black,double,thick] (k)--(k|-kk)--(kk);
\node[right,red] at (12,4) {$q$};
\end{tikzpicture}}}
& = \sum_{l=0}^k (-1)^{k-l}
\vcenter{\hbox{
\begin{tikzpicture}[scale=0.22,every node/.style={scale=0.5},decoration={
    markings,
    mark=at position 0.5 with {\arrow{>}}}]
\draw[thick] (0,0) rectangle ++(12,8);
\draw[fill = green!30!white] (0,0) rectangle ++(6,8);
\draw[red,thick] (0,0)--(0,8);
\draw[red,thick] (12,0)--(12,8);
\draw[red,postaction={decorate}]  (4,4) -- (12,4);
\node[red,above right] at (4,4) {$P$};
\node[red] at (4,4) {$\times$};
\coordinate (ll) at (12,2.5);
\coordinate (kll) at (12,3);
\draw[blue,dashed,postaction={decorate}] (0,6) to node[pos=0.2,above] (kl) {$l$} (12,6);
\draw[blue,dashed,postaction={decorate}] (0,2) to node[pos=0.7,below] (l) {$k-l$} (12,2);
\draw[yellow!80!black,double,thick] (l)--(l|-ll)--(ll);
\draw[yellow!80!black,double,thick] (kl)--(kl|-kll)--(kll);
\node[right,yellow!80!black] at (12,4) {$q$};
\node[yellow!80!black] at (12,4) {$\bullet$};
\end{tikzpicture}}} \\
& = \sum_{l=0}^k (-1)^{k-l}  q_{\alpha, \beta}^{-l} q_{\alpha}^{-l(l-1)/2} 
\vcenter{\hbox{
\begin{tikzpicture}[scale=0.22,every node/.style={scale=0.5},decoration={
    markings,
    mark=at position 0.5 with {\arrow{>}}}]
\draw[thick] (0,0) rectangle ++(12,8);
\draw[fill = green!30!white] (0,0) rectangle ++(6,8);
\draw[red,thick] (0,0)--(0,8);
\draw[red,thick] (12,0)--(12,8);
\draw[red,postaction={decorate}]  (4,4) -- (12,4);
\node[red,above right] at (4,4) {$P$};
\node[red] at (4,4) {$\times$};
\coordinate (ll) at (12,3);
\coordinate (kll) at (12,7);
\draw[blue,dashed,postaction={decorate}] (0,6) to node[pos=0.7,below] (kl) {$l$} (12,6);
\draw[blue,dashed,postaction={decorate}] (0,2) to node[pos=0.7,below] (l) {$k-l$} (12,2);
\draw[yellow!80!black,double,thick] (l)--(l|-ll)--(ll);
\draw[yellow!80!black,double,thick] (kl)--(kl|-kll)--(kll);
\node[right,yellow!80!black] at (12,4) {$q$};
\node[yellow!80!black] at (12,4) {$\bullet$};
\end{tikzpicture}}}
\end{align*}
The first equality is an application of the cutting rule (Prop.~\ref{prop_breaking_rule}~(2)). The second one is a reorganization of yellow handles, coefficients showing up are deduced from the handle rule (Rem.~\ref{rmk_handle_rule}). It concludes the proof. 
\end{proof}

\subsubsection{Quantum-Serre relation}\label{s:quantum_serre}

We begin with a corollary of the two propositions from the previous section, where we can relax working in ${\CHbm_c}'$ and work in $\CHbm_c$ instead. 

\begin{coro}\label{coro_ident_partition_squash}
For all $k \in \BN$, the following holds in $\CHbm_c$:
\[
\sum_{l=0}^k (-1)^{k-l}  q_{\alpha, \beta}^{-l} q_{\alpha}^{-l(l-1)/2} 
\vcenter{\hbox{
\begin{tikzpicture}[scale=0.3,every node/.style={scale=0.6},decoration={
    markings,
    mark=at position 0.5 with {\arrow{>}}}]
\draw[thick] (0,0) rectangle ++(12,8);
\draw[red,thick] (0,0)--(0,8);
\draw[red,thick] (12,0)--(12,8);
\draw[red,postaction={decorate}]  (0,4) -- (12,4);
\coordinate (ll) at (12,3);
\coordinate (kll) at (12,7);
\draw[blue,dashed,postaction={decorate}] (0,6) to node[pos=0.7,below] (kl) {$l$} (12,6);
\draw[blue,dashed,postaction={decorate}] (0,2) to node[pos=0.7,below] (l) {$k-l$} (12,2);
\draw[yellow!80!black,double,thick] (l)--(l|-ll)--(ll);
\draw[yellow!80!black,double,thick] (kl)--(kl|-kll)--(kll);
\node[right,yellow!80!black] at (12,4) {$q$};
\node[yellow!80!black] at (12,4) {$\bullet$};
\end{tikzpicture}}} 
= 
q_{\alpha}^{\frac{-k(k-1}{2}}\prod_{l=0}^{k-1} \left( q_{\alpha}^{-l} q_{\alpha,\beta}^{-1} - q_{\alpha,\beta} \right)  \vcenter{\hbox{
\begin{tikzpicture}[scale=0.3,decoration={
    markings,
    mark=at position 0.5 with {\arrow{>}}}]
\draw[thick] (0,0) rectangle ++(12,8);
\draw[red,thick] (0,0)--(0,8);
\draw[red,thick] (12,0)--(12,8);
\draw[red,postaction={decorate}]  (0,4) to coordinate[pos=0.6] (kr) (6,4);
\draw[blue,dashed,postaction={decorate}]  (6,4) to node[midway,above] (k) {$k$} (12,4);
\coordinate (kk) at (12,2);
\coordinate (kkr) at (12,1);
\draw[yellow!80!black,double,thick] (k)--(k|-kk)--(kk);
\draw[yellow!80!black] (kr)--(kr|-kkr)--(kkr);
\end{tikzpicture}}}
\]
\end{coro}
\begin{proof}
A summary of Propositions~\ref{prop_squashing} and \ref{prop_partition} both expressing $\qSerre_{\alpha,\beta}^{(k)}$ in ${\CHbm_c}'$ is the following:
\[
\sum_{l=0}^k (-1)^{k-l}  q_{\alpha, \beta}^{-l} q_{\alpha}^{-l(l-1)/2} 
\vcenter{\hbox{
\begin{tikzpicture}[scale=0.3,every node/.style={scale=0.6},decoration={
    markings,
    mark=at position 0.5 with {\arrow{>}}}]
\draw[thick] (0,0) rectangle ++(12,8);
\draw[fill = green!30!white] (0,0) rectangle ++(6,8);
\draw[red,thick] (0,0)--(0,8);
\draw[red,thick] (12,0)--(12,8);
\draw[red,postaction={decorate}]  (4,4) -- (12,4);
\node[red,above right] at (4,4) {$P$};
\node[red] at (4,4) {$\times$};
\coordinate (ll) at (12,3);
\coordinate (kll) at (12,7);
\draw[blue,dashed,postaction={decorate}] (0,6) to node[pos=0.7,below] (kl) {$l$} (12,6);
\draw[blue,dashed,postaction={decorate}] (0,2) to node[pos=0.7,below] (l) {$k-l$} (12,2);
\draw[yellow!80!black,double,thick] (l)--(l|-ll)--(ll);
\draw[yellow!80!black,double,thick] (kl)--(kl|-kll)--(kll);
\node[right,yellow!80!black] at (12,4) {$q$};
\node[yellow!80!black] at (12,4) {$\bullet$};
\end{tikzpicture}}} 
= 
q_{\alpha}^{\frac{-k(k-1}{2}}\prod_{l=0}^{k-1} \left( q_{\alpha}^{-l} q_{\alpha,\beta}^{-1} - q_{\alpha,\beta} \right)  \vcenter{\hbox{
\begin{tikzpicture}[scale=0.3,decoration={
    markings,
    mark=at position 0.5 with {\arrow{>}}}]
\draw[thick] (0,0) rectangle ++(12,8);
\draw[fill = green!30!white] (0,0) rectangle ++(6,8);
\draw[red,thick] (0,0)--(0,8);
\draw[red,thick] (12,0)--(12,8);
\draw[red,postaction={decorate}]  (4,4) to coordinate[pos=0.6] (kr) (8,4);
\draw[blue,dashed,postaction={decorate}]  (8,4) to node[midway,above] (k) {$k$} (12,4);
\node[red] at (4,4) {$\times$};
\coordinate (kk) at (12,2);
\coordinate (kkr) at (12,1);
\draw[yellow!80!black,double,thick] (k)--(k|-kk)--(kk);
\draw[yellow!80!black] (kr)--(kr|-kkr)--(kkr);
\end{tikzpicture}}}
\]
which are two expressions for $\qSerre_{\alpha,\beta}^{(k)}$ only defined in ${\CHbm_c}'$. Under the isomorphism $\CHbm_c \simeq {\CHbm_c}'$, we get the corollary.
\end{proof}

\begin{prop}[Quantum Serre relations]\label{prop_QuantumSerre}
Let $\alpha_i$ and $\alpha_j$ be two different roots in $\Pi$. We recall that $a_{i,j} := \frac{(\alpha_i, \alpha_j)}{d_i}$. We have:
\[
\sum_{l=0}^{1-a_{i,j}} (-1)^l \CF_{\alpha}^{(l)} \CF_{\beta}^{(1)} \CF_{\alpha}^{(1-a_{i,j})-l)} = 0
\]
\end{prop}
\begin{proof}
It is a consequence of the equality in Coro. \ref{coro_ident_partition_squash} evaluated for $k= 1-a_{i,j}$. For this value of $k$, the RHS vanishes as:
\[
q_{\alpha_i}^{a_{i,j}} q_{\alpha_i,\beta_i}^{-1} - q_{\alpha_i, \beta_j} = q^{d_i \frac{(\alpha_i,\alpha_j)}{d_i} - \frac{(\alpha_i,\alpha_j)}{2}} - q^{(\alpha_i, \alpha_j)/2} = 0
\]
is a factor (the last) of the involved product. The LHS provides the left hand linear combination of the proposition. Namely, the RHS is:
\begin{equation}\label{proof_prop_qSerre_LHS}
\sum_{l=0}^k (-1)^{k-l}  q_{\alpha, \beta}^{-l} q_{\alpha}^{-l(l-1)/2} 
\vcenter{\hbox{
\begin{tikzpicture}[scale=0.3,every node/.style={scale=0.6},decoration={
    markings,
    mark=at position 0.5 with {\arrow{>}}}]
\draw[thick] (0,0) rectangle ++(12,8);
\draw[red,thick] (0,0)--(0,8);
\draw[red,thick] (12,0)--(12,8);
\draw[red,postaction={decorate}]  (0,4) -- (12,4);
\coordinate (ll) at (12,3);
\coordinate (kll) at (12,7);
\draw[blue,dashed,postaction={decorate}] (0,6) to node[pos=0.7,below] (kl) {$l$} (12,6);
\draw[blue,dashed,postaction={decorate}] (0,2) to node[pos=0.7,below] (l) {$k-l$} (12,2);
\draw[yellow!80!black,double,thick] (l)--(l|-ll)--(ll);
\draw[yellow!80!black,double,thick] (kl)--(kl|-kll)--(kll);
\node[right,yellow!80!black] at (12,4) {$q$};
\node[yellow!80!black] at (12,4) {$\bullet$};
\end{tikzpicture}}}
\end{equation}
for $k:= -a_{i,j} + 1$. From the product defined in Sec. \ref{sec_product} and the divided powers defined in Def. \ref{def_div_powers_Fs}, one notices that:
\[
\vcenter{\hbox{
\begin{tikzpicture}[scale=0.3,every node/.style={scale=0.6},decoration={
    markings,
    mark=at position 0.5 with {\arrow{>}}}]
\draw[thick] (0,0) rectangle ++(12,8);
\draw[red,thick] (0,0)--(0,8);
\draw[red,thick] (12,0)--(12,8);
\draw[red,postaction={decorate}]  (0,4) -- (12,4);
\coordinate (ll) at (12,3);
\coordinate (kll) at (12,7);
\draw[blue,dashed,postaction={decorate}] (0,6) to node[pos=0.7,below] (kl) {$k-l$} (12,6);
\draw[blue,dashed,postaction={decorate}] (0,2) to node[pos=0.7,below] (l) {$l$} (12,2);
\draw[yellow!80!black,double,thick] (l)--(l|-ll)--(ll);
\draw[yellow!80!black,double,thick] (kl)--(kl|-kll)--(kll);
\node[right,yellow!80!black] at (12,4) {$q$};
\node[yellow!80!black] at (12,4) {$\bullet$};
\end{tikzpicture}}}=q^{d_{\alpha} \frac{l(l-1)}{2}} \CF_{\alpha}^{(l)} \CF_{\beta}^{(1)} q^{d_{\alpha} \frac{(k-l)(k-l-1}{2}} \CF_{\alpha}^{(k-l)} 
\]
Finally, recalling that $q_{\alpha,\beta} = q^{d_{\alpha} (1-k)}$, Eq.~\eqref{proof_prop_qSerre_LHS} is:
\begin{align}
\sum_{l=0}^k (-1)^{k-l}  q_{\alpha, \beta}^{-l} q_{\alpha}^{-l(l-1)/2} q^{d_{\alpha} \frac{l(l-1)}{4}} \CF_{\alpha}^{(l)} \CF_{\beta}^{(1)} q^{d_{\alpha} \frac{(k-l)(k-l-1}{4}} \CF_{\alpha}^{(k-l)} & 
= \sum_{l=0}^k (-1)^{k-l} q_{\alpha}^{\frac{1}{2} \left( l(k-1) + \frac{(k-l)(k-l-1)-l(l-1)}{2} \right)} \CF_{\alpha}^{(l)} \CF_{\beta}^{(1)} \CF_{\alpha}^{(k-l)}\\
& = \sum_{l=0}^k (-1)^{k-l} q_{\alpha}^{\frac{1}{2} \left( l(k-1) + l(1-k) + \frac{k(k-1)}{2} \right)} \CF_{\alpha}^{(l)} \CF_{\beta}^{(1)} \CF_{\alpha}^{(k-l)}\\
& = (-1)^kq_{\alpha}^{\frac{k(k-1)}{4}} \sum_{l=0}^k (-1)^l \CF_{\alpha}^{(l)} \CF_{\beta}^{(1)} \CF_{\alpha}^{(k-l)}
\end{align}
which concludes the proof. 
\end{proof}

\begin{thm}[Homological and integral version of $U_q^{<0} \fg$]\label{T:homological_version_for_Uqgm}
Let $\Pi$ be a root system, and $\Coloring_{\Pi}$ be the set of colorings by $\Pi$. For $\mq := q^{\frac{1}{2}}$, there is an algebra morphism:
\[
\bapp
U_{\mq}^{<0} \fg & \to & \CHbm:= \bigoplus_{c \in \Coloring_{\Pi}} \CHbm_c \\
F_{\alpha}^{(k)} & \mapsto & \CF_{\alpha}^{(k)} \text{ for all } \alpha \in \Pi
\eapp .
\]
Above algebras are $\BZ \left[\mq^{\pm 1} \right]$ resp. $\BZ \left[q^{\pm 1} \right]$ algebras. 
\end{thm}
\begin{proof}
The fact that it is a well defined algebra morphism comes from the fact that the generators satisfy the \textit{Quantum Serre relation} Prop.~\ref{prop_QuantumSerre} and the \textit{divided power property} Prop.~\ref{prop_DividedPowers}. It is the two relations holding in the integral version of the strictly negative part of the quantum group $U_{\mq} \fg$, see \cite{Lus90}. It justifies the homological construction. (This strictly negative part being defined by that generated by divided powers of $F$'s). 
\end{proof}

\begin{rmk}
\begin{itemize}
\item In the next section, we will show that the morphism from the above theorem is injective while working in the field, which justifies the name of integral version. We will use another approach defining the map without any mention to the quantum Serre relations. This theorem emphasizes a homological recovering of the presentation of the integral $\Uqgm$ with divided powers and quantum Serre relations.  

\item Another way to show that it is injective is to recover homologically the bases of Poincaré--Birkhoff--Witt type that Lusztig has provided for $\Uqgm$, e.g. in \cite[Part.~5]{Lus90} (from which we have adapted notations). This basis is described using some action of \emph{the braid group} associated with $\fg$ by automorphisms of $U_{\mq} \fg$. Hence it requires a little more work, and in the next section we also recover the part of this action that deals with $\Uqgm$.
\end{itemize} 
\end{rmk}

\subsection{From standard to Borel--Moore homology stands $\Uqgm$}\label{S:homology_to_Borel_Moore}

In the whole section we work on the field $\mathbb{Q}(q^{\frac{1}{2}})$. It fits with \cite[Sec.~1]{LusztigBook}. It means that we tensor all algebras with $\mathbb{Q}(q^{\frac{1}{2}})$ over the obvious injection of $\Laurent$. There is a map:
\[
\iota : \CH \to \CHbm
\]
induced by inclusion of chain complexes. Notice that source and target modules are isomorphic thanks to the Poincaré duality (see Sec.~\ref{A:the_pairing}) but this map is far from being an isomorphism. Indeed we show that it is a quotient by quantum Serre relations. We let:
\[
\overline{\CH} := \im (\iota: \CH \to \CHbm).
\]
Now recall that the there exists a perfect pairing:
\[
\langle \cdot, \cdot \rangle: \CH \otimes \CHbm \to \Laurent,
\]
obtained from that of \eqref{E:the_pairing} by first applying a $\pi/2$-rotation to the left term. Then while precomposing it by $1 \otimes \iota$ we obtain a bilinear form:
\[
(\cdot , \cdot) : \CH \otimes \CH \to \Laurent. 
\]
We follow Sec.~1.1 of \cite{LusztigBook} to show that $\Uqgm$, that he calls \textit{algebra $f$}, is the quotient of this pairing by its radical. It indeed relies on a pairing on the free algebra generated by simple roots that he denotes $'f$ and which is isomorphic to $\CH$ thanks to Theorem~\ref{T:free_algebra}. Following Lusztig's definition of $f$ we need to define an algebra morphism that resemble a coproduct and denoted by $r$ and to show that it satisfies axioms of Proposition~1.2.3 in his book. 

We start, as in \cite[1.2.2]{LusztigBook} by defining the map $r$:
\[
r : \CH \to \CH \otimes \CH, 
\]
by giving an algorithm of computation from diagrams. Let $\mathcal{D} \in \CH$ and $d$ be a vertical axis, the first step of $r$ is to define a cut along $d$. Up to isotopy one can put all the necklaces of $\CH$ in horizontal position, and $d$ to meet no ends of handles, then cutting along $d$ gives two diagrams in disks. Here is an example:
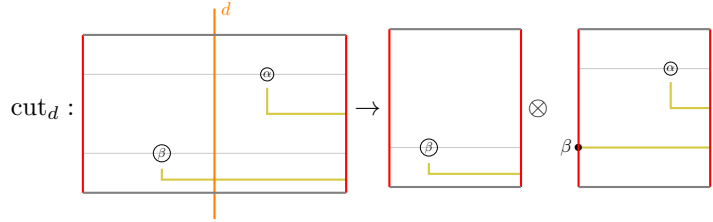
\[
\cut_d : \vcenter{\hbox{\begin{tikzpicture}[scale=0.35, every node/.style={scale=0.7},decoration={
    markings,
    mark=at position 0.5 with {\arrow{>}}}
    ]
\coordinate (w0g) at (-5,1.5) {};
\coordinate (w0d) at (5,1.5) {};

\coordinate (w1g) at (-5,-1.5) {};
\coordinate (w1d) at (5,-1.5) {};

\draw[gray!40!white, thin] (w0g) -- (w0d) node[pos=0.7, black] (a1) {$\encircled{\alpha}$};
\draw[gray!40!white, thin] (w1g) -- (w1d) node[pos=0.3, black] (a2) {$\encircled{\beta}$};

\coordinate (a1p) at (5,0);
\coordinate (a2p) at (5,-2.5);
\draw[yellow!80!black,thick] (a1p)--(a1p-|a1)--(a1);
\draw[yellow!80!black,thick] (a2p)--(a2p-|a2)--(a2);

\draw[thick,orange] (0,4)--(0,-4);
\node[right,orange] at (0,4) {$d$};

\draw[red, thick] (-5,-3) -- (-5,3);
\draw[red, thick] (5,-3) -- (5,3);
\draw[gray, thick] (-5,-3) -- (5,-3);
\draw[gray, thick] (-5,3) -- (5,3);
\end{tikzpicture}}}
\to 
\vcenter{\hbox{\begin{tikzpicture}[scale=0.35, every node/.style={scale=0.7},decoration={
    markings,
    mark=at position 0.5 with {\arrow{>}}}
    ]
\coordinate (w0g) at (-5,1.5) {};
\coordinate (w0d) at (5,1.5) {};

\coordinate (w1g) at (-5,-1.5) {};
\coordinate (w1d) at (0,-1.5) {};

\draw[gray!40!white, thin] (w1g) -- (w1d) node[pos=0.3, black] (a2) {$\encircled{\beta}$};

\coordinate (a1p) at (5,0);
\coordinate (a2p) at (0,-2.5);
\draw[yellow!80!black,thick] (a2p)--(a2p-|a2)--(a2);


\draw[red, thick] (-5,-3) -- (-5,3);
\draw[red, thick] (0,-3) -- (0,3);
\draw[gray, thick] (-5,-3) -- (0,-3);
\draw[gray, thick] (-5,3) -- (0,3);
\end{tikzpicture}}}
\otimes 
\vcenter{\hbox{\begin{tikzpicture}[scale=0.35, every node/.style={scale=0.7},decoration={
    markings,
    mark=at position 0.5 with {\arrow{>}}}
    ]
\coordinate (w0g) at (0,1.5) {};
\coordinate (w0d) at (5,1.5) {};

\coordinate (w1g) at (0,-1.5) {};
\coordinate (w1d) at (5,-1.5) {};

\draw[gray!40!white, thin] (w0g) -- (w0d) node[pos=0.7, black] (a1) {$\encircled{\alpha}$};
\draw[thick,\hyellow] (w1g) -- (w1d);
\node at (w1g) {$\bullet$};
\node[left] at (w1g) {$\beta$};

\coordinate (a1p) at (5,0);
\coordinate (a2p) at (5,-2.5);
\draw[yellow!80!black,thick] (a1p)--(a1p-|a1)--(a1);
%

\draw[red, thick] (0,-3) -- (0,3);
\draw[red, thick] (5,-3) -- (5,3);
\draw[gray, thick] (0,-3) -- (5,-3);
\draw[gray, thick] (0,3) -- (5,3);
\end{tikzpicture}}}
\]
We let $Y_c^i$ be the subspace of $\Conf_c$ made of configurations with at least $i$ points in $\partial^- D$. Notice that last diagram defines a homology class in $\Hnot_{1}(Y^1_{\alpha+\beta},Y^2_{\alpha+\beta},\Laurent_{\alpha+\beta})$. It corresponds to an embedding of $I$, that sends it to $\{x \} \times I_\alpha$ where $x$ is the point with the label $\beta$ and $I_\alpha$ the necklace supporting $\alpha$. A handle reaches a point in it (one goes to $x$ the other to $I$), hence it lifts to a twisted homology class. So to make $r$ well defined at homology one has to sum over all the cuts possible, namely one can move each pearl to the left or to the right of $d$ and sum over all possibilities, it defines a map $r'$. Here is an example:
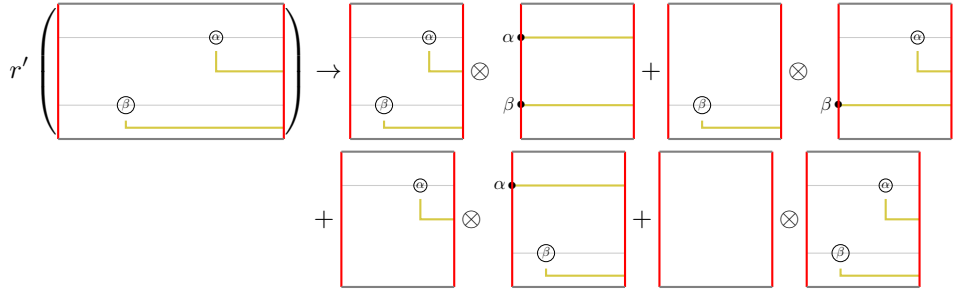
\begin{align*} 
r'\left( \vcenter{\hbox{\begin{tikzpicture}[scale=0.3, every node/.style={scale=0.7},decoration={
    markings,
    mark=at position 0.5 with {\arrow{>}}}
    ]
\coordinate (w0g) at (-5,1.5) {};
\coordinate (w0d) at (5,1.5) {};
\coordinate (w1g) at (-5,-1.5) {};
\coordinate (w1d) at (5,-1.5) {};
\draw[gray!40!white, thin] (w0g) -- (w0d) node[pos=0.7, black] (a1) {$\encircled{\alpha}$};
\draw[gray!40!white, thin] (w1g) -- (w1d) node[pos=0.3, black] (a2) {$\encircled{\beta}$};
\coordinate (a1p) at (5,0);
\coordinate (a2p) at (5,-2.5);
\draw[yellow!80!black,thick] (a1p)--(a1p-|a1)--(a1);
\draw[yellow!80!black,thick] (a2p)--(a2p-|a2)--(a2);
\draw[red, thick] (-5,-3) -- (-5,3);
\draw[red, thick] (5,-3) -- (5,3);
\draw[gray, thick] (-5,-3) -- (5,-3);
\draw[gray, thick] (-5,3) -- (5,3);
\end{tikzpicture}}}\right)
&
\to 
\vcenter{\hbox{\begin{tikzpicture}[scale=0.3, every node/.style={scale=0.7},decoration={
    markings,
    mark=at position 0.5 with {\arrow{>}}}
    ]
\coordinate (w0g) at (-5,1.5) {};
\coordinate (w0d) at (0,1.5) {};
\coordinate (w1g) at (-5,-1.5) {};
\coordinate (w1d) at (0,-1.5) {};
\draw[gray!40!white, thin] (w0g) -- (w0d) node[pos=0.7, black] (a1) {$\encircled{\alpha}$};
\draw[gray!40!white, thin] (w1g) -- (w1d) node[pos=0.3, black] (a2) {$\encircled{\beta}$};
\coordinate (a1p) at (0,0);
\coordinate (a2p) at (0,-2.5);
\draw[yellow!80!black,thick] (a1p)--(a1p-|a1)--(a1);
\draw[yellow!80!black,thick] (a2p)--(a2p-|a2)--(a2);
\draw[red, thick] (-5,-3) -- (-5,3);
\draw[red, thick] (0,-3) -- (0,3);
\draw[gray, thick] (-5,-3) -- (0,-3);
\draw[gray, thick] (-5,3) -- (0,3);
\end{tikzpicture}}}
\otimes 
\vcenter{\hbox{\begin{tikzpicture}[scale=0.3, every node/.style={scale=0.7},decoration={
    markings,
    mark=at position 0.5 with {\arrow{>}}}
    ]
\coordinate (w0g) at (0,1.5) {};
\coordinate (w0d) at (5,1.5) {};
\coordinate (w1g) at (0,-1.5) {};
\coordinate (w1d) at (5,-1.5) {};
\draw[thick,\hyellow] (w0g) -- (w0d);
\draw[thick,\hyellow] (w1g) -- (w1d);
\node at (w1g) {$\bullet$};
\node[left] at (w1g) {$\beta$};
\node at (w0g) {$\bullet$};
\node[left] at (w0g) {$\alpha$};
\coordinate (a1p) at (5,0);
\coordinate (a2p) at (5,-2.5);
%
\draw[red, thick] (0,-3) -- (0,3);
\draw[red, thick] (5,-3) -- (5,3);
\draw[gray, thick] (0,-3) -- (5,-3);
\draw[gray, thick] (0,3) -- (5,3);
\end{tikzpicture}}}
+ 
\vcenter{\hbox{\begin{tikzpicture}[scale=0.3, every node/.style={scale=0.7},decoration={
    markings,
    mark=at position 0.5 with {\arrow{>}}}
    ]
\coordinate (w0g) at (-5,1.5) {};
\coordinate (w0d) at (5,1.5) {};
\coordinate (w1g) at (-5,-1.5) {};
\coordinate (w1d) at (0,-1.5) {};
\draw[gray!40!white, thin] (w1g) -- (w1d) node[pos=0.3, black] (a2) {$\encircled{\beta}$};
\coordinate (a1p) at (5,0);
\coordinate (a2p) at (0,-2.5);
\draw[yellow!80!black,thick] (a2p)--(a2p-|a2)--(a2);
\draw[red, thick] (-5,-3) -- (-5,3);
\draw[red, thick] (0,-3) -- (0,3);
\draw[gray, thick] (-5,-3) -- (0,-3);
\draw[gray, thick] (-5,3) -- (0,3);
\end{tikzpicture}}}
\otimes 
\vcenter{\hbox{\begin{tikzpicture}[scale=0.3, every node/.style={scale=0.7},decoration={
    markings,
    mark=at position 0.5 with {\arrow{>}}}
    ]
\coordinate (w0g) at (0,1.5) {};
\coordinate (w0d) at (5,1.5) {};
\coordinate (w1g) at (0,-1.5) {};
\coordinate (w1d) at (5,-1.5) {};
\draw[gray!40!white, thin] (w0g) -- (w0d) node[pos=0.7, black] (a1) {$\encircled{\alpha}$};
\draw[thick,\hyellow] (w1g) -- (w1d);
\node at (w1g) {$\bullet$};
\node[left] at (w1g) {$\beta$};
\coordinate (a1p) at (5,0);
\coordinate (a2p) at (5,-2.5);
\draw[yellow!80!black,thick] (a1p)--(a1p-|a1)--(a1);
%
\draw[red, thick] (0,-3) -- (0,3);
\draw[red, thick] (5,-3) -- (5,3);
\draw[gray, thick] (0,-3) -- (5,-3);
\draw[gray, thick] (0,3) -- (5,3);
\end{tikzpicture}}} \\
& + 
\vcenter{\hbox{\begin{tikzpicture}[scale=0.3, every node/.style={scale=0.7},decoration={
    markings,
    mark=at position 0.5 with {\arrow{>}}}
    ]
\coordinate (w0g) at (-5,1.5) {};
\coordinate (w0d) at (0,1.5) {};
\coordinate (w1g) at (-5,-1.5) {};
\coordinate (w1d) at (0,-1.5) {};
\draw[gray!40!white, thin] (w0g) -- (w0d) node[pos=0.7, black] (a1) {$\encircled{\alpha}$};
\coordinate (a1p) at (0,0);
\coordinate (a2p) at (0,-2.5);
\draw[yellow!80!black,thick] (a1p)--(a1p-|a1)--(a1);
\draw[red, thick] (-5,-3) -- (-5,3);
\draw[red, thick] (0,-3) -- (0,3);
\draw[gray, thick] (-5,-3) -- (0,-3);
\draw[gray, thick] (-5,3) -- (0,3);
\end{tikzpicture}}}
\otimes 
\vcenter{\hbox{\begin{tikzpicture}[scale=0.3, every node/.style={scale=0.7},decoration={
    markings,
    mark=at position 0.5 with {\arrow{>}}}
    ]
\coordinate (w0g) at (0,1.5) {};
\coordinate (w0d) at (5,1.5) {};
\coordinate (w1g) at (0,-1.5) {};
\coordinate (w1d) at (5,-1.5) {};
\draw[thick,\hyellow] (w0g) -- (w0d);
\draw[gray!40!white, thin] (w1g) -- (w1d) node[pos=0.3, black] (a2) {$\encircled{\beta}$};
\node at (w0g) {$\bullet$};
\node[left] at (w0g) {$\alpha$};
\coordinate (a1p) at (5,0);
\coordinate (a2p) at (5,-2.5);
\draw[yellow!80!black,thick] (a2p)--(a2p-|a2)--(a2);
%
\draw[red, thick] (0,-3) -- (0,3);
\draw[red, thick] (5,-3) -- (5,3);
\draw[gray, thick] (0,-3) -- (5,-3);
\draw[gray, thick] (0,3) -- (5,3);
\end{tikzpicture}}}
+ 
\vcenter{\hbox{\begin{tikzpicture}[scale=0.3, every node/.style={scale=0.7},decoration={
    markings,
    mark=at position 0.5 with {\arrow{>}}}
    ]
\coordinate (w0g) at (-5,1.5) {};
\coordinate (w0d) at (0,1.5) {};
\coordinate (w1g) at (-5,-1.5) {};
\coordinate (w1d) at (0,-1.5) {};
\draw[red, thick] (-5,-3) -- (-5,3);
\draw[red, thick] (0,-3) -- (0,3);
\draw[gray, thick] (-5,-3) -- (0,-3);
\draw[gray, thick] (-5,3) -- (0,3);
\end{tikzpicture}}}
\otimes 
\vcenter{\hbox{\begin{tikzpicture}[scale=0.3, every node/.style={scale=0.7},decoration={
    markings,
    mark=at position 0.5 with {\arrow{>}}}
    ]
\coordinate (w0g) at (0,1.5) {};
\coordinate (w0d) at (5,1.5) {};
\coordinate (w1g) at (0,-1.5) {};
\coordinate (w1d) at (5,-1.5) {};
\draw[gray!40!white, thin] (w0g) -- (w0d) node[pos=0.7, black] (a1) {$\encircled{\alpha}$};
\draw[gray!40!white, thin] (w1g) -- (w1d) node[pos=0.3, black] (a2) {$\encircled{\beta}$};
\coordinate (a1p) at (5,0);
\coordinate (a2p) at (5,-2.5);
\draw[yellow!80!black,thick] (a1p)--(a1p-|a1)--(a1);
\draw[yellow!80!black,thick] (a2p)--(a2p-|a2)--(a2);
%
\draw[red, thick] (0,-3) -- (0,3);
\draw[red, thick] (5,-3) -- (5,3);
\draw[gray, thick] (0,-3) -- (5,-3);
\draw[gray, thick] (0,3) -- (5,3);
\end{tikzpicture}}}
\end{align*}
The map $r$ is obtained from $r'$ by an iteration of $\del_{x_0}$. Namely let $x_0$ be the lower rightmost point of $D$ then we recall the existence of an isomorphism:
\[
\del_{x_0} : \Hnot_{m_c}( Y_{c+\alpha}^1, Y_{c+\alpha}^2;\Laurent_c) \simeq \CH_c
\]
(see the definition in \cite[sec.~2.3.2]{JulesMarco} that adapts to the colored case). By iteration one has:
\[
\Hnot_{m_c}( Y_{c+c_1}^{m_{c_1}}, Y_{c+c_1}^{m_{c_1}+1};\Laurent_c) \simeq \CH_c,
\]
and a more formal argument is given after Eq~\eqref{E:i_boundary}, which adapts to our situation. By $\del_{x_0}$, one can remove a point in $\partial^- D$ if its handle reaches $x_0$. Recalling that, due to a handle rule, we have:
\[
\vcenter{\hbox{\begin{tikzpicture}[scale=0.3, every node/.style={scale=0.7},decoration={
    markings,
    mark=at position 0.5 with {\arrow{>}}}
    ]
\coordinate (w0g) at (0,1.5) {};
\coordinate (w0d) at (5,1.5) {};
\coordinate (w1g) at (0,-1.5) {};
\coordinate (w1d) at (5,-1.5) {};
\draw[thick,\hyellow] (w0g) -- (w0d);
\draw[gray!40!white, thin] (w1g) -- (w1d) node[pos=0.3, black] (a2) {$\encircled{\beta}$};
\node at (w0g) {$\bullet$};
\node[left] at (w0g) {$\alpha$};
\coordinate (a1p) at (5,0);
\coordinate (a2p) at (5,-2.5);
\draw[yellow!80!black,thick] (a2p)--(a2p-|a2)--(a2);
%
\draw[red, thick] (0,-3) -- (0,3);
\draw[red, thick] (5,-3) -- (5,3);
\draw[gray, thick] (0,-3) -- (5,-3);
\draw[gray, thick] (0,3) -- (5,3);
\end{tikzpicture}}} = q_{\alpha,\beta} 
\vcenter{\hbox{\begin{tikzpicture}[scale=0.3, every node/.style={scale=0.7},decoration={
    markings,
    mark=at position 0.5 with {\arrow{>}}}
    ]
\coordinate (w0g) at (0,1.5) {};
\coordinate (w0d) at (5,1.5) {};
\coordinate (w1g) at (0,-1.5) {};
\coordinate (w1d) at (5,-1.5) {};
\draw[thick,\hyellow] (w0g)--(0,-3)--(5,-3); -- (w0d);
\draw[gray!40!white, thin] (w1g) -- (w1d) node[pos=0.3, black] (a2) {$\encircled{\beta}$};
\node at (w0g) {$\bullet$};
\node[left] at (w0g) {$\alpha$};
\coordinate (a1p) at (5,0);
\coordinate (a2p) at (5,-2.5);
\draw[yellow!80!black,thick] (a2p)--(a2p-|a2)--(a2);
%
\draw[red, thick] (0,-3) -- (0,3);
\draw[red, thick] (5,-3) -- (5,3);
\draw[gray, thick] (0,-3) -- (5,-3);
\draw[gray, thick] (0,3) -- (5,3);
\draw[\hyellow,thick] (w0g)--(0,-3)--(5,-3);
\end{tikzpicture}}}
\]
we obtain:
\begin{align*} 
r\left( \vcenter{\hbox{\begin{tikzpicture}[scale=0.3, every node/.style={scale=0.7},decoration={
    markings,
    mark=at position 0.5 with {\arrow{>}}}
    ]
\coordinate (w0g) at (-5,1.5) {};
\coordinate (w0d) at (5,1.5) {};
\coordinate (w1g) at (-5,-1.5) {};
\coordinate (w1d) at (5,-1.5) {};
\draw[gray!40!white, thin] (w0g) -- (w0d) node[pos=0.7, black] (a1) {$\encircled{\alpha}$};
\draw[gray!40!white, thin] (w1g) -- (w1d) node[pos=0.3, black] (a2) {$\encircled{\beta}$};
\coordinate (a1p) at (5,0);
\coordinate (a2p) at (5,-2.5);
\draw[yellow!80!black,thick] (a1p)--(a1p-|a1)--(a1);
\draw[yellow!80!black,thick] (a2p)--(a2p-|a2)--(a2);
\draw[red, thick] (-5,-3) -- (-5,3);
\draw[red, thick] (5,-3) -- (5,3);
\draw[gray, thick] (-5,-3) -- (5,-3);
\draw[gray, thick] (-5,3) -- (5,3);
\end{tikzpicture}}}\right)
&
=
\vcenter{\hbox{\begin{tikzpicture}[scale=0.3, every node/.style={scale=0.7},decoration={
    markings,
    mark=at position 0.5 with {\arrow{>}}}
    ]
\coordinate (w0g) at (-5,1.5) {};
\coordinate (w0d) at (0,1.5) {};
\coordinate (w1g) at (-5,-1.5) {};
\coordinate (w1d) at (0,-1.5) {};
\draw[gray!40!white, thin] (w0g) -- (w0d) node[pos=0.7, black] (a1) {$\encircled{\alpha}$};
\draw[gray!40!white, thin] (w1g) -- (w1d) node[pos=0.3, black] (a2) {$\encircled{\beta}$};
\coordinate (a1p) at (0,0);
\coordinate (a2p) at (0,-2.5);
\draw[yellow!80!black,thick] (a1p)--(a1p-|a1)--(a1);
\draw[yellow!80!black,thick] (a2p)--(a2p-|a2)--(a2);
\draw[red, thick] (-5,-3) -- (-5,3);
\draw[red, thick] (0,-3) -- (0,3);
\draw[gray, thick] (-5,-3) -- (0,-3);
\draw[gray, thick] (-5,3) -- (0,3);
\end{tikzpicture}}}
\otimes 
\vcenter{\hbox{\begin{tikzpicture}[scale=0.3, every node/.style={scale=0.7},decoration={
    markings,
    mark=at position 0.5 with {\arrow{>}}}
    ]
\coordinate (w0g) at (0,1.5) {};
\coordinate (w0d) at (5,1.5) {};
\coordinate (w1g) at (0,-1.5) {};
\coordinate (w1d) at (5,-1.5) {};
\coordinate (a1p) at (5,0);
\coordinate (a2p) at (5,-2.5);
%
\draw[red, thick] (0,-3) -- (0,3);
\draw[red, thick] (5,-3) -- (5,3);
\draw[gray, thick] (0,-3) -- (5,-3);
\draw[gray, thick] (0,3) -- (5,3);
\end{tikzpicture}}}
+ 
\vcenter{\hbox{\begin{tikzpicture}[scale=0.3, every node/.style={scale=0.7},decoration={
    markings,
    mark=at position 0.5 with {\arrow{>}}}
    ]
\coordinate (w0g) at (-5,1.5) {};
\coordinate (w0d) at (5,1.5) {};
\coordinate (w1g) at (-5,-1.5) {};
\coordinate (w1d) at (0,-1.5) {};
\draw[gray!40!white, thin] (w1g) -- (w1d) node[pos=0.3, black] (a2) {$\encircled{\beta}$};
\coordinate (a1p) at (5,0);
\coordinate (a2p) at (0,-2.5);
\draw[yellow!80!black,thick] (a2p)--(a2p-|a2)--(a2);
\draw[red, thick] (-5,-3) -- (-5,3);
\draw[red, thick] (0,-3) -- (0,3);
\draw[gray, thick] (-5,-3) -- (0,-3);
\draw[gray, thick] (-5,3) -- (0,3);
\end{tikzpicture}}}
\otimes 
\vcenter{\hbox{\begin{tikzpicture}[scale=0.3, every node/.style={scale=0.7},decoration={
    markings,
    mark=at position 0.5 with {\arrow{>}}}
    ]
\coordinate (w0g) at (0,1.5) {};
\coordinate (w0d) at (5,1.5) {};
\coordinate (w1g) at (0,-1.5) {};
\coordinate (w1d) at (5,-1.5) {};
\draw[gray!40!white, thin] (w0g) -- (w0d) node[pos=0.7, black] (a1) {$\encircled{\alpha}$};
\coordinate (a1p) at (5,0);
\coordinate (a2p) at (5,-2.5);
\draw[yellow!80!black,thick] (a1p)--(a1p-|a1)--(a1);
%
\draw[red, thick] (0,-3) -- (0,3);
\draw[red, thick] (5,-3) -- (5,3);
\draw[gray, thick] (0,-3) -- (5,-3);
\draw[gray, thick] (0,3) -- (5,3);
\end{tikzpicture}}} \\
& + 
q_{\alpha,\beta} \vcenter{\hbox{\begin{tikzpicture}[scale=0.3, every node/.style={scale=0.7},decoration={
    markings,
    mark=at position 0.5 with {\arrow{>}}}
    ]
\coordinate (w0g) at (-5,1.5) {};
\coordinate (w0d) at (0,1.5) {};
\coordinate (w1g) at (-5,-1.5) {};
\coordinate (w1d) at (0,-1.5) {};
\draw[gray!40!white, thin] (w0g) -- (w0d) node[pos=0.7, black] (a1) {$\encircled{\alpha}$};
\coordinate (a1p) at (0,0);
\coordinate (a2p) at (0,-2.5);
\draw[yellow!80!black,thick] (a1p)--(a1p-|a1)--(a1);
\draw[red, thick] (-5,-3) -- (-5,3);
\draw[red, thick] (0,-3) -- (0,3);
\draw[gray, thick] (-5,-3) -- (0,-3);
\draw[gray, thick] (-5,3) -- (0,3);
\end{tikzpicture}}}
\otimes 
\vcenter{\hbox{\begin{tikzpicture}[scale=0.3, every node/.style={scale=0.7},decoration={
    markings,
    mark=at position 0.5 with {\arrow{>}}}
    ]
\coordinate (w0g) at (0,1.5) {};
\coordinate (w0d) at (5,1.5) {};
\coordinate (w1g) at (0,-1.5) {};
\coordinate (w1d) at (5,-1.5) {};
\draw[gray!40!white, thin] (w1g) -- (w1d) node[pos=0.3, black] (a2) {$\encircled{\beta}$};
\coordinate (a1p) at (5,0);
\coordinate (a2p) at (5,-2.5);
\draw[yellow!80!black,thick] (a2p)--(a2p-|a2)--(a2);
%
\draw[red, thick] (0,-3) -- (0,3);
\draw[red, thick] (5,-3) -- (5,3);
\draw[gray, thick] (0,-3) -- (5,-3);
\draw[gray, thick] (0,3) -- (5,3);
\end{tikzpicture}}}
+ 
\vcenter{\hbox{\begin{tikzpicture}[scale=0.3, every node/.style={scale=0.7},decoration={
    markings,
    mark=at position 0.5 with {\arrow{>}}}
    ]
\coordinate (w0g) at (-5,1.5) {};
\coordinate (w0d) at (0,1.5) {};
\coordinate (w1g) at (-5,-1.5) {};
\coordinate (w1d) at (0,-1.5) {};
\draw[red, thick] (-5,-3) -- (-5,3);
\draw[red, thick] (0,-3) -- (0,3);
\draw[gray, thick] (-5,-3) -- (0,-3);
\draw[gray, thick] (-5,3) -- (0,3);
\end{tikzpicture}}}
\otimes 
\vcenter{\hbox{\begin{tikzpicture}[scale=0.3, every node/.style={scale=0.7},decoration={
    markings,
    mark=at position 0.5 with {\arrow{>}}}
    ]
\coordinate (w0g) at (0,1.5) {};
\coordinate (w0d) at (5,1.5) {};
\coordinate (w1g) at (0,-1.5) {};
\coordinate (w1d) at (5,-1.5) {};
\draw[gray!40!white, thin] (w0g) -- (w0d) node[pos=0.7, black] (a1) {$\encircled{\alpha}$};
\draw[gray!40!white, thin] (w1g) -- (w1d) node[pos=0.3, black] (a2) {$\encircled{\beta}$};
\coordinate (a1p) at (5,0);
\coordinate (a2p) at (5,-2.5);
\draw[yellow!80!black,thick] (a1p)--(a1p-|a1)--(a1);
\draw[yellow!80!black,thick] (a2p)--(a2p-|a2)--(a2);
%
\draw[red, thick] (0,-3) -- (0,3);
\draw[red, thick] (5,-3) -- (5,3);
\draw[gray, thick] (0,-3) -- (5,-3);
\draw[gray, thick] (0,3) -- (5,3);
\end{tikzpicture}}}.
\end{align*}
The algorithm that the example illustrates well defines a morphism:
\[
r: \CH_c \to \bigoplus_{c_1+ c_2=c} \CH_{c_1} \otimes \CH_{c_2}. 
\]
It is well defined at homology as it is invariant of isotopies of diagrams thanks to the summation.

More formally, the algorithm described is the composition:
\[
\CH_c \xrightarrow{\diag} \Hnot_c(\Conf_c \times \Conf_c, S\times \Conf_c \cup \Conf_c \times S; \Laurent_c \times \Laurent_c) \xrightarrow{\AM} \bigoplus_{c_1+c_2=c} \CH_{c_1} \otimes \CH_{c_2}. 
\]
where $\diag$ is the diagonal map, and $\AM$ is the \textit{Alexander--Whitney} chain map (see an easy definition in \cite{B}). Notice that $\AM$ is an isomorphism whenever modules are torsion free, as the inverse of the Künneth isomorphism, see \cite[Th.~1.6]{G} for a twisted version. For a relative version of the isomorphism we refer the reader to the seminal paper of {\em Eilenberg--Zilber} \cite{EZ}. The target of the map $\AM$ is actually:
\[
\bigoplus_{c_1 + c_2 = c} \Hnot(Y_{c_1}^{m_{c_2}}, Y_{c_1}^{m_{c_2}+1}) \otimes \Hnot(Y_{c_2}^{m_{c_1}},Y_{c_2}^{m_{c_1}+1})
\]
and it is followed by the iteration of $\del_{x_0}$ that passes from $r'$ to $r$.

Just like \cite[Sec.~1.2]{LusztigBook}, we endow $\CH \otimes \CH$ with the following algebra structure:
\[
(x_1\otimes x_2)(x'_1\otimes x'_2) = q^{\frac{(c_2,c_1')}{2}} x_1x_1' \otimes x_2x_2',
\]
where $x'_1 \in \CH_{c_1'}$ and $x_2 \in \CH_{c_2}$, and $(c_2,c_1')$ is the inner product between roots, extended to linear combinations. 
\begin{prop}\label{P:recovering_r}
The map $r$ is the unique algebra homomorphism that satisfies:
\[
r(\CF^{[1]}_{\alpha}) = \CF^{[1]}_{\alpha} \otimes 1 + 1 \otimes \CF^{[1]}_{\alpha},
\]
for any generator $\CF^{[1]}_{\alpha}$. It is hence the homomorphism $r$ on the corresponding free algebra $'f$ from \cite[Sec.~1.2]{LusztigBook}. 
\end{prop}
\begin{proof}
The image by $r$ of generators is straightforward, we just need to check that it is an algebra homomorphism. Let $\mathcal{D}_c \in \CH_c$ be given by a diagram. Then:
\[
\mathcal{D} \CF^{[1]}_{\alpha} = 
\vcenter{\hbox{\begin{tikzpicture}[scale=0.4, every node/.style={scale=0.75},decoration={
    markings,
    mark=at position 0.5 with {\arrow{>}}}
    ]
\coordinate (w0g) at (-5,1.5) {};
\coordinate (w0d) at (5,1.5) {};
\coordinate (w1g) at (-5,-2) {};
\coordinate (w1d) at (5,-2) {};
\draw[fill=gray!40!white] (-5,0) rectangle ++(10,3);
\node at (0,1.5) {$\mathcal{D}_c$};
\draw[gray!40!white, thin] (w1g) -- (w1d) coordinate[pos=0.6] (a2) node[pos=0.5, black] {$\encircled{\alpha}$};
\coordinate (a1p) at (5,0);
\coordinate (a2p) at (5,-1.5);
\draw[yellow!80!black,thick] (a2p)--(a2p-|a2)--(a2);
\draw[red, thick] (-5,-3) -- (-5,3);
\draw[red, thick] (5,-3) -- (5,3);
\draw[gray, thick] (-5,-3) -- (5,-3);
\draw[gray, thick] (-5,3) -- (5,3);
\end{tikzpicture}}}
\]
Fix a given cut of $\mathcal{D}_c$ represented in the following picture:
\[
\mathcal{D} \CF^{[1]}_{\alpha} = 
\vcenter{\hbox{\begin{tikzpicture}[scale=0.4, every node/.style={scale=0.75},decoration={
    markings,
    mark=at position 0.5 with {\arrow{>}}}
    ]
\coordinate (w0g) at (-5,1.5) {};
\coordinate (w0d) at (5,1.5) {};
\coordinate (w1g) at (-5,-2) {};
\coordinate (w1d) at (5,-2) {};
\draw[fill=gray!40!white] (-5,0) rectangle ++(10,3);
\node at (-2.5,1.5) {$\mathcal{D}_{c_1}$};
\node at (2.5,1.5) {$\mathcal{D}_{c_2}$};
\draw[gray!40!white, thin] (w1g) -- (w1d) coordinate[pos=0.5] (a2) node[pos=0.5, black] {$\encircled{\alpha}$};
\coordinate (a1p) at (5,0);
\coordinate (a2p) at (5,-1.5);
\draw[yellow!80!black,thick] (a2p)--(a2p-|a2)--(a2);
\draw[thick,orange] (0,4)--(0,-4);
\node[right,orange] at (0,4) {$d$};
\draw[red, thick] (-5,-3) -- (-5,3);
\draw[red, thick] (5,-3) -- (5,3);
\draw[gray, thick] (-5,-3) -- (5,-3);
\draw[gray, thick] (-5,3) -- (5,3);
\end{tikzpicture}}}
\]
so that the pearl $\alpha$ can either go to the left or to the right of $d$. It gives the two terms:
\[
\vcenter{\hbox{\begin{tikzpicture}[scale=0.4, every node/.style={scale=0.75},decoration={
    markings,
    mark=at position 0.5 with {\arrow{>}}}
    ]
\coordinate (w0g) at (-5,1.5) {};
\coordinate (w0d) at (5,1.5) {};
\coordinate (w1g) at (-5,-2) {};
\coordinate (w1d) at (0,-2) {};
\draw[fill=gray!40!white] (-5,0) rectangle ++(5,3);
\node at (-2.5,1.5) {$\mathcal{D}_{c_1}$};
\draw[gray!40!white, thin] (w1g) -- (w1d) coordinate[pos=0.5] (a2) node[pos=0.5, black] {$\encircled{\alpha}$};
\coordinate (a1p) at (5,0);
\coordinate (a2p) at (0,-1.5);
\draw[yellow!80!black,thick] (a2p)--(a2p-|a2)--(a2);
\draw[red, thick] (-5,-3) -- (-5,3);
\draw[red, thick] (0,-3) -- (0,3);
\draw[gray, thick] (-5,-3) -- (0,-3);
\draw[gray, thick] (-5,3) -- (0,3);
\end{tikzpicture}}}
\otimes
\vcenter{\hbox{\begin{tikzpicture}[scale=0.4, every node/.style={scale=0.75},decoration={
    markings,
    mark=at position 0.5 with {\arrow{>}}}
    ]
\coordinate (w0g) at (-5,1.5) {};
\coordinate (w0d) at (5,1.5) {};
\coordinate (w1g) at (0,-2) {};
\coordinate (w1d) at (5,-2) {};
\draw[fill=gray!40!white] (0,0) rectangle ++(5,3);
\node at (2.5,1.5) {$\mathcal{D}_{c_2}$};
\draw[\hyellow,thick] (w1g) -- (w1d);
\coordinate (a1p) at (5,0);
\coordinate (a2p) at (5,-1.5);
\draw[red, thick] (0,-3) -- (0,3);
\draw[red, thick] (5,-3) -- (5,3);
\draw[gray, thick] (0,-3) -- (5,-3);
\draw[gray, thick] (0,3) -- (5,3);
\end{tikzpicture}}} + 
\vcenter{\hbox{\begin{tikzpicture}[scale=0.4, every node/.style={scale=0.75},decoration={
    markings,
    mark=at position 0.5 with {\arrow{>}}}
    ]
\coordinate (w0g) at (-5,1.5) {};
\coordinate (w0d) at (5,1.5) {};
\coordinate (w1g) at (-5,-2) {};
\coordinate (w1d) at (0,-2) {};
\draw[fill=gray!40!white] (-5,0) rectangle ++(5,3);
\node at (-2.5,1.5) {$\mathcal{D}_{c_1}$};
\coordinate (a1p) at (5,0);
\coordinate (a2p) at (0,-1.5);
\draw[red, thick] (-5,-3) -- (-5,3);
\draw[red, thick] (0,-3) -- (0,3);
\draw[gray, thick] (-5,-3) -- (0,-3);
\draw[gray, thick] (-5,3) -- (0,3);
\end{tikzpicture}}}
\otimes
\vcenter{\hbox{\begin{tikzpicture}[scale=0.4, every node/.style={scale=0.75},decoration={
    markings,
    mark=at position 0.5 with {\arrow{>}}}
    ]
\coordinate (w0g) at (-5,1.5) {};
\coordinate (w0d) at (5,1.5) {};
\coordinate (w1g) at (0,-2) {};
\coordinate (w1d) at (5,-2) {};
\draw[fill=gray!40!white] (0,0) rectangle ++(5,3);
\node at (2.5,1.5) {$\mathcal{D}_{c_2}$};
\draw[gray!40!white, thin] (w1g) -- (w1d) coordinate[pos=0.5] (a2) node[pos=0.5, black] {$\encircled{\alpha}$};
\coordinate (a1p) at (5,0);
\coordinate (a2p) at (5,-1.5);
\draw[yellow!80!black,thick] (a2p)--(a2p-|a2)--(a2);
\draw[red, thick] (0,-3) -- (0,3);
\draw[red, thick] (5,-3) -- (5,3);
\draw[gray, thick] (0,-3) -- (5,-3);
\draw[gray, thick] (0,3) -- (5,3);
\end{tikzpicture}}}
\]
Now for applying $\del_{x_0}$ successively, nothing to do on the left term while on the right one should remember that $m_{c_1}$ points are fixed in $\partial^-D$ so that their handles need swapping with that of the indexed $\alpha$ pearl before reaching $x_0$. The iteration of $\del_{x_0}$ hence sends the second term to:
\[
\prod_{\beta \in \Pi} q_{\alpha,\beta}^{c_1(\beta)} \vcenter{\hbox{\begin{tikzpicture}[scale=0.4, every node/.style={scale=0.75},decoration={
    markings,
    mark=at position 0.5 with {\arrow{>}}}
    ]
\coordinate (w0g) at (-5,1.5) {};
\coordinate (w0d) at (5,1.5) {};
\coordinate (w1g) at (-5,-2) {};
\coordinate (w1d) at (0,-2) {};
\draw[fill=gray!40!white] (-5,0) rectangle ++(5,3);
\node at (-2.5,1.5) {$\mathcal{D}_{c_1}$};
\coordinate (a1p) at (5,0);
\coordinate (a2p) at (0,-1.5);
\draw[red, thick] (-5,-3) -- (-5,3);
\draw[red, thick] (0,-3) -- (0,3);
\draw[gray, thick] (-5,-3) -- (0,-3);
\draw[gray, thick] (-5,3) -- (0,3);
\end{tikzpicture}}}
\otimes
\vcenter{\hbox{\begin{tikzpicture}[scale=0.4, every node/.style={scale=0.75},decoration={
    markings,
    mark=at position 0.5 with {\arrow{>}}}
    ]
\coordinate (w0g) at (-5,1.5) {};
\coordinate (w0d) at (5,1.5) {};
\coordinate (w1g) at (0,-2) {};
\coordinate (w1d) at (5,-2) {};
\draw[fill=gray!40!white] (0,0) rectangle ++(5,3);
\node at (2.5,1.5) {$\mathcal{D}_{c_2}$};
\draw[gray!40!white, thin] (w1g) -- (w1d) coordinate[pos=0.5] (a2) node[pos=0.5, black] {$\encircled{\alpha}$};
\coordinate (a1p) at (5,0);
\coordinate (a2p) at (5,-1.5);
\draw[yellow!80!black,thick] (a2p)--(a2p-|a2)--(a2);
\draw[red, thick] (0,-3) -- (0,3);
\draw[red, thick] (5,-3) -- (5,3);
\draw[gray, thick] (0,-3) -- (5,-3);
\draw[gray, thick] (0,3) -- (5,3);
\end{tikzpicture}}}.
\]
It proves that any term in the sum $r(\mathcal{D})$ is multiplied by $\CF^{[1]}_{\alpha} \otimes 1 + 1 \otimes \CF^{[1]}_{\alpha}$ in $r(\mathcal{D}\CF^{[1]}_{\alpha})$ since the coefficient showing up is precisely the one prescribed by the algebra structure endowing $\CH\otimes \CH$. It proves $$r(\mathcal{D}\CF^{[1]}_{\alpha})= r(\mathcal{D})r(\CF^{[1]}_{\alpha})$$ and the claim by an immediate recursion on any word in the generators of the free algebra. 
\end{proof}

Now we renormalize generators $\CF_\alpha^{(1)}$ to $\CF_\alpha^{1}$ so that:
\[
\CF_\alpha^{(1)} = \frac{1}{(1-q_\alpha^{-1})}\CF_\alpha^{1}
\]
and we renormalize $\iota$ also to the unique algebra morphism that sends $\CF_\alpha^{[1]}$ to $\CF_\alpha^1$. 

\begin{prop}\label{P:recovering_coproduct}
The bilinear form $(\cdot,\cdot)$ is the unique one such that $(1,1)=1$, and:
\begin{itemize}
\item[(a)] $(\CF_\alpha^{[1]},\CF_\beta^{[1]}) = \delta_{\alpha,\beta} \frac{1}{(1-q_\alpha^{-1})}$ for all $\alpha,\beta\in\Pi$.
\item[(b)] $(x,y'y'')=(r(x),y'\otimes y'')$ for all $x,y,y'' \in \CH$. 
\item[(c)] $(x x',y'') = (x\otimes x',r(y''))$ for all $x, x',y'' \in \CH$. 
\end{itemize}
where the bilinear form on tensor products is given by:
\[
(x_1\otimes x_2, x'_1\otimes x'_2) = (x_1\otimes x_1')(x_2\otimes x_2'). 
\]
\end{prop}
\begin{proof}
We will prove first that for any $x,y',y'' \in\CH$ we have:
\[
(x,y'y'') = (r(x),y' \otimes y'').
\]
We can assume $x\in\CH_c$ is a basis vector and since a $\frac{\pi}{2}$-rotation shall be applied to it before being paired, that it has the following form after the rotation:
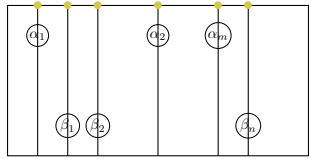
\[
\vcenter{\hbox{\begin{tikzpicture}[scale=0.2, every node/.style={scale=0.75},decoration={
    markings,
    mark=at position 0.5 with {\arrow{>}}}
    ]
    \draw (-10,-5) rectangle ++(20,10);
    \coordinate (lh1) at (-8,5);
    \coordinate (lb1) at (-8,-5);
    \coordinate (lh2) at (-6,5);
    \coordinate (lb2) at (-6,-5);
    \coordinate (lhm) at (-4,5);
    \coordinate (lbm) at (-4,-5);
    \coordinate (lhmm) at (0,5);
    \coordinate (lbmm) at (0,-5);
    \coordinate (lhn1) at (4,5);
    \coordinate (lbn1) at (4,-5);
    \coordinate (lhn) at (6,5);
    \coordinate (lbn) at (6,-5);   
    \coordinate (nag) at (-10,2);
    \coordinate (nad) at (10,2);
    \coordinate (nbg) at (-10,-4);
    \coordinate (nbd) at (10,-4);    
    \draw (lh1)--(lb1) node[pos=0.2] {$\encircled{\alpha_1}$} node[pos=0,\hyellow,thick] {$\bullet$};
    \draw (lh2)--(lb2) node[pos=0.8] {$\encircled{\beta_1}$} node[pos=0,\hyellow,thick] {$\bullet$};
    \draw (lhm)--(lbm) node[pos=0.8] {$\encircled{\beta_2}$} node[pos=0,\hyellow,thick] {$\bullet$};
    \draw (lhmm)--(lbmm) node[pos=0.2] {$\encircled{\alpha_2}$} node[pos=0,\hyellow,thick] {$\bullet$};
    \draw (lhn1)--(lbn1) node[pos=0.2] {$\encircled{\alpha_m}$} node[pos=0,\hyellow,thick] {$\bullet$};
    \draw (lhn)--(lbn) node[pos=0.8] {$\encircled{\beta_n}$} node[pos=0,\hyellow,thick] {$\bullet$};
\end{tikzpicture}}}
\]
where we have indicated dots fixed on the higher part of $\partial D$ for their handles, before the rotation this was showing how the class is related to a vertically aligned configuration. Notice that $c=\sum \alpha_i + \beta_i$. We can assume $y' y''$ has the following form:
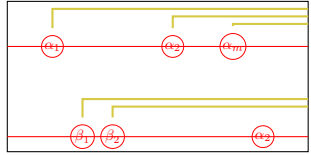
\[
\vcenter{\hbox{\begin{tikzpicture}[scale=0.2, every node/.style={scale=0.75},decoration={
    markings,
    mark=at position 0.5 with {\arrow{>}}}
    ]
    \draw (-10,-5) rectangle ++(20,10);
    \coordinate (lh1) at (-8,5);
    \coordinate (lb1) at (-8,-5);
    \coordinate (lh2) at (-6,5);
    \coordinate (lb2) at (-6,-5);
    \coordinate (lhm) at (-4,5);
    \coordinate (lbm) at (-4,-5);
    \coordinate (lhmm) at (0,5);
    \coordinate (lbmm) at (0,-5);
    \coordinate (lhn1) at (4,5);
    \coordinate (lbn1) at (4,-5);
    \coordinate (lhn) at (6,5);
    \coordinate (lbn) at (6,-5);   
    \coordinate (nag) at (-10,2);
    \coordinate (nad) at (10,2);
    \coordinate (nbg) at (-10,-4);
    \coordinate (nbd) at (10,-4);    
    \draw[red] (nag)--(nad) node[pos=0.15] (a1) {$\encircled {\alpha_1}$}  node[pos=0.55] (a2) {$\encircled {\alpha_2}$} node[pos=0.75] (am) {$\encircled {\alpha_m}$}; 
    \draw[red] (nbg)--(nbd) node[pos=0.25] (b1) {$\encircled{\beta_1}$} node[pos=0.35] (b2) {$\encircled {\beta_2}$} node[pos=0.85] (bn) {$\encircled {\alpha_2}$};
    \coordinate (a1p) at (10,4.5);
    \coordinate (a2p) at (10,4);
    \coordinate (amp) at (10,3.5);
    \draw[\hyellow,thick] (a1p)--(a1p-|a1)--(a1);
    \draw[\hyellow,thick] (a2p)--(a2p-|a2)--(a2);
    \draw[\hyellow,thick] (amp)--(amp-|am)--(am);
    \coordinate (b1p) at (10,-1.5);
    \coordinate (b2p) at (10,-2);
    \coordinate (bnp) at (10,-2.5);
    \draw[\hyellow,thick] (b1p)--(b1p-|b1)--(b1);
    \draw[\hyellow,thick] (b2p)--(b2p-|b2)--(b2);
\end{tikzpicture}}}
\]
This is because $y'$ and $y''$ are first sent by $\iota$ to the Borel--Moore homology and hence could be expressed in the basis of pearl necklaces. The above diagram is indeed the product (stacking) of two such diagrams, and hence we look at a term in their linear expansions in bases elements. Now one notes that there are precisely the same pearls namely that while reading the diagram from left to right, one reads exactly the same list of pearls as $x$. This arrangement of lists should exist for the pairing not to be zero. We refer the reader to the definition of the pairing in the Appendix~\ref{A:the_pairing} and to note that contributions to the pairing exist for intersection configurations namely when there is a configuration standing on both submanifolds. For the latter to be possible, one must be able to read from left to right exactly the same list of pearls as that of $x$, so to be able to put one point decorated with the appropriate pearl on each vertical black necklaces of $x$. Hence elements $y'y''$ of the above form are the only one contributing to the pairing. To pair them we need to put them on the same diagram:
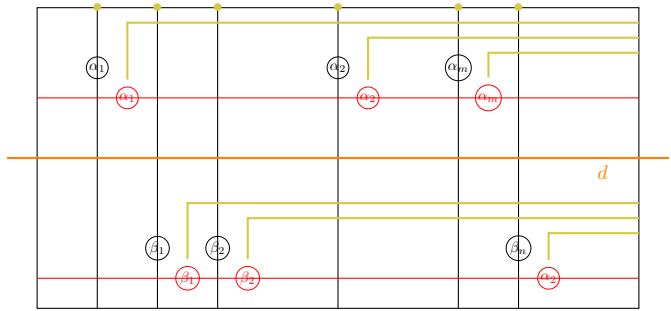
\[
\vcenter{\hbox{\begin{tikzpicture}[scale=0.4, every node/.style={scale=0.75},decoration={
    markings,
    mark=at position 0.5 with {\arrow{>}}}
    ]
    \draw (-10,-5) rectangle ++(20,10);
    \coordinate (lh1) at (-8,5);
    \coordinate (lb1) at (-8,-5);
    \coordinate (lh2) at (-6,5);
    \coordinate (lb2) at (-6,-5);
    \coordinate (lhm) at (-4,5);
    \coordinate (lbm) at (-4,-5);
    \coordinate (lhmm) at (0,5);
    \coordinate (lbmm) at (0,-5);
    \coordinate (lhn1) at (4,5);
    \coordinate (lbn1) at (4,-5);
    \coordinate (lhn) at (6,5);
    \coordinate (lbn) at (6,-5);   
    \coordinate (nag) at (-10,2);
    \coordinate (nad) at (10,2);
    \coordinate (nbg) at (-10,-4);
    \coordinate (nbd) at (10,-4);    
    \draw (lh1)--(lb1) node[pos=0.2] {$\encircled{\alpha_1}$} node[pos=0,\hyellow,thick] {$\bullet$};
    \draw (lh2)--(lb2) node[pos=0.8] {$\encircled{\beta_1}$} node[pos=0,\hyellow,thick] {$\bullet$};
    \draw (lhm)--(lbm) node[pos=0.8] {$\encircled{\beta_2}$} node[pos=0,\hyellow,thick] {$\bullet$};
    \draw (lhmm)--(lbmm) node[pos=0.2] {$\encircled{\alpha_2}$} node[pos=0,\hyellow,thick] {$\bullet$};
    \draw (lhn1)--(lbn1) node[pos=0.2] {$\encircled{\alpha_m}$} node[pos=0,\hyellow,thick] {$\bullet$};
    \draw (lhn)--(lbn) node[pos=0.8] {$\encircled{\beta_n}$} node[pos=0,\hyellow,thick] {$\bullet$};
    \draw[red] (nag)--(nad) node[pos=0.15] (a1) {$\encircled {\alpha_1}$}  node[pos=0.55] (a2) {$\encircled {\alpha_2}$} node[pos=0.75] (am) {$\encircled {\alpha_m}$}; 
    \draw[red] (nbg)--(nbd) node[pos=0.25] (b1) {$\encircled{\beta_1}$} node[pos=0.35] (b2) {$\encircled {\beta_2}$} node[pos=0.85] (bn) {$\encircled {\alpha_2}$};
    \coordinate (a1p) at (10,4.5);
    \coordinate (a2p) at (10,4);
    \coordinate (amp) at (10,3.5);
    \draw[\hyellow,thick] (a1p)--(a1p-|a1)--(a1);
    \draw[\hyellow,thick] (a2p)--(a2p-|a2)--(a2);
    \draw[\hyellow,thick] (amp)--(amp-|am)--(am);
    \coordinate (b1p) at (10,-1.5);
    \coordinate (b2p) at (10,-2);
    \coordinate (bnp) at (10,-2.5);
    \draw[\hyellow,thick] (b1p)--(b1p-|b1)--(b1);
    \draw[\hyellow,thick] (b2p)--(b2p-|b2)--(b2);
    \draw[\hyellow,thick] (bnp)--(bnp-|bn)--(bn);
    \draw[orange,thick] (-11,0)--(11,0) node[pos=0.9,below] {$d$};
\end{tikzpicture}}}
\]
and there is one clear intersection configuration that contributes the pairing indicated by pairs of identical pearls that have been put voluntarily close to each other. This intersection configuration contributes to the pairing by a monomial given by the image in the local system $\rho_c$ of a loop $L$ in $\Conf_c$ which is the following composition of paths:
\begin{itemize}
\item $H$ that starts from a vertically aligned configuration and reaches the pearls following the handles of $y'y''$,
\item $X$ that starts from the pearl and reaches the dots following the support of $x$,
\item $W$ that starts from the dots and reaches the initial vertically aligned configuration by traveling along $\partial D$ clockwise. 
\end{itemize}
There is a preferred cut along the axis $d$ associated with the picture. Now $(\cut_d,y'\otimes y'')$ is displayed in the following picture:
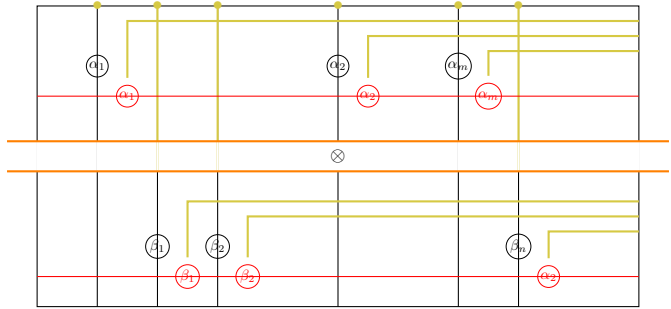
\[
\vcenter{\hbox{\begin{tikzpicture}[scale=0.4, every node/.style={scale=0.75},decoration={
    markings,
    mark=at position 0.5 with {\arrow{>}}}
    ]
    \draw (-10,-5) rectangle ++(20,10);
    \coordinate (lh1) at (-8,5);
    \coordinate (lb1) at (-8,-5);
    \coordinate (lh2) at (-6,5);
    \coordinate (lb2) at (-6,-5);
    \coordinate (lhm) at (-4,5);
    \coordinate (lbm) at (-4,-5);
    \coordinate (lhmm) at (0,5);
    \coordinate (lbmm) at (0,-5);
    \coordinate (lhn1) at (4,5);
    \coordinate (lbn1) at (4,-5);
    \coordinate (lhn) at (6,5);
    \coordinate (lbn) at (6,-5);   
    \coordinate (nag) at (-10,2);
    \coordinate (nad) at (10,2);
    \coordinate (nbg) at (-10,-4);
    \coordinate (nbd) at (10,-4);    
    \draw (lh1)--(lb1) node[pos=0.2] {$\encircled{\alpha_1}$} node[pos=0,\hyellow,thick] {$\bullet$};
    \draw (lh2)--(lb2) node[pos=0.8] {$\encircled{\beta_1}$} node[pos=0,\hyellow,thick] {$\bullet$};
    \draw (lhm)--(lbm) node[pos=0.8] {$\encircled{\beta_2}$} node[pos=0,\hyellow,thick] {$\bullet$};
    \draw (lhmm)--(lbmm) node[pos=0.2] {$\encircled{\alpha_2}$} node[pos=0,\hyellow,thick] {$\bullet$};
    \draw (lhn1)--(lbn1) node[pos=0.2] {$\encircled{\alpha_m}$} node[pos=0,\hyellow,thick] {$\bullet$};
    \draw (lhn)--(lbn) node[pos=0.8] {$\encircled{\beta_n}$} node[pos=0,\hyellow,thick] {$\bullet$};
    \draw[white,thick] (-10,-0.5)--(-10,0.5);
    \draw[white,thick] (10,-0.5)--(10,0.5);
    \draw[\hyellow,thick] (-6,-0.5)--(-6,5);
    \draw[white,thick] (-6,-0.5)--(-6,0.5);
    \draw[\hyellow,thick] (-4,-0.5)--(-4,5);
    \draw[white,thick] (-4,-0.5)--(-4,0.5);
    \draw[\hyellow,thick] (6,-0.5)--(6,5);
    \draw[white,thick] (6,-0.5)--(6,0.5);
    \draw[white,thick] (-8,-0.5)--(-8,0.5);
    \draw[white,thick] (0,-0.5)--(0,0.5);
    \draw[white,thick] (4,-0.5)--(4,0.5);
    \draw[red] (nag)--(nad) node[pos=0.15] (a1) {$\encircled {\alpha_1}$}  node[pos=0.55] (a2) {$\encircled {\alpha_2}$} node[pos=0.75] (am) {$\encircled {\alpha_m}$}; 
    \draw[red] (nbg)--(nbd) node[pos=0.25] (b1) {$\encircled{\beta_1}$} node[pos=0.35] (b2) {$\encircled {\beta_2}$} node[pos=0.85] (bn) {$\encircled {\alpha_2}$};
    \coordinate (a1p) at (10,4.5);
    \coordinate (a2p) at (10,4);
    \coordinate (amp) at (10,3.5);
    \draw[\hyellow,thick] (a1p)--(a1p-|a1)--(a1);
    \draw[\hyellow,thick] (a2p)--(a2p-|a2)--(a2);
    \draw[\hyellow,thick] (amp)--(amp-|am)--(am);
    \coordinate (b1p) at (10,-1.5);
    \coordinate (b2p) at (10,-2);
    \coordinate (bnp) at (10,-2.5);
    \draw[\hyellow,thick] (b1p)--(b1p-|b1)--(b1);
    \draw[\hyellow,thick] (b2p)--(b2p-|b2)--(b2);
    \draw[\hyellow,thick] (bnp)--(bnp-|bn)--(bn);
    \draw[orange,thick] (-11,0.5)--(11,0.5);
    \draw[orange,thick] (-11,-0.5)--(11,-0.5);
    \node at (0,0) {$\otimes$};
\end{tikzpicture}}}
\]
Now there is a loop $L_\beta$ (resp. $L_\alpha$) which image by the local system $\rho_{c_1}$ ($c_1= \sum \beta_i$) (resp. $\rho_{c_2}$) gives the contribution to $(\cut_d,y'\otimes y'')$ of below's diagram (resp. above diagram). It is the composition of:
\begin{itemize}
\item $H_\beta$ (resp. $H_\alpha$) that follows the handles reaching $\beta$-pearls (resp. $\alpha$-pearls),
\item $X_\beta$ (resp. $X_\alpha$) that reaches dots in the above side of the boundary of the lower (resp. the higher) diagram,
\item $W_\beta$ (resp. $W_\alpha$) that travels along the boundary clockwisely so to reach the initial vertically aligned configuration. 
\end{itemize}
One can stack $L_\beta$ and $L_\alpha$ to give a loop in $\Conf_c$ denoted $L_{\alpha,\beta}$. The image in $\Laurent$ by $\rho_c$ of $L_{\alpha,\beta}$ is clearly the product of both those of $L_\beta$ and $L_\alpha$. Notice that $L$ is the composition of $L_{\alpha,\beta}$ with a loop $L'$ defined by the composition of paths:
\begin{itemize}
\item Below points (resp above ones) starting from the end of $L_\beta$ (resp. $L_\alpha$) traveling counterclockwise along below's boundary (resp. above one) so to reach the dots. 
\item Then belows points reach the above dots following yellow line in the above disk, other one stay fixed,
\item Then they all go back to a vertically aligned configuration in the boundary following the boundary clockwise. 
\end{itemize}
We claim that the image of $L'$ in $\Laurent$ is precisely the coefficient showing up by iteration of $\del_{x_0}$ applied to $\cut_d$. Now it might be that there is another intersection configuration between $x$ and $y'y''$ but it will correspond to another cut. Since we sum over all possible cuts, it proves that $(x,y'y'')=(r(x),y' \otimes y'')$. The proof that $(xx',y'')=(x\otimes x',r(y''))$ works the same. Now (a) is a straightforward computation. It proves the proposition. 
\end{proof}

\begin{theorem}\label{T:Uqgm_is_homological}
The space $\overline{\CH}$ is a $\mathbb{Q}(q)$-subalgebra of $\CHbm$ which is isomorphic to $\Uqgm$. 
\end{theorem}
\begin{proof}
We recall that $\CH$ is the free algebra generated by simple roots from Theorem~\ref{T:free_algebra}, hence it is the algebra $'f$ described in \cite[Sec.~1.2.1]{LusztigBook}. The map $r$ is the algebra morphism described in \cite[Sec.~1.2.2]{LusztigBook} from Prop.~\ref{P:recovering_r} and it satisfies the axiom of \cite[Prop.~1.2.3]{LusztigBook} that are exactly the three from Proposition~\ref{P:recovering_coproduct}. Then $\Uqgm$ is the quotient of $\CH$ by the radical of $(\cdot,\cdot)$ as defined in \cite[1.2.5]{LusztigBook}, where he calls $f$ the algebra $\Uqgm$. That the generators satisfy the quantum Serre relations in $\CH$ is reproved (it was proved in Theorem~\ref{T:homological_version_for_Uqgm}) since it is proved in \cite[Prop.~1.4.3]{LusztigBook} that it holds in $f$. Notice that due to the renormalization of $\iota$ our divided powers are renormalized from those of Lusztig. We refer the reader to \cite{JK,Jules_Verma,JulesMarco} to notice that homological actions for divided powers of $F$'s are systematically renormalized in comparison with those of Lusztig, already in the $\slt$ case and it generalizes here to any $\mathfrak{g}$. Now $(\cdot, \cdot)$ pairs $\CH$ and $\iota(\CH)$ hence $\iota(\CH)$ is isomorphic to $f$ while quotienting by the radical of $(\cdot,\cdot)$ in $\CH$. 
\end{proof}

\subsubsection{Braid group action}\label{S:PBW}

There exists a \textit{braid group} associated with the Lie algebra $\fg$ usually denoted by $\CB_{\fg}$, see \cite[Sec.~6.2.1]{KS} for instance. It turns out that this group acts on $\Uqg$ by algebra automorphisms. The group $\CB_{\fg}$ is typically defined with generators in correspondence with simple roots, thus its action on $\Uqg$ consists in defining an algebra automorphism denoted by $\T_i$ for $i=1,\ldots, l$, see \cite[Theorem~3.1]{Lus90} or \cite[Theorem~22]{KS} for the defining formulas. The fact that these operators satisfy the \textit{braid relations} is due to Lusztig (see \cite[8.15]{Jan}), we mention it for completeness while we won't use it.  Now this action is the one used to find bases of $\Uqgm$ (and by extension to $\Uqg$). More precisely, so to find such bases, the first step is to assign a generator with each positive root (rather than only simple roots, what we did until now). By iteration of the action of $\CB_{\fg}$ on simple root generators, one can construct Poincaré--Birkhoff--Witt bases just like in the classical Lie algebras, the quantum version is due to Lusztig. This section is devoted to define these algebra morphisms, so that Poincaré--Birkhoff--Witt bases would yield bases of $\overline{\CH}$ thanks to Theorem~\ref{T:Uqgm_is_homological}.  

\begin{defn}
For $i\neq j$ chosen in $\lbrace 1 , \ldots, l \rbrace$, we define the following operator:
\[
\T_i \left( \CF_j \right) := q^{-\frac{d_{\alpha_i}k(k-1)}{4}} \qSerre_{\alpha_i, \alpha_j}^{(k)}
\]
where $k=-a_{i,j}$. 
\end{defn}

\begin{prop}\label{P:Ti_on_generators}
We have:
\[
\T_i \left( \CF_j \right) = \sum_{l=0}^k (-1)^{k-l} q_{\alpha_i}^{\frac{l}{2}} \CF_{\alpha_i}^{(l)} \CF_{\alpha_j}^{(1)} \CF_{\alpha_i}^{(k-l)}
\]
\end{prop}
\begin{proof}
We recall that Coro.~\ref{coro_ident_partition_squash} is a summary of two expressions for $\qSerre_{\alpha_i, \alpha_j}^{(k)}$ for any $k$, here we put $k=-a_{i,j}$. We put $\alpha_i=\alpha$ and $\alpha_j=\beta$ and use blue and red colors to represent these simple roots. Then, the proof starts likewise that of Prop.~\ref{prop_QuantumSerre}:
\begin{align*}
q^{\frac{d_{\alpha}k(k-1)}{4}}\T_i \left( \CF_j \right) & = \sum_{l=0}^k (-1)^{k-l}  q_{\alpha, \beta}^{-l} q_{\alpha}^{-l(l-1)/2} 
\vcenter{\hbox{
\begin{tikzpicture}[scale=0.3,every node/.style={scale=0.6},decoration={
    markings,
    mark=at position 0.5 with {\arrow{>}}}]
\draw[thick] (0,0) rectangle ++(12,8);
\draw[red,thick] (0,0)--(0,8);
\draw[red,thick] (12,0)--(12,8);
\draw[red,postaction={decorate}]  (0,4) -- (12,4);
\coordinate (ll) at (12,3);
\coordinate (kll) at (12,7);
\draw[blue,dashed,postaction={decorate}] (0,6) to node[pos=0.7,below] (kl) {$l$} (12,6);
\draw[blue,dashed,postaction={decorate}] (0,2) to node[pos=0.7,below] (l) {$k-l$} (12,2);
\draw[yellow!80!black,double,thick] (l)--(l|-ll)--(ll);
\draw[yellow!80!black,double,thick] (kl)--(kl|-kll)--(kll);
\node[right,yellow!80!black] at (12,4) {$q$};
\node[yellow!80!black] at (12,4) {$\bullet$};
\end{tikzpicture}}} \\
& = \sum_{l=0}^k (-1)^{k-l}  q_{\alpha, \beta}^{-l} q_{\alpha}^{-l(l-1)/2} q^{d_{\alpha} \frac{l(l-1)}{4}} \CF_{\alpha}^{(l)} \CF_{\beta}^{(1)} q^{d_{\alpha} \frac{(k-l)(k-l-1)}{4}} \CF_{\alpha}^{(k-l)}.
\end{align*}
The difference with Prop.~\ref{prop_QuantumSerre} is that now $k=-a_{i,j}$ (rather than $1-a_{i,j}$), thus $q_{\alpha,\beta}^{-l} = q^{d_{\alpha} \frac{kl}{2}}$. Hence:
\begin{align*}
q^{\frac{d_{\alpha}k(k-1)}{4}}\T_i \left( \CF_j \right) & = \sum_{l=0}^k (-1)^{k-l}  q^{d_{\alpha} \frac{kl}{2}} q^{-d_{\alpha}\frac{l(l-1)}{2}} q^{d_{\alpha} \frac{l(l-1)}{4}} q^{d_{\alpha} \frac{(k-l)(k-l-1)}{4}} \CF_{\alpha}^{(l)} \CF_{\beta}^{(1)}  \CF_{\alpha}^{(k-l)} \\
& = \sum_{l=0}^k (-1)^{k-l} q^{d_{\alpha} \left( \frac{kl}{2} + \frac{(k-l)(k-l-1)}{4} - \frac{l(l-1)}{4} \right) } \CF_{\alpha}^{(l)} \CF_{\beta}^{(1)} \CF_{\alpha}^{(k-l)} \\
& = \sum_{l=0}^k (-1)^{k-l} q^{\frac{d_{\alpha}}{2} \left( \frac{k(k-1)}{2}+l \right) } \CF_{\alpha}^{(l)} \CF_{\beta}^{(1)} \CF_{\alpha}^{(k-l)} =  q^{\frac{d_{\alpha}k(k-1)}{4}} \sum_{l=0}^k (-1)^{k-l} q^{\frac{d_{\alpha}l}{2}} \CF_{\alpha}^{(l)} \CF_{\beta}^{(1)} \CF_{\alpha}^{(k-l)} .
\end{align*}
\end{proof}

\begin{prop}\label{P:truncature}
Let $c\in \Coloring_{\Pi}$ and $\alpha_i \in \Pi$, such that $c(\alpha_i) = 0$. We define $$k(i,c) := \sum_{\alpha_j \in \Pi} (-a_{i,j}).$$ 
Notice that it is additive, namely $k(i,c_1+c_2) = k(i,c_1)+k(i,c_2)$. We claim that if $k > k(i,c)$, we have:
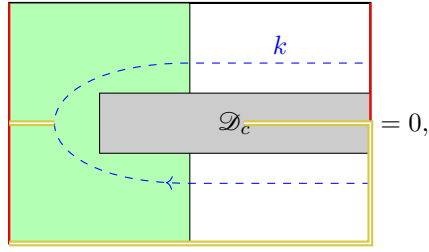
\[
\vcenter{\hbox{
\begin{tikzpicture}[scale=0.4,decoration={
    markings,
    mark=at position 0.3 with {\arrow{>}}}]
\draw[thick] (0,0) rectangle ++(12,8);
\draw[fill = green!30!white] (0,0) rectangle ++(6,8);
\draw[fill = black!20!white] (3,3) rectangle ++(9,2);
\draw[red,thick] (0,0)--(0,8);
\draw[red,thick] (12,0)--(12,8);
\coordinate (kk) at (12,2.5);
\draw[blue,dashed,postaction={decorate}] (12,2) to coordinate[midway] (k) (6,2) .. controls (0,2) and (0,6) .. (6,6) to node[midway,above] {$k$} (12,6);
\node at (7.5,4) {$\mathcal{D}_{c}$};
\draw[yellow!80!black,double,thick] (7.8,4)--(12,4)--(12,0)--(0,0);
\draw[yellow!80!black,double,thick] (1.5,4)--(0,4);
\end{tikzpicture}}} = 0,
\]
for any diagrammatic class $\mathcal{D}_{c}$ in $\overline{\CH}_c \subset \CHbm_c$, where the blue color illustrates $\alpha_i$, and where a handle was represented for $\mathcal{D}_{c}$ so to show how to make it reach a vertically aligned configuration, but it won't serve in the proof. 
\end{prop}
\begin{proof}
It is an induction on $m_c$, and the initialization is the quantum Serre relation. Namely let $c=\beta \in \Pi$ ($m_c = 1$). Then as $\beta \neq \alpha_i$, it is the quantum Serre relation.
 

Now we do the inductive step in a basis element for simplicity, namely for $c = \sum \alpha_{r_i} \in \Coloring_\Pi$:
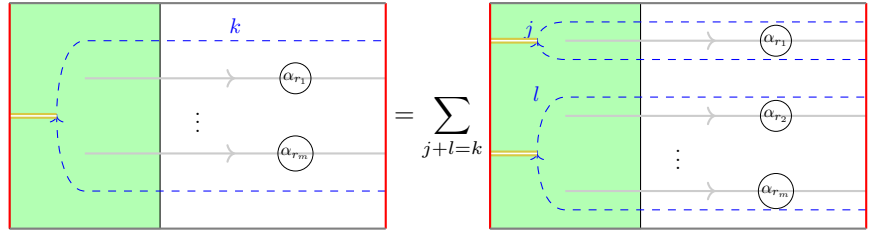
\begin{align*}
\vcenter{\hbox{\begin{tikzpicture}[scale=0.5, every node/.style={scale=0.8},decoration={
    markings,
    mark=at position 0.5 with {\arrow{>}}}
    ] 
\draw[fill = green!30!white] (-5,-3) rectangle ++(4,6);     
\coordinate (w0g) at (-3,1) {};
\coordinate (w0d) at (5,1) {};
\coordinate (w1g) at (-3,-1) {};
\coordinate (w1d) at (5,-1) {};
\coordinate (h1) at (5,-2.5);
\coordinate (h2) at (5,-1.5);
\coordinate (h3) at (5,-1);
\draw[gray!40!white, thick, postaction={decorate}] (w0g) -- (w0d) node[pos=0.7,black] (a1) {$\encircled{\alpha_{r_1}}$};
\draw[gray!40!white, thick, postaction={decorate}] (w1g) -- (w1d) node[pos=0.7,black] (a1) {$\encircled{\alpha_{r_m}}$};
\node at (0,0) {$\vdots$};
\draw[blue,dashed,postaction={decorate}] (5,-2) to coordinate[midway] (k) (-3,-2) .. controls (-4,-2) and (-4,2) .. (-3,2) coordinate[midway] (k) to node[midway,above] {$k$} (5,2);
\draw[yellow!80!black,double,thick] (k)--(-5,0);
\draw[red, thick] (-5,-3) -- (-5,3);
\draw[red, thick] (5,-3) -- (5,3);
\draw[gray, thick] (-5,-3) -- (5,-3);
\draw[gray, thick] (-5,3) -- (5,3);
\end{tikzpicture}}} 
&
= \sum_{j+l=k} 
\vcenter{\hbox{\begin{tikzpicture}[scale=0.5, every node/.style={scale=0.8},decoration={
    markings,
    mark=at position 0.5 with {\arrow{>}}}
    ] 
\draw[fill = green!30!white] (-5,-3) rectangle ++(4,6);     
\coordinate (w0g) at (-3,2) {};
\coordinate (w0d) at (5,2) {};
\coordinate (w1g) at (-3,-2) {};
\coordinate (w1d) at (5,-2) {};
\coordinate (w2g) at (-3,0) {};
\coordinate (w2d) at (5,-0) {};
\coordinate (h1) at (5,-2.5);
\coordinate (h2) at (5,-1.5);
\coordinate (h3) at (5,-1);
\draw[gray!40!white, thick, postaction={decorate}] (w0g) -- (w0d) node[pos=0.7,black] (a1) {$\encircled{\alpha_{r_1}}$};
\draw[gray!40!white, thick, postaction={decorate}] (w1g) -- (w1d) node[pos=0.7,black] (a1) {$\encircled{\alpha_{r_m}}$};
\draw[gray!40!white, thick, postaction={decorate}] (w2g) -- (w2d) node[pos=0.7,black] (a1) {$\encircled{\alpha_{r_2}}$};
\node at (0,-1) {$\vdots$};
\draw[blue,dashed,postaction={decorate}] (5,-2.5) to coordinate[midway] (k1) (-3,-2.5) .. controls (-4,-2.5) and (-4,0.5) .. (-3,0.5) node[pos=0.8,above left] {$l$} coordinate[midway] (k1) to (5,0.5);
\draw[blue,dashed,postaction={decorate}] (5,1.5) to coordinate[midway] (k2) (-3,1.5) .. controls (-4,1.5) and (-4,2.5) .. (-3,2.5) node[pos=0.7,left] {$j$} coordinate[midway] (k2) to (5,2.5);
\draw[yellow!80!black,double,thick] (k1)--(-5,-1);
\draw[yellow!80!black,double,thick] (k2)--(-5,2);
\draw[red, thick] (-5,-3) -- (-5,3);
\draw[red, thick] (5,-3) -- (5,3);
\draw[gray, thick] (-5,-3) -- (5,-3);
\draw[gray, thick] (-5,3) -- (5,3);
\end{tikzpicture}}}  
\end{align*}
where we have hidden the handles of necklaces with single pearls for clarity and since they don't play any role. We have applied a cutting rule (Prop.~\ref{prop_breaking_rule}).

Now if $l>k(i,c-\alpha_{r_1})$, then the lower part of the last diagram is zero by induction, and if $j>k(i,\alpha_{r_1})$, then the higher part of the diagram is zero y induction. Hence if $k=j+l>k(i,c-\alpha_{r_1})+k(i,\alpha_{r_1})=k(i,c)$ the whole class is zero since it is a stacking, it is a product of lower and higher part of the diagram. 
\end{proof}

we define an operator $\T_i$ by its action on diagrams:
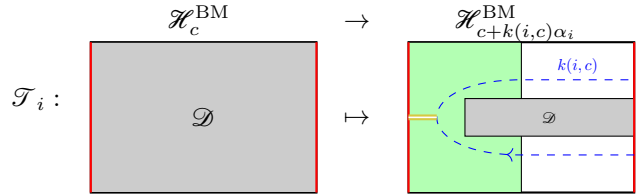
\[
\T_i : \begin{array}{ccc} 
\CHbm_c &  \to & \CHbm_{c+k(i,c)\alpha_i} \\
\vcenter{\hbox{
\begin{tikzpicture}[scale=0.25,decoration={
    markings,
    mark=at position 0.3 with {\arrow{>}}}]
\draw[thick] (0,0) rectangle ++(12,8);
\draw[fill = black!20!white] (0,0) rectangle ++(12,8);
\draw[red,thick] (0,0)--(0,8);
\draw[red,thick] (12,0)--(12,8);
\node at (6,4) {$\mathcal{D}$};
\end{tikzpicture}}}

& \mapsto &  
\vcenter{\hbox{
\begin{tikzpicture}[scale=0.25, every node/.style={scale=0.6},decoration={
    markings,
    mark=at position 0.3 with {\arrow{>}}}]
\draw[thick] (0,0) rectangle ++(12,8);
\draw[fill = green!30!white] (0,0) rectangle ++(6,8);
\draw[fill = black!20!white] (3,3) rectangle ++(9,2);
\draw[red,thick] (0,0)--(0,8);
\draw[red,thick] (12,0)--(12,8);
\coordinate (kk) at (12,2.5);
\draw[blue,dashed,postaction={decorate}] (12,2) to coordinate[midway] (k) (6,2) .. controls (0,2) and (0,6) .. (6,6) to node[midway,above] {$k(i,c)$} (12,6);
\node at (7.5,4) {$\mathcal{D}$};
\draw[yellow!80!black,double,thick] (1.5,4)--(0,4);
\end{tikzpicture}}}
\end{array}
\]
where $\mathcal{D}$ stands for any diagram contained in the gray shaded rectangle. This diagrammatic definition is enough since bases of $\CHbm_c$ are given by diagrams, although the action of $\T_i$ could have been defined getting rid of diagrams but with a cumbersome process. 
\begin{coro}\label{C:Ti_is_Lusztig}
For a given simple root $\alpha_i \in \Pi$, the map $\T_i$ restricted to
\[
\bigoplus_{c, c(\alpha_i) = 0} \overline{\CH}_c
\]
is the appropriate restriction of the algebra morphism associated with the corresponding braid in $\CB_{\fg}$ (following its definition from \cite[Theorem~3.1]{Lus90}). 
\end{coro}
\begin{proof}
In \cite[Theorem~3.1]{Lus90}, it is shown that algebra morphisms $\T_i$ are uniquely defined by their value on algebra generators. Proposition~\ref{P:Ti_on_generators} shows that it has the same value on generators. It remains to proof that our $\T_i$ are algebra morphisms. We prove this fact using diagrams. Let $\mathcal{D}_{c_1}$ (resp. $\mathcal{D}_{c_2}$) be a diagram defining a homology class in $\overline{\CH}_{c_1}$ (resp. in $\overline{\CH}_{c_2}$), such that $c_1(\alpha_i)=c_2(\alpha_i) = 0$. We recall the product:
\begin{align*}
\mathcal{D}_{c_1}\mathcal{D}_{c_2} & 
= \vcenter{\hbox{
\begin{tikzpicture}[scale=0.25,decoration={
    markings,
    mark=at position 0.3 with {\arrow{>}}}]
\draw[thick] (0,0) rectangle ++(12,8);
\draw[fill = black!20!white] (0,0) rectangle ++(12,8);
\draw[dashed] (0,4)--(12,4);
\draw[red,thick] (0,0)--(0,8);
\draw[red,thick] (12,0)--(12,8);
\node at (6,6) {$\mathcal{D}_{c_1}$};
\node at (6,2) {$\mathcal{D}_{c_2}$};
\end{tikzpicture}}}
\end{align*}
and that:
\begin{align*}
\T_i(\mathcal{D}_{c_1}\mathcal{D}_{c_2}) & =
\vcenter{\hbox{
\begin{tikzpicture}[scale=0.3, every node/.style={scale=0.6},decoration={
    markings,
    mark=at position 0.3 with {\arrow{>}}}]
\draw[thick] (0,0) rectangle ++(12,8);
\draw[fill = green!30!white] (0,0) rectangle ++(6,8);
\draw[fill = black!20!white] (3,3) rectangle ++(9,2);
\draw[red,thick] (0,0)--(0,8);
\draw[red,thick] (12,0)--(12,8);
\coordinate (kk) at (12,2.5);
\draw[blue,dashed,postaction={decorate}] (12,2) to coordinate[midway] (k) (6,2) .. controls (0,2) and (0,6) .. (6,6) to node[midway,above] {$k(i,c_1+c_2)$} (12,6);
\node at (7.5,4) {$\mathcal{D}_{c_1}\mathcal{D}_{c_2}$};
\draw[yellow!80!black,double,thick] (1.5,4)--(0,4);
\end{tikzpicture}}}.
\end{align*}
One can apply a cutting rule on the blue dashed arc, splitting it on the right left side of the rectangle, following the dashed line limiting the border of both diagrams in the product $\mathcal{D}_{c_1}\mathcal{D}_{c_2}$. It results in:
\begin{align*}
\T_i(\mathcal{D}_{c_1}\mathcal{D}_{c_2}) &  = \sum_{k+k' = k(i,c_1+c_2)} 
\vcenter{\hbox{
\begin{tikzpicture}[scale=0.3, every node/.style={scale=0.6},decoration={
    markings,
    mark=at position 0.3 with {\arrow{>}}}]
\draw[thick] (0,0) rectangle ++(12,8);
\draw[fill = green!30!white] (0,0) rectangle ++(6,8);
\draw[fill = black!20!white] (3,3) rectangle ++(9,2);
\draw[red,thick] (0,0)--(0,8);
\draw[red,thick] (12,0)--(12,8);
\coordinate (kk) at (12,2.5);
\draw[blue,dashed,postaction={decorate}] (12,2) to coordinate[midway] (k) (6,2) .. controls (0,2) and (0,6) .. (6,6) to node[midway,above] {$k$} (12,6);
\node at (7.5,4) {$\mathcal{D}_{c_1}$};
\draw[yellow!80!black,double,thick] (1.5,4)--(0,4);
\draw[thick] (0,8) rectangle ++(12,8);
\draw[fill = green!30!white] (0,8) rectangle ++(6,8);
\draw[fill = black!20!white] (3,11) rectangle ++(9,2);
\draw[red,thick] (0,8)--(0,16);
\draw[red,thick] (12,8)--(12,16);
\coordinate (kkp) at (12,10.5);
\draw[blue,dashed,postaction={decorate}] (12,10) to coordinate[midway] (kp) (6,10) .. controls (0,10) and (0,14) .. (6,14) to node[midway,above] {$k'$} (12,14);
\node at (7.5,12) {$\mathcal{D}_{c_2}$};
\draw[yellow!80!black,double,thick] (1.5,12)--(0,12);
\end{tikzpicture}}} \\
& = \sum_{k+k' = k(i,c_1+c_2)} 
\vcenter{\hbox{
\begin{tikzpicture}[scale=0.2, every node/.style={scale=0.6},decoration={
    markings,
    mark=at position 0.3 with {\arrow{>}}}]
\draw[thick] (0,0) rectangle ++(12,8);
\draw[fill = green!30!white] (0,0) rectangle ++(6,8);
\draw[fill = black!20!white] (3,3) rectangle ++(9,2);
\draw[red,thick] (0,0)--(0,8);
\draw[red,thick] (12,0)--(12,8);
\coordinate (kk) at (12,2.5);
\draw[blue,dashed,postaction={decorate}] (12,2) to coordinate[midway] (k) (6,2) .. controls (0,2) and (0,6) .. (6,6) to node[midway,above] {$k$} (12,6);
\node at (7.5,4) {$\mathcal{D}_{c_1}$};
\draw[yellow!80!black,double,thick] (1.5,4)--(0,4);
\end{tikzpicture}}} \times
\vcenter{\hbox{
\begin{tikzpicture}[scale=0.2, every node/.style={scale=0.6},decoration={
    markings,
    mark=at position 0.3 with {\arrow{>}}}]
\draw[thick] (0,0) rectangle ++(12,8);
\draw[fill = green!30!white] (0,0) rectangle ++(6,8);
\draw[fill = black!20!white] (3,3) rectangle ++(9,2);
\draw[red,thick] (0,0)--(0,8);
\draw[red,thick] (12,0)--(12,8);
\coordinate (kk) at (12,2.5);
\draw[blue,dashed,postaction={decorate}] (12,2) to coordinate[midway] (k) (6,2) .. controls (0,2) and (0,6) .. (6,6) to node[midway,above] {$k'$} (12,6);
\node at (7.5,4) {$\mathcal{D}_{c_2}$};
\draw[yellow!80!black,double,thick] (1.5,4)--(0,4);
\end{tikzpicture}}} .
\end{align*}
Notice that $k(i,c_1+c_2) = k(i,c_1)+k(i,c_2)$ and if $k > k(i,c_1)$, we have:
\[
\vcenter{\hbox{
\begin{tikzpicture}[scale=0.25, every node/.style={scale=0.6},decoration={
    markings,
    mark=at position 0.3 with {\arrow{>}}}]
\draw[thick] (0,0) rectangle ++(12,8);
\draw[fill = green!30!white] (0,0) rectangle ++(6,8);
\draw[fill = black!20!white] (3,3) rectangle ++(9,2);
\draw[red,thick] (0,0)--(0,8);
\draw[red,thick] (12,0)--(12,8);
\coordinate (kk) at (12,2.5);
\draw[blue,dashed,postaction={decorate}] (12,2) to coordinate[midway] (k) (6,2) .. controls (0,2) and (0,6) .. (6,6) to node[midway,above] {$k$} (12,6);
\node at (7.5,4) {$\mathcal{D}_{c_1}$};
\draw[yellow!80!black,double,thick] (1.5,4)--(0,4);
\end{tikzpicture}}} = 0.
\]
It is Prop.~\ref{P:truncature}. We thus have:
\begin{align*}
\T_i(\mathcal{D}_{c_1}\mathcal{D}_{c_2}) & = 
\vcenter{\hbox{
\begin{tikzpicture}[scale=0.25, every node/.style={scale=0.6},decoration={
    markings,
    mark=at position 0.3 with {\arrow{>}}}]
\draw[thick] (0,0) rectangle ++(12,8);
\draw[fill = green!30!white] (0,0) rectangle ++(6,8);
\draw[fill = black!20!white] (3,3) rectangle ++(9,2);
\draw[red,thick] (0,0)--(0,8);
\draw[red,thick] (12,0)--(12,8);
\coordinate (kk) at (12,2.5);
\draw[blue,dashed,postaction={decorate}] (12,2) to coordinate[midway] (k) (6,2) .. controls (0,2) and (0,6) .. (6,6) to node[midway,above] {$k(i,c_1)$} (12,6);
\node at (7.5,4) {$\mathcal{D}_{c_1}$};
\draw[yellow!80!black,double,thick] (1.5,4)--(0,4);
\end{tikzpicture}}} \times
\vcenter{\hbox{
\begin{tikzpicture}[scale=0.25, every node/.style={scale=0.6},decoration={
    markings,
    mark=at position 0.3 with {\arrow{>}}}]
\draw[thick] (0,0) rectangle ++(12,8);
\draw[fill = green!30!white] (0,0) rectangle ++(6,8);
\draw[fill = black!20!white] (3,3) rectangle ++(9,2);
\draw[red,thick] (0,0)--(0,8);
\draw[red,thick] (12,0)--(12,8);
\coordinate (kk) at (12,2.5);
\draw[blue,dashed,postaction={decorate}] (12,2) to coordinate[midway] (k) (6,2) .. controls (0,2) and (0,6) .. (6,6) to node[midway,above] {$k(i,c_2)$} (12,6);
\node at (7.5,4) {$\mathcal{D}_{c_2}$};
\draw[yellow!80!black,double,thick] (1.5,4)--(0,4);
\end{tikzpicture}}} \\
& = \T_i(\mathcal{D}_{c_1}) \times \T_i(\mathcal{D}_{c_2})
\end{align*}
\end{proof}

\begin{rmk}
In a future work, we will investigate how these restrictions of the braid algebra morphisms from \cite[Theorem~3.1]{Lus90} are actually sufficient to recover the Poincaré--Birkhoff--Witt basis displayed in \cite[Sec.~4]{Lus90}. It relies on the fact that in the given recursive expressions, the applications of $\T_i$ only used their restrictions involved in the above corollary. It would provide homological bases of $\overline{\CH}$ that is the image of the standard homology in the Borel--Moore one by the canonical map arising from inclusion of chain complexes. 
\end{rmk}

\section{Representation theory of $\Uqg$ from homology}\label{BS:rep_theory}

In this section, we recover modules on the quantum algebra $\Uqg$ from twisted homologies of colored configuration spaces on punctured disks. We will first construct one of the algebra $\Uqgm$ since it was fully recognized in homologies of colored configuration spaces of empty disks in Theorem~\ref{T:homological_version_for_Uqgm} and Theorem~\ref{T:Uqgm_is_homological}. The action is by stacking the empty disk on punctured ones, then we will add the rest of $\Uqg$ acting by natural homological operators. 

\subsection{Twisted homology of configuration spaces of punctured disks}\label{S:twisted_homology_punctures_bases_etc}

We let $D_n$ be the disk with $n$ punctures. It is a disk (represented by a square) from which we remove $n$ points that are denoted $w_1,\ldots,w_n$. Let also  $D_n^\circ$ be the disk with $n$ holes, namely the disk from which one rather removes open disks, neighboring the punctures. In what follows $D_n^{\#}$ will designate either one or the other. Notice that both are homotopy equivalent and so are their configuration spaces. Here is a picture:
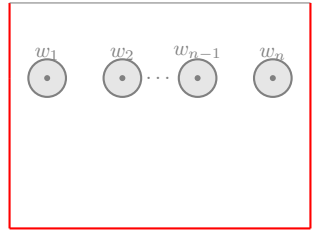
\begin{equation*} 
\vcenter{\hbox{\begin{tikzpicture}[scale=0.5, every node/.style={scale=0.8},decoration={
    markings,
    mark=at position 0.5 with {\arrow{>}}}
    ]
\node (w1) at (-3,0) {};
\node (w2) at (-1,0) {};
\node[gray] at (0.0,0.0) {\ldots};
\node (wn1) at (1,0) {};
\node (wn) at (3,0) {};

    \filldraw[color=gray, fill=gray!20!white,thick] (w1) circle (0.5);  
    \filldraw[color=gray, fill=gray!20!white,thick] (w2) circle (0.5); 
    \filldraw[color=gray, fill=gray!20!white,thick] (wn1) circle (0.5); 
    \filldraw[color=gray, fill=gray!20!white,thick] (wn) circle (0.5); 

%
%
%
%

\node[gray] at (w1)[above=5pt] {$w_1$};
\node[gray] at (w2)[above=5pt] {$w_2$};
\node[gray] at (wn1)[above=5pt] {$w_{n-1}$};
\node[gray] at (wn)[above=5pt] {$w_n$};
\foreach \n in {w1,w2,wn1,wn}
  \node at (\n)[gray,circle,fill,inner sep=1pt]{};

\draw[thick,red] (-4,0) -- (-4,2);
\draw[thick,red] (-4,0) -- (-4,-4);
\draw[thick, red] (-4,-4) -- (4,-4);
\draw[gray] (4,2) -- (-4,2);
\draw[thick, red] (4,-4) -- (4,2);

\end{tikzpicture}}}
\end{equation*}
We denote by $\partial^- D_n$ the part of the boundary that is marked in red. We define its colored configuration spaces: when $c \in \Coloring_{\Pi}$, we denote by $\Conf_c$ the $c$-colored configuration space, defined as follows:
\[
\Conf_c(D_n^{\#}) := \left( (D_n^{\#})^{m_c} \setminus \bigcup_{i<j} \{z_i=z_j\} \right) \Big/ \Sk_{c(\alpha_1)} \times \cdots \times \Sk_{c(\alpha_l)}
\]
We enlarge the local system in comparison with the empty disk case (Sec.~\ref{S:local_system_empty_disk}). Namely, associated with a coloring $c$, we consider the map $\Phi_c(D_n)$ defined as follows:
\[
\bapp
\Conf_c & \to & (\BS^1)^{l} \times (\BS^1)^{\frac{l(l-1)}{2}} \times \left((\BS^1)^{l}\right)^n \\
(\bz^{\alpha_1}, \ldots , \bz^{\alpha_l}) & \mapsto & \Phi_c \times \prod_{i=1}^n \left( g(\bz^{\alpha_1}, w_i), \ldots , g(\bz^{\alpha_l},w_i) \right)
\eapp
\]
where $\Phi_c$ was defined in the empty disk case (Sec.~\ref{S:local_system_empty_disk}), this time we don't use a rotation, and we fix $\CW_c(D_n):= \Phi_c(D_n)$. In addition to the writhe that $\CW_c$ computes, $\CW_c(D_n)$ also computes the winding numbers of packages of colored points around the punctures.  

We denote by $B_c(D_n) := (\BS^1)^{l} \times (\BS^1)^{\frac{l(l-1)}{2}} \times \left((\BS^1)^{l}\right)^n$ and by $b_{c,n} := (1, \ldots, 1)$ a base point. We denote the fundamental group of $B_c(D_n)$ based at $b_{c,n}$ by $\pi_c(D_n)$ and we recall that:
\[
\pi_c(D_n) \simeq \BZ^l \times \BZ^{\frac{l(l-1)}{2}} \times \left(\BZ^{l}\right)^n := \pi_c \oplus \bigoplus_{1\le i \le l, 1\le j \le n} \BZ \langle \mathfrak{w}(i,j) \rangle 
\]
where to $\pi_c$ was added generators $\mathfrak{w}(i,j)$ which monodromy computes the winding number of the $i$-colored package of points around puncture $w_j$. Notice the important remark that this time not composing by a rotation makes any horizontally aligned configuration also aligned with punctures play the role of base point. So to use horizontally aligned configuration on $\partial D_n$, one has to first apply a small transformation to the disk, we will show example.


We recall that the fundamental group of $\Conf_c(D_n^{\#})$ (up to a chosen base point) is made of the simple braids (colored by roots) that were part of $\Conf_c$ plus the following ones. Let $\mathfrak{b}_{i}(k)$ be a loop in $\Conf_c(D_n^{\#})$ where every coordinates of the base point stay fixed at all time besides one, colored by $\alpha_i$, that goes and wind once around the puncture $w_k$. We refer the reader to the picture \cite[Figure~2]{Jules_Verma}, we think about the braid denoted $B_{r,k}$ there but in the colored case so that now the strand that winds is colored by $\alpha_i$. 

Since we wish to construct modules on $\Uqg$ we mention now that we will do \textit{weight modules} which are usually parametrized by a choice of a \textit{character on the Cartan part of $\Uqg$} corresponding in fixing a highest weight in the module. For an $n$-times punctured disk will correspond a $n$-fold tensor product of modules, hence to a choice of $n$ such characters. Traditionally, one thinks about integral characters so that a generator $K_{\alpha}$ of the Cartan acts on the hisghest weight vector by multiplication by $q^{n_\alpha}$ where $n_\alpha \in \BN$ is the value of the character on $K_{\alpha}$. It gives rise to the standard notion of Verma modules, but when dealing with infinite dimensional modules we can actually choose the character to lie anywhere and we talk about \textit{universal Verma modules}. Instead of using the notation $q^{\lambda_\alpha}$ with $\lambda_\alpha \in \BC$, which makes one thinks that $q$ is a complex number with a given choice of logarithm so that it can be expanded to the power $\lambda_\alpha$, we will replace $q^{\lambda_\alpha}$ by a formal variable traditionally denoted by $s_{\alpha}$ (but there will be many of them in tensor products). Most of our modules will then have the very strength of being integral, namely having coordinates in an integral ring of Laurent polynomials in variables $q$ and $s_\alpha$'s, it is allowed by homology theories that are usually integral. We talk about integral universal Verma modules. To read about Verma modules on $\Uqg$ (weights, characters etc), we refer the reader to \cite{Pedro} (where Verma modules are also categorified something we won't discuss in the present paper). \textit{All this weight structure and character will be fixed and encoded in the local system.} Since every puncture will encode one fold of a tensor product of modules, at each puncture we will choose one formal variable per simple root. We let then, for any $1\le i\le n$:
\[
L_i = (s_{i,1}, \ldots, s_{i,l}),
\]
be a list of formal variables, one for each simple root. We are set to define a representation of $\pi_c(D_n^{\#})$ that will be the local ring of coefficients with which we will twist the homology. We denote by $\boldsymbol{L}=(L_1,\ldots,L_n)$ the total set of formal variables. We define the ring $\Laurent_{\boldsymbol{L}} := \BZ \left[ q^{\pm 1} , (s^i_j)^{\pm 1} \right]_{1\le i \le n,1\le j \le l}$, and we define a morphism:
\[
\varphi_c(D_n^{\#}) : \bfct
\BZ \left[ \pi_c(D_n^{\#}) \right]  &  \to & \Laurent_{\boldsymbol{L}} \\
k_i & \mapsto & -q^{(\alpha_i,\alpha_i)/2} = -q^{d_i} \\
k_{(i,j)} & \mapsto & q^{(\alpha_i, \alpha_j)/2}\\
\mathfrak{w}(i,j) & \mapsto & s_{i,j}.
\efct
\]
We will denote $\Laurent_c(D_n)$ the entire data set $(\Laurent_{\boldsymbol{L}}, \CW_c(D_n^{\#}), \varphi_c(D_n^{\#}))$ reunited under the name of local system (or local ring of coefficients). We now use it to define the corresponding twisted homology. Notice that winding one half around a puncture pays a variable $s_\alpha$. 

We define $S(D_n) \in \partial \Conf_c(D_n^{\#})$ to be made of configurations having a coordinate in $\partial^-D_n$, and we denote:
\[
\CHbm_c(D_n^{\#}) := \Hlf_{m_c} (\Conf_c(D_n^{\#}), S(D_n); \Laurent_c(D_n)).
\]
We also define $\CHbm(D_n^{\#}) :=  \bigoplus_{c\in\Coloring_{\Pi}} \CHbm_c(D_n)$ since this will be the module on $\Uqg$ for which the construction is in progress. 

Let $c\in \Coloring_{\Pi}$ and for $1\le i \le n$, let $\boldsymbol{r}^i := (r^i_1, \ldots , r^i_{m_i}) \in \lbrace 1, \ldots , l \rbrace^m$ be such that $(\boldsymbol{r}^1 , \ldots, \boldsymbol{r}^n)\in \CP_c$ yields a partition of the coloring $c$ (for instance $\sum m_i = m_c$). The following diagram naturally defines an element of $\CHbm_c(D_n)$:
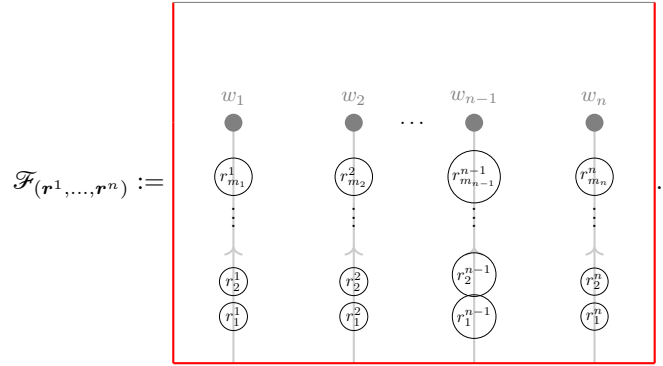
\[
\CF_{(\boldsymbol{r}^1 , \ldots, \boldsymbol{r}^n)}:= 
\vcenter{\hbox{\begin{tikzpicture}[scale=0.8, every node/.style={scale=0.8},decoration={
    markings,
    mark=at position 0.5 with {\arrow{>}}}
    ]
    
\node (w1) at (-3,0) {};
\node (w2) at (-1,0) {};
\node[thick] at (0.0,0.0) {\ldots};
\node (wn1) at (1,0) {};
\node (wn) at (3,0) {};
\draw[gray!40!white, thick, postaction={decorate}] (-3,-4) -- (w1) node[pos=0.2,black] (a1) {$\encircled{r^1_1}$} node[pos=0.35,black] (a2) {$\encircled{r^1_2}$} node[pos=0.55,above,black,thick] {$\vdots$} node[pos=0.8, black] (a3) {$\encircled{r^1_{m_1}}$};
\draw[gray!40!white, thick, postaction={decorate}] (-1,-4) -- (w2) node[pos=0.2,black] (a1) {$\encircled{r^2_1}$} node[pos=0.35,black] (a2) {$\encircled{r^2_2}$} node[pos=0.55,above,black,thick] {$\vdots$} node[pos=0.8, black] (a3) {$\encircled{r^2_{m_2}}$};
\draw[gray!40!white, thick, postaction={decorate}] (1,-4) -- (wn1) node[pos=0.2,black] (a1) {$\encircled{r^{n-1}_1}$} node[pos=0.38,black] (a2) {$\encircled{r^{n-1}_2}$} node[pos=0.55,above,black,thick] {$\vdots$} node[pos=0.8, black] (a3) {$\encircled{r^{n-1}_{m_{n-1}}}$};
\draw[gray!40!white, thick, postaction={decorate}]  (3,-4) -- (wn) node[pos=0.2,black] (a1) {$\encircled{r^{n}_1}$} node[pos=0.35,black] (a2) {$\encircled{r^{n}_2}$} node[pos=0.55,above,black,thick] {$\vdots$} node[pos=0.8, black] (a3) {$\encircled{r^{n}_{m_n}}$};


\node[gray] at (w1)[above=5pt] {$w_1$};
\node[gray] at (w2)[above=5pt] {$w_2$};
\node[gray] at (wn1)[above=5pt] {$w_{n-1}$};
\node[gray] at (wn)[above=5pt] {$w_n$};
\foreach \n in {w1,w2,wn1,wn}
  \node at (\n)[gray,circle,fill,inner sep=3pt]{};

\draw[thick,red] (-4,0) -- (-4,2);
\draw[thick,red] (-4,0) -- (-4,-4);
\draw[thick, red] (-4,-4) -- (4,-4);
\draw[gray] (4,2) -- (-4,2);
\draw[thick, red] (4,-4) -- (4,2);

\end{tikzpicture}}} .
\]
This diagram describes a homology class exactly as it was described in Notation~\ref{N:pearl_necklace}. This time we have replaced $\alpha_{r^i_j}$ simply by the index ${r^i_j}$ for more clarity. Actually it defines a non twisted homology class since yellow handles are missing. One must choose one for each necklace, a choice is hidden here to avoid cumbersome notations. We simplify even more the notation and display our choice of handles in the following diagram:

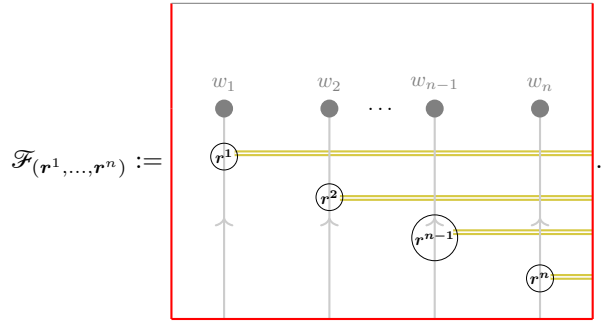
\[
\CF_{(\boldsymbol{r}^1 , \ldots, \boldsymbol{r}^n)}:= 
\vcenter{\hbox{\begin{tikzpicture}[scale=0.7, every node/.style={scale=0.8},decoration={
    markings,
    mark=at position 0.5 with {\arrow{>}}}
    ]
    
\draw[double,\hyellow,thick] (-2.8,-0.85)--(4,-0.85);
\draw[double,\hyellow,thick] (-0.8,-1.7)--(4,-1.7);
\draw[double,\hyellow,thick] (1.35,-2.35)--(4,-2.35);
\draw[double,\hyellow,thick] (3.2,-3.2)--(4,-3.2);
\node (w1) at (-3,0) {};
\node (w2) at (-1,0) {};
\node[thick] at (0.0,0.0) {\ldots};
\node (wn1) at (1,0) {};
\node (wn) at (3,0) {};
\draw[gray!40!white, thick, postaction={decorate}] (-3,-4) -- (w1) node[pos=0.8,black] (a1) {$\encircled{\boldsymbol{r^1}}$};
\draw[gray!40!white, thick, postaction={decorate}] (-1,-4) -- (w2) node[pos=0.6,black] (a1) {$\encircled{\boldsymbol{r^2}}$};
\draw[gray!40!white, thick, postaction={decorate}] (1,-4) -- (wn1) node[pos=0.4,black] (a1) {$\encircled{\boldsymbol{r^{n-1}}}$};
\draw[gray!40!white, thick, postaction={decorate}]  (3,-4) -- (wn) node[pos=0.2,black] (a1) {$\encircled{\boldsymbol{r^n}}$};


\node[gray] at (w1)[above=5pt] {$w_1$};
\node[gray] at (w2)[above=5pt] {$w_2$};
\node[gray] at (wn1)[above=5pt] {$w_{n-1}$};
\node[gray] at (wn)[above=5pt] {$w_n$};
\foreach \n in {w1,w2,wn1,wn}
  \node at (\n)[gray,circle,fill,inner sep=3pt]{};

\draw[thick,red] (-4,0) -- (-4,2);
\draw[thick,red] (-4,0) -- (-4,-4);
\draw[thick, red] (-4,-4) -- (4,-4);
\draw[gray] (4,2) -- (-4,2);
\draw[thick, red] (4,-4) -- (4,2);

\end{tikzpicture}}} .
\]
where all the pearls on the same necklace are replaced by one bold pearl encoding all the information, and yellow handles are replaced by rays of yellow handles that shall be thought as parallel handles. 

Following the same simplified notation for pearls, the following diagrams define natural classes in $\CHbm(D_n^\circ)$. 

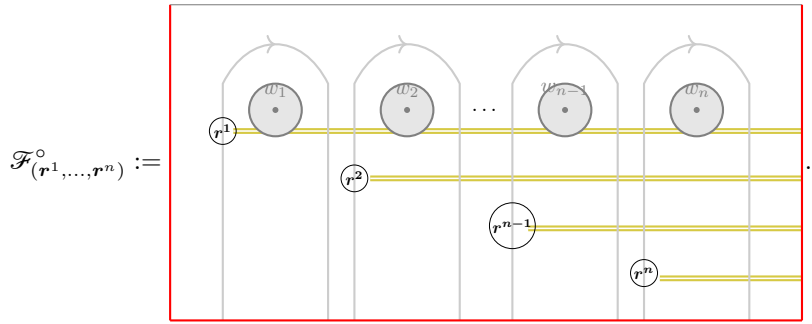
\[
\CF^\circ_{(\boldsymbol{r}^1 , \ldots, \boldsymbol{r}^n)}:= 
\vcenter{\hbox{\begin{tikzpicture}[scale=0.7, every node/.style={scale=0.8},decoration={
    markings,
    mark=at position 0.5 with {\arrow{>}}}
    ]
    
\draw[double,\hyellow,thick] (-3.8,-0.4)--(7,-0.4);
\draw[double,\hyellow,thick] (-1.2,-1.3)--(7,-1.3);
\draw[double,\hyellow,thick] (1.8,-2.25)--(7,-2.25);
\draw[double,\hyellow,thick] (4.3,-3.2)--(7,-3.2);
\node (w1) at (-3,0) {};
\node (w2) at (-0.5,0) {};
\node[thick] at (1.0,0.0) {\ldots};
\node (wn1) at (2.5,0) {};
\node (wn) at (5,0) {};

    \filldraw[color=gray, fill=gray!20!white,thick] (w1) circle (0.5);  
    \filldraw[color=gray, fill=gray!20!white,thick] (w2) circle (0.5); 
    \filldraw[color=gray, fill=gray!20!white,thick] (wn1) circle (0.5); 
    \filldraw[color=gray, fill=gray!20!white,thick] (wn) circle (0.5);

%
%
%


    \draw[gray!40!white, thick, postaction={decorate}] (-4,-4)--(-4,0.5) node[pos=0.8,black] {$\encircled{\boldsymbol{r^1}}$} .. controls (-3.5,1.5) and (-2.5,1.5) .. (-2,0.5)  --(-2,-4);
    \draw[gray!40!white, thick, postaction={decorate}] (-1.5,-4)--(-1.5,0.5) node[pos=0.6,black] {$\encircled{\boldsymbol{r^2}}$} .. controls (-1,1.5) and (0,1.5) .. (0.5,0.5)  --(0.5,-4);
    \draw[gray!40!white, thick, postaction={decorate}] (1.5,-4)--(1.5,0.5) node[pos=0.4,black] {$\encircled{\boldsymbol{r^{n-1}}}$} .. controls (2,1.5) and (3,1.5) .. (3.5,0.5)  --(3.5,-4);
    \draw[gray!40!white, thick, postaction={decorate}] (4,-4)--(4,0.5) node[pos=0.2,black] {$\encircled{\boldsymbol{r^n}}$} .. controls (4.5,1.5) and (5.5,1.5) .. (6,0.5)  --(6,-4);


\node[gray] at (w1)[above=2.5pt] {$w_1$};
\node[gray] at (w2)[above=2.5pt] {$w_2$};
\node[gray] at (wn1)[above=2.5pt] {$w_{n-1}$};
\node[gray] at (wn)[above=2.5pt] {$w_n$};
\foreach \n in {w1,w2,wn1,wn}
  \node at (\n)[gray,circle,fill,inner sep=1pt]{};

\draw[thick,red] (-5,0) -- (-5,2);
\draw[thick,red] (-5,0) -- (-5,-4);
\draw[thick, red] (-5,-4) -- (7,-4);
\draw[gray] (7,2) -- (-5,2);
\draw[thick, red] (7,-4) -- (7,2);

\end{tikzpicture}}} .
\]

The Proposition~\ref{structure_result} adapts perfectly to punctured disks, and the result provides the structure of these homologies as modules on $\Laurent_{\boldsymbol{L}}$. For the case of punctured disks the proof would also be almost word by word that of \cite[Proposition~3.6]{Jules_Verma}. In both cases the only difference would be first for the retract used in the proof to use that of $D_n^{\#}$ on a neighborhood of the graph composed of the red part plus the light gray strings of necklaces, and second to use colored configuration space (and this is transparent). 

\begin{prop}\label{P:structure_result_punctures}
Let $c \in \Coloring_{\Pi}$, then the module $\CHbm_c(D_n^\#)$ is:
\begin{itemize}
\item a free $\Laurent_{\boldsymbol{L}}$-module,
\item for which the set $\CB_{\CHbm_c(D_n^\#)} := \left\lbrace \CF^\#_{(\br^1, \ldots, \br^n)} \text{ s.t. } (\br^1, \ldots, \br^n) \in \CP_c \right\rbrace$ is a basis.
\item It is the only non-vanishing module from the sequence $\Hlf_{\bullet} (\Conf_c(D_n^\#), S(D_n); \Laurent_c(D_n))$ (the homology is concentrated in the middle dimension).
\end{itemize}
\end{prop}

We also define standard homologies as follows:

\[
\CH_c(D_n^{\#}) = \Hnot_c\left(\Conf_c(\overline{D}_n^\#) , S(D_n), \Laurent_c(D_n^{\#}) \right),
\]
and $\CH(D_n^{\#}) :=  \bigoplus_{c\in\Coloring_{\Pi}} \CH_c(D_n)$. The following diagrams naturally define classes in $\CH_c(D_n^{\#})$:
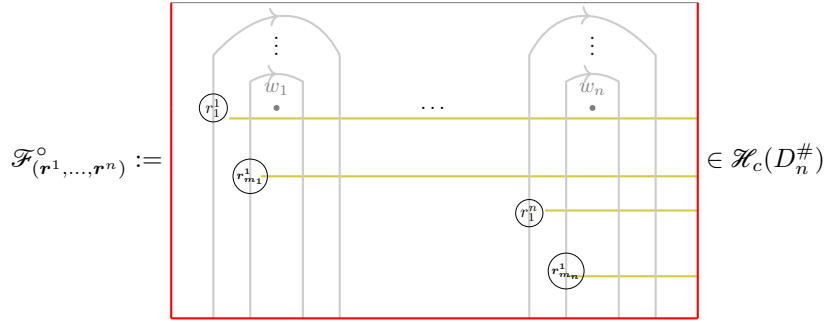
\[
\CF^\circ_{(\boldsymbol{r}^1 , \ldots, \boldsymbol{r}^n)}:= 
\vcenter{\hbox{\begin{tikzpicture}[scale=0.7, every node/.style={scale=0.8},decoration={
    markings,
    mark=at position 0.5 with {\arrow{>}}}
    ]
    
\draw[\hyellow,thick] (-3.9,-0.2)--(5,-0.2);
\draw[\hyellow,thick] (-3.3,-1.3)--(5,-1.3);
\draw[\hyellow,thick] (2.1,-1.95)--(5,-1.95);
\draw[\hyellow,thick] (2.5,-3.2)--(5,-3.2);
\node (w1) at (-3,0) {};
\node at (-3,1.3) {$\vdots$};
\node[thick] at (0.0,0.0) {\ldots};

\node (wn) at (3,0) {};
\node at (3,1.3) {$\vdots$};

    \draw[gray!40!white, thick, postaction={decorate}] (-4.2,-4)--(-4.2,1) node[pos=0.8,black] {$\encircled{{r_1^1}}$} .. controls (-3.3,2) and (-2.3,2) .. (-1.8,1)  --(-1.8,-4);
    \draw[gray!40!white, thick, postaction={decorate}] (-3.5,-4)--(-3.5,0.5) node[pos=0.6,black] {\scriptsize $\encircled{\boldsymbol{r_{m_1}^1}}$} to[bend left] (-2.5,0.5)  --(-2.5,-4);
    \draw[gray!40!white, thick, postaction={decorate}] (1.8,-4) -- (1.8,1) node[pos=0.4,black] (m1) {$\encircled{{r_1^n}}$} .. controls (2.3,2) and (3.3,2) .. (4.2,1) -- (4.2,-4);
    \draw[gray!40!white, thick, postaction={decorate}] (2.5,-4) -- (2.5,0.5) node[pos=0.2,black] (mn) {\scriptsize $\encircled{\boldsymbol{r_{m_n}^1}}$} to[bend left] (3.5,0.5) -- (3.5,-4);
    
    

\node[gray] at (w1)[above=2.5pt] {$w_1$};
\node[gray] at (wn)[above=2.5pt] {$w_n$};
\foreach \n in {w1,wn}
  \node at (\n)[gray,circle,fill,inner sep=1pt]{};

\draw[thick,red] (-5,0) -- (-5,2);
\draw[thick,red] (-5,0) -- (-5,-4);
\draw[thick, red] (-5,-4) -- (5,-4);
\draw[gray] (5,2) -- (-5,2);
\draw[thick, red] (5,-4) -- (5,2);

\end{tikzpicture}}} \in \CH_c(D_n^\#)
\]
where this time there are $m_i$ necklaces nested into each others around puncture $w_i$ each of them carrying one pearl decorated by one $r^i_j$. Since there is one pearl per necklace they define standard homology classes and we regard them as standard homology classes. We have the following structure theorem.

\begin{prop}\label{P:structure_result_punctures_standard}
There is a perfect pairing:
\begin{equation}\label{E:the_pairing_punctures}
\langle \cdot, \cdot \rangle : \CHbm_c(D_n) \otimes \CH_c(D_n) \to \Laurent(\boldsymbol{L}).
\end{equation}
As a consequence for $c \in \Coloring_{\Pi}$, the module $\CH_c(D_n)$ is:
\begin{itemize}
\item a free $\Laurent_{\boldsymbol{L}}$-module,
\item for which the set $\CB_{\CH_c(D_n)} := \left\lbrace \CF^{[\br^1, \ldots, \br^n]} \text{ s.t. } (\br^1, \ldots, \br^n) \in \CP_c \right\rbrace$ is a basis.
\end{itemize}
\end{prop}
\begin{proof}
There exists a perfect pairing:
\[
\langle \cdot, \cdot \rangle_{D_n} : \CH_c(D_n) \otimes \Hlf_{m_c}(\Conf_c(D_n), T(D_n); \Laurent_c)  \to \Laurent(\boldsymbol{L})
\]
from Appendix~\ref{A:the_pairing}, where $T(D_n)$ means configurations with at least a coordinate in $\partial^+D_n = \partial D_n \setminus \partial ^- D_n$ (namely the gray part of the boundary). We define $\langle \cdot, \cdot \rangle$ first by flipping the disk (using the complex conjugation) in the right term. It follows the exact conventions of \cite[(19)]{Pierre}. Now families $\CB_{\CH_c(D_n)}$ and $\CB_{\CHbm_c(D_n)}$ are dual for this pairing and the computation adapts word by word from the $\slt$ uncolored case, which is performed in \cite[Theorem~5.15]{Pierre}. Notice that this computation of dual bases was already displayed in \cite[Prop.~65]{JulesSonny} but they work relatively to a point in the boundary (rather than a whole interval) which explains the slight difference in diagrams. The structure of $\CH_c(D_n)$ is then a consequence of Proposition~\ref{P:structure_result_punctures}.
\end{proof}

Now we recall that our diagrams above are a bit wrong since handles do not reach base points, as they don't reach horizontally aligned configurations (also aligned with punctures). We give an example of the stretch to apply to the boundary so to see horizontally aligned configurations in our diagrams. Notice that it is a chosen convention. Namely, by a stretch of the boundary one can descend the entire red part to the bottom of the picture:
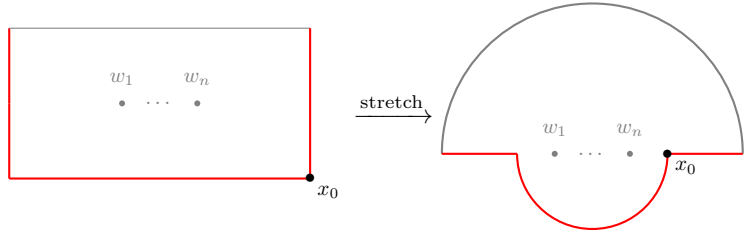
\begin{equation*} 
\vcenter{\hbox{\begin{tikzpicture}[scale=0.5, every node/.style={scale=0.8},decoration={
    markings,
    mark=at position 0.5 with {\arrow{>}}}
    ]
\node (w1) at (-1,0) {};
\node (w2) at (-1,0) {};
\node[gray] at (0.0,0.0) {\ldots};
\node (wn1) at (1,0) {};
\node (wn) at (1,0) {};

\node[gray] at (w1)[above=5pt] {$w_1$};
\node[gray] at (wn)[above=5pt] {$w_n$};
\foreach \n in {w1,wn}
  \node at (\n)[gray,circle,fill,inner sep=1pt]{};

\draw[thick,red] (-4,0) -- (-4,2);
\draw[thick,red] (-4,0) -- (-4,-2);
\draw[thick, red] (-4,-2) -- (4,-2);
\draw[gray] (4,2) -- (-4,2);
\draw[thick, red] (4,-2) -- (4,2);

\node[below right] at (4,-2) {$x_0$};
\node at (4,-2) {$\bullet$};
\end{tikzpicture}}}
\xrightarrow{\str}
\vcenter{\hbox{\begin{tikzpicture}[scale=0.5, every node/.style={scale=0.8},decoration={
    markings,
    mark=at position 0.5 with {\arrow{>}}}
    ]
\node (w1) at (-1,-2) {};
\node (w2) at (-1,0) {};
\node[gray] at (0.0,-2) {\ldots};
\node (wn1) at (1,0) {};
\node (wn) at (1,-2) {};

\node[gray] at (w1)[above=5pt] {$w_1$};
\node[gray] at (wn)[above=5pt] {$w_n$};
\foreach \n in {w1,wn}
  \node at (\n)[gray,circle,fill,inner sep=1pt]{};

\coordinate (x0) at (2,-2);
\coordinate  (x0p) at (4,-2);
\draw[thick,gray] (x0p) arc (0:180:4cm);
\draw[thick,red] (x0) arc (0:-180:2cm);
\draw[thick, red] (-4,-2) -- (-2,-2);
\draw[thick, red] (2,-2) -- (4,-2);
\node at (x0) {$\bullet$};
\node[below right] at (x0) {$x_0$};
\end{tikzpicture}}}
\end{equation*}
This stretch transforms diagrams with handles reaching vertically aligned configurations into ones reaching horizontally aligned configurations also aligned with punctures. What is important is that if handles reach the right vertical side of the boundary in an initial diagram, the configuration it reaches will have real part greater than those of the punctures after the stretch. And passing from the right vertical side to the left one winds one half around punctures, which fits with the local system. The lower right point of the boundary in squared diagrams, denoted $x_0$ is indicated in the picture for helping the reader and because it will serve later on.

We let $\CH^e(D_n^\#)$ designate either $\CH(D_n^\#)$ or $\CHbm(D_n^\#)$ in what follows. There is a map:
\[
\CH(D_n^\#) \to \CHbm(D_n^\#)
\]
induced by inclusion of chain complexes. We let:
\[
\overline{\CH}(D_n^\#) := \im \left( \CH(D_n^\#) \to \CHbm(D_n^\#) \right).
\]
Again even when modules are dual, this is not an isomorphism. 


\subsection{Action of $\Uqgm$}\label{S:action_of_Uqgm}

Here we put an action of the algebra $\CH$ involved in Theorem~\ref{T:homological_version_for_Uqgm} on the space $\CHbm(D_n^\#)$. Composed with the algebra morphism of Theorem~\ref{T:homological_version_for_Uqgm}, this action results in a $\Uqgm$-module structure. Hence we define:
\begin{equation}\label{E:module_on_Uqgm}
\CH \otimes \CHbm(D_n^\#) \to \CHbm(D_n^\#),
\end{equation}
simply by stacking diagrams. We represent it in the following figure, again it is a sufficient definition since bases for all involved modules are given by diagrammatic classes and it is sufficient to define the action on given bases, but the reader can guess how to define it without talking about diagrams in the spirit of Section~\ref{sec_product}. We also mention \cite[(16)]{Pierre} for a definition withou diagrams. Let $\mathcal{D}_{c_1} \in \CH_{c_1}$ and $\mathcal{D}_{c_2} \in \CHbm_{c_2}(D_n^\#)$ both given by diagrams, the action of the first on the second results in the following:
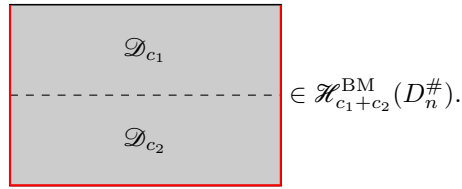
\[
\vcenter{\hbox{
\begin{tikzpicture}[scale=0.3,decoration={
    markings,
    mark=at position 0.3 with {\arrow{>}}}]
\draw[thick] (0,0) rectangle ++(12,8);
\draw[fill = black!20!white] (0,0) rectangle ++(12,8);
\draw[dashed] (0,4)--(12,4);
\draw[red,thick] (0,0)--(0,8);
\draw[red,thick] (12,0)--(12,8);
\draw[red,thick] (0,0)--(12,0);
\node at (6,6) {$\mathcal{D}_{c_1}$};
\node at (6,2) {$\mathcal{D}_{c_2}$};
\end{tikzpicture}}} \in \CHbm_{c_1+c_2}(D_n^\#).
\]
Notice that the resulting picture is again a diagram in a disk with $n$ punctures (resp. holes), where punctures (resp. holes) are furnished by below's diagram. The subset $\partial^-D_n$ is also unchanged, and if both diagrams were defining homology classes, the resulting one naturally defines one in $\CHbm_{c_1+c_2}(D_n^\#)$. 

\begin{prop}
The described action well defines an $\CHbm$-module structure on $\CHbm(D_n)$. The action is graded by $\Coloring_{\Pi}$. Hence $\CHbm(D_n^\#)$ is a $\Uqgm$-module with coefficients in $\Laurent_{\boldsymbol{L}}$. 
\end{prop}
\begin{proof}
The associativity of the action together with the unit is as in Section~\ref{sec_product}, namely it is straightforward from the associativity of the stacking operation, and of the fact that stacking the empty diagram does not change the initial one. Notice that an algebraic operation is hidden, namely $\CHbm$ needs first to have coefficients extended to $\Laurent_{\boldsymbol{L}}$ so for the map to be well defined. Hidden is $\CHbm \otimes \Laurent_{\boldsymbol{L}}$ were this is given by the obvious injection of $\Laurent \subset \Laurent_{\boldsymbol{L}}$. The graded property is obvious from the definition, and the $\Uqgm$-module structure is obtained by precomposing the $\CH$-action by the algebra morphism of Theorem~\ref{T:homological_version_for_Uqgm}. 
\end{proof}

To work in the non Borel--Moore case, one has to restrict to an $\overline{\CH}$ (namely Borel--Moore classes in $\CHbm$ coming from non Borel--Moore ones). 

\begin{prop}
$\overline{\CH}(D_n^\#)$ is is a $\overline{\CH}$-module with coefficients in $\Laurent_{\boldsymbol{L}}$.
\end{prop}
\begin{proof}
It is straightforward to remark that $\overline{\CH}(D_n^\#)$ is stable under $\overline{\CH}$ since both correspond to Borel--Moore classes arising from non Borel--Moore ones. 
\end{proof}

Recall that $\overline{\CH}$ is isomorphic to $\Uqgm$ as $\mathbb{Q}(q)$-algebras thanks to Theorem~\ref{T:Uqgm_is_homological}.

\subsection{Actions of half integral versions of $\Uqg$}\label{S:actions_of_half_integral_Uq}

We recall that in Theorem~\ref{T:homological_version_for_Uqgm} and \ref{T:Uqgm_is_homological} only the strictly negative part of $\Uqg$ denoted $\Uqgm$ was homologically recovered. We will now add an action of the positive Borel that completes $\Uqg$ by means of homological operators on $\CHe(D_n^\#)$. In what follows $\alpha = \alpha_i \in \Pi$ will designate any simple root. 

\subsubsection{Borel--Moore homological modules on a half integral quantum group}\label{S:BMaction_of_half_Uq}

In this section, we add the action of a single $E$-type operator (instead of one for each simple root, which correspond to the usual $\Uqg$). Then we truncate it to get one generator for each simple root. This case is a generalization of the description at the beginning of \cite[Sec.~2.3.2]{JulesMarco}, the reader will find more details there since definitions adapt almost word by word.

For every $n \ge 0$, let us consider the triple $\Conf_c(D_n^\#) \supset S(D_n) \supset Y^2(D_n)$, where
\[
 Y^2(D_n) := \{ \underline{x} = \{ x_1,\ldots,x_n \} \in S(D_n) | \exists 1 \le i < j \le m_c \quad x_i, x_j \in \partial_- D_n \}.
\]
Now let $x_0$ be the lower rightmost point of $\partial^-D_n$. We recall the long exact sequence of the triple:
\begin{center}
 \begin{tikzpicture}[descr/.style={fill=white}]
  \node[anchor=west] (P1) at (0,0) {$\cdots$};
  \node (P2) at (4,0) {$\Hlf_{m_c}(\Conf_c, Y^2(D_n);\Laurent_c(D_n))$};
  \node (P3) at (8,0) {$\CHbm_c(D_n^\#)$};
  \node (P4) at (12,0) {$\Hlf_{m_c-1}(S(D_n^\#), Y^2(D_n);\Laurent_c(D_n))$};
  \node[anchor=east] (P5) at (16,0) {$\cdots$};
  \draw
  (P1) edge[->] (P2)
  (P2) edge[->] (P3)
  (P3) edge[->] node[above] {\scriptsize $\partial_*$} (P4)
  (P4) edge[->] (P5);
 \end{tikzpicture}
\end{center}
where $\partial_*$ denotes the connection homomophism.

\begin{defn}\label{D:homological_E_half_integral}
 For every $c\in\Coloring_{\Pi}$, the operator
 \[
  \CE : \CHbm_c(D_n^\#) \to \bigoplus_{\alpha \in \Pi,c(\alpha) >0} \CHbm_{c-\alpha}(D_n^\#)
 \]
 is the composition
 \begin{center}
  \begin{tikzpicture}[descr/.style={fill=white}]
   \node (P1) at (0,0) {$\CHbm_c(D_n^\#)$};
   \node (P2) at (5,0) {$\Hlf_{m_c-1}(S(D_n),Z(D_n);\Laurent_c(D_n))$};
   \node (P3) at (10,0) {$\displaystyle{\bigoplus_{\alpha \in \Pi,c(\alpha) >0}} \CHbm_{c-\alpha}(D_n^\#),$};
   \draw
   (P1) edge[->] node[above] {\scriptsize $(-1)^{n-1} \partial_*$} (P2)
   (P2) edge[->] node[above] {\scriptsize $\mathrm{del}_{{x}_0}$} (P3);
  \end{tikzpicture}
 \end{center}
 where $\mathrm{del}_{{x}_0}$ is exactly that of \cite[Def.~2.14]{JulesMarco}.
 
We also define $\CE_\alpha$ to be $\left( \prod_{p=1,\ldots, n} s_{p,\alpha}^{-1} \right) p_{c-\alpha} \circ \CE$, where $p_{c-\alpha}$ is the projection onto the term $\CHbm_{c-\alpha}$ in the direct sum. 
\end{defn}

The operator $\CE$ computes the boundary map $\partial_*$, followed by some standard homological identifications. This identification roughly speaking allows to remove a coordinate of the configuration that is in $\partial^-D_n$ when it reaches $x_0$.

\begin{prop}[The fundamental relation]\label{P:fundamental_relation_E}
Since $\CF_\alpha \in \CHbm$, it acts on $\CHbm_c(D_n^\#)$. As operators acting on $\CHbm_c(D_n^\#)$, $\CF_\alpha$ and $\CE$ satisfy the following commutation relation:
\[
\left[\CE_\alpha,\CF^{(k+1)}_\alpha\right] = \CF^{(k+1)}_\alpha \circ \left( q_\alpha^{-\frac{k}{2}}\CK_\alpha - q_\alpha^{\frac{k}{2}}\CK_\alpha^{-1} \right),
\]
where $\CK_\alpha := \left( \prod_{p=1,\ldots, n} s_{p,\alpha}^{-1} \right) \prod_{\beta \in \Pi} q_{\alpha,\beta}^{-c(\beta)} \Id_{\CHbm_c(D_n)}$. 
\end{prop}
\begin{proof}
To compute the commutator we need first to apply $\CF_\alpha^{(k)}$ to a diagram $\mathcal{D}_c \in \CHbm_c(D_n^\#)$. We recall that stacking an element from $\CHbm$ is defined on squared diagrams, and the reader shall think that the transformation $\str$ is applied afterwards sending handles to horizontally aligned configurations. We use squared diagrams since it is easy to compare real parts of configuration points in comparison with others and with punctures: the height compares configuration points one with each other, while the side of the boundary compares configuration points with punctures. 

Using the handle rule (Rem.~\ref{rmk_handle_rule}), we have for every $\mathcal{D}_c \in \CHbm_c(D_n^\#)$:
\[
q^{d_\alpha \frac{k(k-1)}{4}} \CF_\alpha^{(k)} (\mathcal{D}_c)  = \vcenter{\hbox{\begin{tikzpicture}[scale=0.4, every node/.style={scale=0.8},decoration={
    markings,
    mark=at position 0.5 with {\arrow{>}}}
    ]    
    
\coordinate (w0g) at (-5,1) {};
\coordinate (w0d) at (5,1) {};

\coordinate (h1) at (5,-2.5);
\coordinate (h2) at (5,-1.5);
\coordinate (h3) at (5,-1);

\draw[dashed, thick] (w0g) -- (w0d) node[pos=0.15,black] (a1) {} node[pos=0.3,black,above]  {$k$} node[pos=0.55,black] (a2) {} coordinate[pos=0.8] (a3) {};


\coordinate (a33) at (5,2);
\coordinate (a22) at (5,2.5);
\draw[yellow!80!black,double] (a33)--(a33-|a3)--(a3);
\draw[yellow!80!black,thick] (a22)--(a22-|a2)--(a2);

\draw[fill = black!20!white] (-5,-3) rectangle ++(10,3);
\node at (0,-1.5) {$\mathcal{D}_c$};

\draw[red, thick] (-5,-3) -- (-5,3);
\draw[red, thick] (5,-3) -- (5,3);
\draw[red, thick] (-5,-3) -- (5,-3);
\draw[gray, thick] (-5,3) -- (5,3);
\end{tikzpicture}}} =  \prod_{p=1,\ldots, n} s_{p,\alpha}^{2} q_\alpha^{k-1}\prod_{\beta \in \Pi} q_{\alpha,\beta}^{c(\beta)}
\vcenter{\hbox{\begin{tikzpicture}[scale=0.4, every node/.style={scale=0.8},decoration={
    markings,
    mark=at position 0.5 with {\arrow{>}}}
    ]    
    
\coordinate (w0g) at (-5,1) {};
\coordinate (w0d) at (5,1) {};

\coordinate (h1) at (5,-2.5);
\coordinate (h2) at (5,-1.5);
\coordinate (h3) at (5,-1);

\draw[fill = black!20!white] (-5,-3) rectangle ++(10,3);
\node at (0,-1.5) {$\mathcal{D}_c$};

\draw[red, thick] (-5,1) -- (-5,3);
\draw[red, thick] (5,-3) -- (5,3);
\draw[gray, thick] (-5,3) -- (5,3);

\draw[dashed, thick] (w0g) -- (w0d) node[pos=0.15,black] (a1) {} node[pos=0.3,black] (a2) {} node[pos=0.55,above,black] {$k$} coordinate[pos=0.8] (a3) {};

\coordinate (a33) at (5,2);
\draw[yellow!80!black,double] (a33)--(a33-|a3)--(a3);
\draw[yellow!80!black,double] (0.3,-1.5)--(5,-1.5);

\draw[yellow!80!black,very thick] (-5,1)--(-5,-3)--(5,-3);
\end{tikzpicture}}},
\]
and also:
\[
q^{d_\alpha \frac{k(k-1)}{4}} \CF_\alpha^{(k)} (\mathcal{D}_c) =\vcenter{\hbox{\begin{tikzpicture}[scale=0.4, every node/.style={scale=0.8},decoration={
    markings,
    mark=at position 0.5 with {\arrow{>}}}
    ]    
    
\coordinate (w0g) at (-5,1) {};
\coordinate (w0d) at (5,1) {};

\coordinate (h1) at (5,-2.5);
\coordinate (h2) at (5,-1.5);
\coordinate (h3) at (5,-1);

\draw[dashed, thick] (w0g) -- (w0d) node[pos=0.15,black] (a1) {} node[pos=0.3,black,above]  {$k$} node[pos=0.8,black] (a2) {} coordinate[pos=0.4] (a3) {};


\coordinate (a33) at (5,2.5);
\coordinate (a22) at (5,2);
\draw[yellow!80!black,double] (a33)--(a33-|a3)--(a3);
\draw[yellow!80!black,thick] (a22)--(a22-|a2)--(a2);

\draw[fill = black!20!white] (-5,-3) rectangle ++(10,3);
\node at (0,-1.5) {$\mathcal{D}_c$};

\draw[yellow!80!black,double,thick] (0.3,-1.5)--(5,-1.5);

\draw[red, thick] (-5,-3) -- (-5,3);
\draw[red, thick] (5,-3) -- (5,3);
\draw[red, thick] (-5,-3) -- (5,-3);
\draw[gray, thick] (-5,3) -- (5,3);
\end{tikzpicture}}} =  \prod_{\beta \in \Pi} q_{\alpha,\beta}^{-c(\beta)}
\vcenter{\hbox{\begin{tikzpicture}[scale=0.4, every node/.style={scale=0.8},decoration={
    markings,
    mark=at position 0.5 with {\arrow{>}}}
    ]    
    
\coordinate (w0g) at (-5,1) {};
\coordinate (w0d) at (5,1) {};

\coordinate (h1) at (5,-2.5);
\coordinate (h2) at (5,-1.5);
\coordinate (h3) at (5,-1);

\draw[fill = black!20!white] (-5,-3) rectangle ++(10,3);
\node at (0,-1.5) {$\mathcal{D}_c$};

\draw[red, thick] (-5,-3) -- (-5,3);
\draw[red, thick] (5,1) -- (5,3);
\draw[gray, thick] (-5,-3) -- (5,-3);
\draw[gray, thick] (-5,3) -- (5,3);

\draw[dashed, thick] (w0g) -- (w0d) node[pos=0.15,black] (a1) {} node[pos=0.3,black] (a2) {} node[pos=0.55,above,black] {$k$} coordinate[pos=0.3] (a3) {};


\coordinate (a33) at (5,2.75);
\draw[yellow!80!black,double] (a33)--(a33-|a3)--(a3);

\draw[yellow!80!black,double] (4.5,0)--(4.5,2)--(5,2);
\draw[yellow!80!black,double,thick] (0.3,-1.5)--(4.5,-1.5)--(4.5,0);

\draw[yellow!90!black, very thick] (5,1)--(5,-3);
\end{tikzpicture}}},
\]
We emphasize the yellow handle paths in diagrams, they are important for applying the handle rule, they relate diagrams to vertically aligned configurations sitting on the right boundary. First we have used a ray representing the way the diagram $\mathcal{D}_c$ is attached to a vertically aligned configuration on the boundary. Of course it is a representation and the handles can go anywhere in the diagram, we will only change them locally near the boundary. On left hand sides diagrams we have a ray of $k-1$ parallel handles drawn and one single (leftmost on the first, rightmost on the second) relating to the dashed diagram indexed by $k$ where the black colored is used to represent $\alpha$. Only the single handle moves, the ray stays fixed. 

Now, just like in the proof of \cite[Prop.~2.19]{JulesMarco}, since the class $q^{d_\alpha \frac{k(k-1)}{4}} \CF_\alpha^{(k)} (\mathcal{D}_c)$ corresponds to a cartesian product of the simplex supported by the dashed arc and the one described by $\mathcal{D}_c$, and thanks to the compatibility rule of the cartesian product with $\partial_*$,  we have:
\[
\CE \circ q^{d_\alpha \frac{k(k-1)}{4}} \CF_\alpha^{(k)} (\mathcal{D}_c) = q^{d_\alpha \frac{k(k-1)}{4}} \CF_\alpha^{(k)} \circ \CE(\mathcal{D}_c) + \left(\prod_{\beta \in \Pi} q_{\alpha,\beta}^{-c(\beta)}  - \left( \prod_{p=1,\ldots, n} s_{p,\alpha}^{2} \right) q_\alpha^{k-1} \prod_{\beta \in \Pi} q_{\alpha,\beta}^{c(\beta)} \right) \vcenter{\hbox{\begin{tikzpicture}[scale=0.4, every node/.style={scale=0.8},decoration={
    markings,
    mark=at position 0.5 with {\arrow{>}}}
    ]    
    
\coordinate (w0g) at (-5,1) {};
\coordinate (w0d) at (5,1) {};

\coordinate (h1) at (5,-2.5);
\coordinate (h2) at (5,-1.5);
\coordinate (h3) at (5,-1);

\draw[dashed, thick] (w0g) -- (w0d) node[pos=0.15,black] (a1) {} node[pos=0.3,black,above]  {$k-1$} node[pos=0.55,black] (a2) {} coordinate[pos=0.8] (a3) {};


\coordinate (a33) at (5,2);
\coordinate (a22) at (5,2.5);
\draw[yellow!80!black,double] (a33)--(a33-|a3)--(a3);

\draw[fill = black!20!white] (-5,-3) rectangle ++(10,3);
\node at (0,-1.5) {$\mathcal{D}_c$};

\draw[red, thick] (-5,-3) -- (-5,3);
\draw[red, thick] (5,-3) -- (5,3);
\draw[red, thick] (-5,-3) -- (5,-3);
\draw[gray, thick] (-5,3) -- (5,3);
\end{tikzpicture}}},
\]
so that the following holds when operators are restricted to $\CHbm_c(D_n^\#)$:
\begin{align*}
\left[\CE,q^{d_\alpha \frac{k(k-1)}{4}}\CF^{(k)}_\alpha\right] & = q^{d_\alpha \frac{(k-1)(k-2)}{4}}\CF^{(k-1)}_\alpha \left(\prod_{\beta \in \Pi} q_{\alpha,\beta}^{-c(\beta)}  - q_\alpha^{k-1} \left( \prod_{p=1,\ldots, n} s_{p,\alpha}^{2} \right) \prod_{\beta \in \Pi} q_{\alpha,\beta}^{c(\beta)} \right) \Id_{\CHbm_c(D_n)} 
\end{align*}
which gives the result after multiplication by $q^{-d_\alpha\frac{k(k-1)}{4}}$ and replacing $\CE$ by $\CE_\alpha$ which brings a coefficient $\prod_{p=1,\ldots, n} s^{-1}_{p,\alpha}$ to the equation.
\end{proof}

Let $\Uqg^{E}$ denote the $\BZ[q^{\pm 1}]$-algebra generated by generators $E,K_\alpha^{\pm 1},F_\alpha^{(k)}$ for any $\alpha \in \Pi$ and any $k \in \BN$, and relations:
\begin{gather*}
 F_\alpha^{(k)} F_\alpha^{(l)} = \qbin{k+l}{k}_{q_\alpha^{\frac{k}{2}}} F_\alpha^{(k+l)}, \qquad K_\alpha K_\alpha^{-1} = K_\alpha^{-1} K_\alpha = 1, \\*
 K_{\alpha_i} E_{\alpha_j} K_{\alpha_i}^{-1} = q_{\alpha_i}^{-\frac{a_{i,j}}{2}} E_{\alpha_j}, \qquad K_{\alpha_i} F_{\alpha_j}^{(k)} K_{\alpha_i}^{-1} = q_{\alpha_i}^{k\frac{a_{i,j}}{2}} F_{\alpha_j}^{(k)}, \qquad
 [E_\alpha,F_\alpha^{(k+1)}] = F_\alpha^{(k)}(q_\alpha^{-\frac{k}{2}}K - q_\alpha^{\frac{k}{2}} K^{-1}).
\end{gather*}
and the quantum Serre relations (only) for $F$'s:
\begin{equation}
\sum_{l=0}^{1-a_{i,j}} (-1)^l F_{\alpha}^{(l)} F_{\beta}^{(1)} F_{\alpha}^{(1-a_{i,j})-l)} = 0.
\end{equation}

\begin{theorem}\label{T:half_integral_modules}
The $\Laurent(\boldsymbol{L})$-module $\CHbm(D_n^\#)$ is endowed with a $\Uqg^{E}$ structure, where the action is given by sending generators of $\Uqg$ to their calligraphic analogs in homology. Furthermore, there is an injection:
\[
\begin{array}{ccc}
\CHbm(D_n^\circ) & \to & \CHbm(D_n).
\end{array}
\]
\end{theorem}
\begin{proof}
We verify that calligraphic generators satisfy the relations of $\Uqg^E$. The first relation is straightforward iteration of the divided power property Prop.~\ref{prop_DividedPowers}. The second is straightforward. Next two are easily checked from source and target spaces of $\CE_\alpha$ (and $\CF_\alpha$) that is: $\CH_c \to \CH_{c-\alpha}$ (resp. the converse) and of the diagonal action of $\CK_\alpha$. The last one is the fundamental relation for $\CE$ at homology, Prop.~\ref{P:fundamental_relation_E}.

The change of basis between generators $\CF^\circ$ and $\CF$ from Prop.~\ref{P:structure_result_punctures} could be directly deduced from \cite[Prop.~7.4]{Jules_Verma} (removing the quantum factorials in the formula). It corresponds to a diagonal change of basis, while diagonal coefficients are not invertible in $\Laurent_{\boldsymbol{L}}$, hence it is only an injection while working on rings. 
\end{proof}

This theorem when $\mathfrak{g}=\slt$ is precisely \cite[Theorem~1]{Jules_Verma}. It is thus a wide generalization of the latter to all semisimple Lie algebras. We call these modules, half integral Verma and coVerma modules. 

\subsubsection{Left adjoint homological action to that of the half integral quantum group}\label{standard_action_of_half_Uq}

We first add an action of $\Uqg^F$ on $\CH(D_n)$ that is left adjoint to that of $\Uqg^E$ described in the previous section, regarding the duality $\langle \cdot , \cdot  \rangle$ from Prop.~\ref{P:structure_result_punctures_standard}. We generalize to arbitrary Lie algebras (and to integral Laurent polynomials coefficients with $q$ generic) \cite[Theorem~5.9]{Pierre}. It has much inspired our definition of homological operators in this section. 

Let $\Uqg^{F}$ denote the $\BZ[q^{\pm 1}]$-algebra generated by generators $E_\alpha^{(k)},K_\alpha^{\pm 1},F_\alpha^{(1)}$ for any $\alpha \in \Pi$ and any $k \in \BN$, and relations:
\begin{gather*}
 E_\alpha^{(k)} E_\alpha^{(l)} = \qbin{k+l}{k}_{q_\alpha^{\frac{k}{2}}} E_\alpha^{(k+l)}, \qquad K_\alpha K_\alpha^{-1} = K_\alpha^{-1} K_\alpha = 1, \\*
 K_{\alpha_i} E^{(k)}_{\alpha_j} K_{\alpha_i}^{-1} = q_{\alpha_i}^{-k\frac{a_{i,j}}{2}}E^{(k)}_{\alpha_j}, \qquad K_{\alpha_i} F_{\alpha_j}^{(1)} K_{\alpha_i}^{-1} = q_{\alpha_i}^{\frac{a_{i,j}}{2}} F_{\alpha_j}^{(1)}, \qquad
 [E^{(k+1)}_\alpha,F_\alpha^{(1)}] = E_\alpha^{(k)}(q_\alpha^{-\frac{k}{2}}K - q_\alpha^{\frac{k}{2}} K^{-1}).
\end{gather*}
and the quantum Serre relations for $E$'s:
\begin{equation}
\sum_{l=0}^{1-a_{i,j}} (-1)^l E_{\alpha}^{(l)} E_{\beta}^{(1)} E_{\alpha}^{(1-a_{i,j})-l)} = 0.
\end{equation}

We follow definitions from (20) in \cite{Pierre}. We let:
\begin{align}
Y^i_c := \{ \boldsymbol{z} \in \Conf_c(D_n) , | \boldsymbol{z} \cap \partial^- D_n | \ge i \} \\
Y^i_c(\alpha) := \{ \boldsymbol{z} \in \Conf_c(D_n) , | \boldsymbol{z^\alpha} \cap \partial^- D_n | \ge i \} \\
Y^+_c = T(D_n) := \{ \boldsymbol{z} \in \Conf_c(D_n) , |\boldsymbol{z} \cap \partial^+ D_n | \ge 1 \} . 
\end{align}
such that $Y^i_c(\alpha) \subset Y^i_c$. We will use $Y^+_c$ when colors will shift. There is a map induced by restriction at the compactly supported cohomology:
\[
\Hnot^{*}_{\mathrm{compact}}( \Conf_c(D_n) , T(D_n) ; \Laurent_c(D_n) ) \to \Hnot^{*}_{\mathrm{compact}} (Y^i_c(\alpha) , Y^i_c(\alpha) \cap T(D_n) ) .
\]
Notice that:
\[
\partial Y^i_c(\alpha) = \left( Y^i_c(\alpha)\cap T(D_n) \right) \cup \left(Y^i_c(\alpha)\cap Y^{i+1}_c \right).
\]
so that we can use the Poincaré duality for relative homology that requires a decomposition of the boundary (see Appendix~\ref{A:the_pairing}). From isomorphisms between homology and cohomology, we get the Poincaré dual map:
\begin{equation}\label{E:i_boundary}
\partial^i_c(\alpha) : \Hnot_{m_c}( \Conf_c(D_n) , S(D_n) ; \Laurent_c(D_n) ) \to \Hnot_{m_c-i} (Y^i_c(\alpha) , Y^i_c(\alpha)\cap Y^{i+1}_c; \Laurent_c  ).
\end{equation}
where the shift in dimensions is due to that of $Y^i_c(\alpha)$ which is $2m_c-i$. 

There is an isomorphism:
\[
\del_i(\alpha) : \Hnot_{m_c-i} (Y^i_c(\alpha) , Y^i_c(\alpha)\cap Y^{i+1}_c; \Laurent_c  ) \simeq \CH_{c-i\alpha}(D_n).
\]
We rephrase the argument in \cite{Pierre} proving that above homologies are isomorphic. The pair $(Y^i_c(\alpha) , Y^i_c(\alpha)\cap Y^{i+1}_c)$ is Poincaré dual to $(Y^i_c(\alpha) , Y^i_c(\alpha)\cap Y^+_c)$. Now $Y^i_c(\alpha) \setminus (Y^i_c(\alpha)\cap Y^+_c)$ is homeomorphic to $\Conf_{c-i\alpha}(D_n) \times \Conf_{i\alpha}{\partial^-D_n}$, where $\Conf_{i\alpha}{\partial^-D_n} \simeq \Conf_i((0,1))$. It proves:
\begin{align*}
\Hnot^\bullet_{\mathrm{compact}}(Y^i_c(\alpha) , Y^i_c(\alpha)\cap Y^+_c)&  \simeq \Hnot^\bullet_{\mathrm{compact}}(Y^i_c(\alpha) \setminus Y^i_c(\alpha)\cap Y^+_c) \\
& \simeq \Hnot^\bullet_{\mathrm{compact}}(\Conf_{c-i\alpha}(D_n) \setminus T(D_n)) \\
& \simeq \Hnot^\bullet_{\mathrm{compact}}(\Conf_{c-i\alpha}(D_n) , T(D_n))
\end{align*}
where the first and last equalities are a standard property of the Borel--Moore cochain complex, second is due to the mentioned homeomorphism. It gives the desired isomorphism at the Poincaré dual level.

\begin{defn}
For $\alpha \in \Pi$, we define the $k$-th divided power of $\CE_\alpha^{[1]}$ to be a map:
\[
\CE_\alpha^{[k]} : \CH_c(D_n) \to \CH_{c-k\alpha}(D_n)
\]
defined by the composition $q^{-d_{\alpha} \frac{k(k-1)}{4}}  \mathrm{del}_{i}(\alpha) \circ \partial^i_c(\alpha)$. 
\end{defn}

We also define $\CF^{[1]}_\alpha$ to be the action of $\left( \prod_{p=1,\ldots, n} s_{p,\alpha} \right) \CF^{(1)}_\alpha \in \CH$ on $\CH_c(D_n)$ by first noticing that the diagram defining $\CF^{(1)}_\alpha$ also defines a non Borel--Moore class. 

\begin{prop}
We recall the perfect pairing from Proposition~\ref{P:structure_result_punctures_standard}:
\[
\langle \cdot , \cdot \rangle : \CH(D_n)   \times \CHbm(D_n) \to \Laurent(\boldsymbol{L}). 
\]
For this form, $\CK_\alpha$ is left adjoint to $\CK_\alpha^{-1}$, $\CE^{[k]}_\alpha$ is left adjoint to $\CF^{(k)}_\alpha$ and $\CF^{[1]}_\alpha$ is left adjoint to $\CE_\alpha$. 
\end{prop}
\begin{proof}
This proof is adapted from \cite[Theorem~5.9]{Pierre}. We first explain why $\CE^{[k]}_\alpha$ and $\CF^{(k)}_\alpha$ are adjoint operators. Let $c\in \Coloring_\Pi$, the action of $\CF^{(l)}_\alpha$ composed by $\flip$ gives:
\[
\flip\left( q^{d_\alpha \frac{k(k-1)}{4}} \CF_\alpha^{(k)} (\mathcal{D}_c) \right)  = \vcenter{\hbox{\begin{tikzpicture}[scale=0.4, every node/.style={scale=0.8},decoration={
    markings,
    mark=at position 0.5 with {\arrow{>}}}
    ]    
    
\coordinate (w0g) at (-5,-1) {};
\coordinate (w0d) at (5,-1) {};

\coordinate (h1) at (5,-2.5);
\coordinate (h2) at (5,-1.5);
\coordinate (h3) at (5,-1);

\draw[dashed, thick] (w0g) -- (w0d) node[pos=0.15,black] (a1) {} node[pos=0.5,black,above]  {{$k$}} node[pos=0.55,black] (a2) {} coordinate[pos=0.7] (a3) {};


\coordinate (a33) at (5,-2);
\coordinate (a22) at (5,2.5);
\draw[yellow!80!black,double] (a33)--(a33-|a3)--(a3);

\draw[fill = black!20!white] (-5,0) rectangle ++(10,3);
\node at (0,1.5) {\reflectbox{\rotatebox{180}{$\mathcal{D}_c$}}};

\draw[gray, thick] (-5,-3) -- (-5,3);
\draw[gray, thick] (5,-3) -- (5,3);
\draw[red, thick] (-5,-3) -- (5,-3);
\draw[gray, thick] (-5,3) -- (5,3);
\end{tikzpicture}}} .
\]
(notice that a stretch has been applied to the boundary so that $\partial^-D_n$ in red only stands on the bottom part of $\partial D_n$). 
Hence $\flip \circ \CF^{(k)}$ defines a map:
\[
\Hlf_c\left( \Conf_c(D_n), T(D_n), \Laurent_c(D_n) \right) \to \Hlf_{c+k\alpha} \left( \Conf_{c+k\alpha}(D_n), T(D_n), \Laurent_{c+k\alpha}(D_n) \right)
\]
which factors through $Y^k_{c+k\alpha}(\alpha)$ (by a homotopy shrinking the white part of the above diagram, and sending the dashed arc to the lower (and red) part of the boundary). It gives the following commutative diagram:

 \begin{center}
  \begin{tikzpicture}[descr/.style={fill=white}]
   \node (P1) at (-2.5,0) {$\Hlf_{c+k\alpha} \left( Y^k_{c+k\alpha}(\alpha), Y^k_{c+k\alpha}(\alpha) \cap Y^+_{c+k\alpha}, \Laurent_{c+k\alpha}(D_n) \right)$};
   \node (P2) at (4,1.5) {$\Hlf_c\left( \Conf_c(D_n), T(D_n), \Laurent_c(D_n) \right)$};
   \node (P3) at (4,-1.5) {$\Hlf_{c+k\alpha} \left( \Conf_{c+k\alpha}(D_n), T(D_n), \Laurent_{c+k\alpha}(D_n) \right)$};
   \draw
   (P2) edge[->] node[above] {$r$} (P1)
   (P1) edge[->] node[above right] {$\iota$} (P3)
   (P2) edge[->] node[right] {$\flip \circ \CF^{(k)}$} (P3);
  \end{tikzpicture}
 \end{center}
 where $r$ is the gluing of the dashed arc pushed to the boundary, and $\iota$ is the homological map induced by the inclusion of $Y^k_{c+k\alpha}(\alpha)$ in $\Conf_{c+k\alpha}(D_n)$. Notice that the dual map to the inclusion yields the restriction on cochains which shows that the dual diagram (passed to homology thanks to the perfect pairing from Prop.~\ref{P:structure_result_punctures_standard}) is the following:
 \begin{center}
  \begin{tikzpicture}[descr/.style={fill=white}]
   \node (P1) at (-3.5,0) {$\Hnot_{c+k\alpha} \left( Y^k_{c+k\alpha}(\alpha), Y^{k}_{c+k\alpha}(\alpha) \cap Y^{k+1}_{c+k\alpha}, \Laurent_{c+k\alpha}(D_n) \right)$};
   \node (P2) at (4,1.5) {$\Hnot_c\left( \Conf_c(D_n), S(D_n), \Laurent_c(D_n) \right)$};
   \node (P3) at (4,-1.5) {$\Hnot_{c+k\alpha} \left( \Conf_{c+k\alpha}(D_n), S(D_n), \Laurent_{c+k\alpha}(D_n) \right)$};
   \draw
   (P2) edge[<-] node[above] {$\del_i(\alpha)$} (P1)
   (P1) edge[<-] node[above right] {$\partial^i_c(\alpha)$} (P3)
   (P2) edge[<-] node[below] {} (P3);
  \end{tikzpicture}
 \end{center}

Then one notes that $r$ and $\del_i(\alpha)$ are inverse isomorphisms by remarking that gluing the dashed arc is exactly the inverse of the deleting of $\Conf_{i\alpha}(\partial^- D_n)$ that is performed in $\del_i(\alpha)$. It proves that $\flip \circ \CF^{(k)}$ and $\CE^{[k]}_\alpha = \mathrm{del}_{x_0}(\alpha) \circ \partial^i_c(\alpha)$ are dual operators and hence are adjoint for \eqref{E:the_pairing_punctures} (since they have same renormalization factor). 

Now to prove that $\CF^{[1]}_\alpha$ and $\CE_\alpha$ are also adjoint for \eqref{E:the_pairing_punctures}, we do the same starting with the standard homology, and we obtain the two following dual commutative diagrams. 

 \begin{center}
  \begin{tikzpicture}[descr/.style={fill=white}]
   \node (P1) at (-2.5,0) {$\Hnot_{c+\alpha} \left( T_\alpha(D_n), T_\alpha(D_n) \cap S(D_n), \Laurent_{c+k\alpha}(D_n) \right)$};
   \node (P2) at (4,1.5) {$\Hnot_c\left( \Conf_c(D_n), S(D_n), \Laurent_c(D_n) \right)$};
   \node (P3) at (4,-1.5) {$\Hnot_{c+k\alpha} \left( \Conf_{c+\alpha}(D_n), S(D_n), \Laurent_{c+\alpha}(D_n) \right)$};
   \draw
   (P2) edge[->] node[above] {$r'$} (P1)
   (P1) edge[->] node[above right] {} (P3)
   (P2) edge[->] node[right] {$\CF^{[1]}$} (P3);
  \end{tikzpicture}
 \end{center}
 where $T_\alpha(D_n)$ is made of configurations with (at least) one $\alpha$-colored coordinate in $\partial^+ D_n$. Its dual is:
 \begin{center}
  \begin{tikzpicture}[descr/.style={fill=white}]
   \node (P1) at (-3.5,0) {$\Hlf_{c+\alpha} \left( T_\alpha(D_n), T_\alpha(D_n) \cap T^2_\alpha(D_n), \Laurent_{c+\alpha}(D_n) \right)$};
   \node (P2) at (4,1.5) {$\Hlf_c\left( \Conf_c(D_n), T(D_n), \Laurent_c(D_n) \right)$};
   \node (P3) at (4,-1.5) {$\Hlf_{c+\alpha} \left( \Conf_{c+\alpha}(D_n), T(D_n), \Laurent_{c+\alpha}(D_n) \right)$};
   \draw
   (P2) edge[<-] node[above] {$\del_{x_0}$} (P1)
   (P1) edge[<-] node[above right] {$\partial_\alpha$} (P3)
   (P2) edge[<-] node[right] {$\CE_\alpha \circ \flip$} (P3);
  \end{tikzpicture}
 \end{center}
 where $T^2_\alpha(D_n)$ is made of configurations with (at least) two $\alpha$-colored coordinate in $\partial^+ D_n$, recalling that the map $\partial_\alpha$ involved is dual to restriction. 

Now $K$ and $K^{-1}$ being adjoint is straightforward since the map $\flip$ obviously conjugates the local system. 
\end{proof}

\begin{coro}\label{T:adjoint_modules_to_half_integral}
The $\Laurent_{\boldsymbol{L}}$-module $\CH(D_n)$ is a module on $\Uqg^F$. 
\end{coro}
\begin{proof}
It is a direct consequence of Theorem~\ref{T:half_integral_modules} and of the previous proposition. In particular, we get the quantum Serre relations for $E$'s on the homology side from those for $F$'s on the Borel--Moore side. 
\end{proof}

\subsection{Modules on $\Uqg$}\label{S:action_whole_Uq}

\subsubsection{The Verma module}\label{S:Verma_module}

We denote by $\Uqg$ the Drinfel'd--Jimbo quantum group associated with $\mathfrak{g}$ and we recall that it is the $\mathbb{Q}(q)$-algebra generated by generators $E_\alpha,K_\alpha^{\pm 1},F_\alpha$ for any $\alpha \in \Pi$ and any $k \in \BN$, and relations:
\begin{gather*}
 K_\alpha K_\alpha^{-1} = K_\alpha^{-1} K_\alpha = 1, \\*
 K_{\alpha_i} E_{\alpha_j} K_{\alpha_i}^{-1} = q_{\alpha_i}^{-\frac{a_{i,j}}{2}}E_{\alpha_j}, \qquad K_{\alpha_i} F_{\alpha_j}^{(1)} K_{\alpha_i}^{-1} = q_{\alpha_i}^{\frac{a_{i,j}}{2}} F_{\alpha_j}^{(1)}, \qquad
 [E_\alpha,F_\alpha^{(1)}] = (q_\alpha^{-\frac{1}{2}}K - q_\alpha^{\frac{1}{2}} K^{-1}).
\end{gather*}
and the quantum Serre relations for $E$'s and $F$'s:
\begin{align*}
& \sum_{l=0}^{1-a_{i,j}} (-1)^l F_{\alpha}^{(l)} F_{\beta}^{(1)} F_{\alpha}^{(1-a_{i,j})-l)} = 0 \\
& \sum_{l=0}^{1-a_{i,j}} (-1)^l E_{\alpha}^{(l)} E_{\beta}^{(1)} E_{\alpha}^{(1-a_{i,j})-l)} = 0
\end{align*}
where we have fixed $F_\alpha^{(i)} := (q_\alpha - q_\alpha^{-1})^{i} \frac{F_\alpha^i}{[i]_{q_\alpha^{\frac{1}{2}}}!}$ and $E_\alpha^{(i)} := (q_\alpha - q_\alpha^{-1})^{i} \frac{E_\alpha^i}{[i]_{q_\alpha^{\frac{1}{2}}}!}$. 

\begin{theorem}\label{T:full_Uqg_module}
The space $\overline{\CH}(D_n)$ is a $\Uqg$-module. The action is provided by sending generators to their calligraphic analogs.
\end{theorem}
\begin{proof}
First coefficients of $\overline{\CH}(D_n)$ need being extended to the field by injecting:
\[
\Laurent_{\mathcal{L}} \to \mathbb{Q}(q) [\boldsymbol{L}^{\pm 1} ]
\]
where $\mathbb{Q}(q) [\boldsymbol{L}^{\pm 1} ]$ means Laurent polynomials in variables from $\boldsymbol{L}$ and with coefficients in $\mathbb{Q}(q)$. 

Notice that $\overline{\CH}(D_n)$ is endowed with actions of $\CE_\alpha$, $\CF_\alpha^{(1)}$ and $K_\alpha^{\pm 1}$ inherited from those involved in Theorem~\ref{T:half_integral_modules}. All relations but quantum Serre ones for $E$'s are satisfied due to the latter result. Now thanks to Theorem~\ref{T:adjoint_modules_to_half_integral}, the quantum Serre relations for $E$'s are also satisfied. 
\end{proof}


\begin{prop}
The module $\overline{\CH}(D_1^\circ)$ is the Verma module of $\Uqg$. 
\end{prop}
\begin{proof}
From Proposition~\ref{P:structure_result_punctures_standard} we observe that $\overline{\CH}(D_n^\circ)$ is $\bigoplus_{c \in \Coloring_\Pi} \CH_c v_0$ where $v_0$ is the empty diagram. Now $\CK_\alpha v_0 = s^1_\alpha$ and this choice of character defines the highest weight of the Verma module (see Def.~7 in \cite[sec.~6.2.5]{KS}). It proves the claim. 
\end{proof}

\subsection{Monoidality and braiding}\label{S:braiding_and_monoidality}

\subsubsection{Monoidality}\label{S:monoidality}

We recall that $\Uqg$ is a Hopf algebra endowed with a coproduct making its category of modules monoidal. In this section, we show $\CHbm(D_n) = \CHbm(D_1)^{\otimes n}$ where the quantum group action is given on the right by coproduct. To prove this fact we will need to split the puncture disk into two. Let $D_{1,N}$ be the left part of $D_n$ and $D_{N+1,n}$ the right part, when splitting it vertically in between punctures $N$ and $N+1$ with $1 \le N < n$. Namely $D_{1,N}$ is a disk with punctures $w_1 ,\ldots, w_N$, while $D_{N+1,n}$ is a disk with punctures $w_{N+1}, \ldots, w_n$. Consider a diagram $\mathcal{D} \in \CHbm(D_n)$ with the following shape:
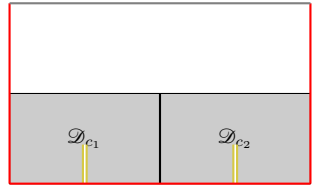
\[
\vcenter{\hbox{\begin{tikzpicture}[scale=0.4, every node/.style={scale=0.8},decoration={
    markings,
    mark=at position 0.5 with {\arrow{>}}}
    ]       
\coordinate (w0g) at (-5,1) {};
\coordinate (w0d) at (5,1) {};
\coordinate (h1) at (5,-2.5);
\coordinate (h2) at (5,-1.5);
\coordinate (h3) at (5,-1);
\draw[fill = black!20!white] (-5,-3) rectangle ++(10,3);
\node at (-2.5,-1.5) {$\mathcal{D}_{c_1}$};
\node at (2.5,-1.5) {$\mathcal{D}_{c_2}$};
\draw[\hyellow, double, thick] (-2.5,-1.7) -- (-2.5,-3);
\draw[\hyellow, double, thick] (2.5,-1.7) -- (2.5,-3);
\draw[thick] (0,-3)--(0,0);
\draw[red, thick] (-5,-3) -- (-5,3);
\draw[red, thick] (5,-3) -- (5,3);
\draw[red, thick] (-5,-3) -- (5,-3);
\draw[gray, thick] (-5,3) -- (5,3);
\end{tikzpicture}}}
\]
Namely:
\begin{itemize}
\item We suppose the entire diagram (pearl necklaces and their handles) concerned with punctures $w_1,\ldots, w_N$ to be contained in the diagram denoted $\mathcal{D}_{c_1}$ and all the ones concerned with punctures $w_{N+1},\ldots, w_n$ to be contained in the diagram denoted $\mathcal{D}_{c_2}$. In other words diagram $\mathcal{D}_{c_1}$ is supported on $D_{1,N}$ and $\mathcal{D}_{c_2}$ on $D_{N+1,n}$. 
\item The handles reach a horizontally aligned configurations on the lower boundary. We claim that if their real part are disjoint from those of the punctures, there is a single obvious vertical line retraction that aligns everyone near the boundary, and this explains how to use the handles. 
\item As the picture suggests, all the handles from $\mathcal{D}_{c_1}$ must have real part lower than those of punctures $w_{N+1}, \ldots, w_n$ and all the ones from $\mathcal{D}_{c_1}$ must have real part greater than those of punctures $w_1, \ldots, w_N$ (at all time). The reader can imagine the puncture to lie right between labels $\mathcal{D}_{c_1}$ and $\mathcal{D}_{c_2}$ in the picture. 
\end{itemize}
We say that the diagram is in split position, and we claim that any diagram can be made into a split position, up to linear combination. One would be convinced from the bases elements which can be put in a split position up to a handle rule, hence up to a diagonal and invertible change of basis. Notice also that the above diagram means below's one:
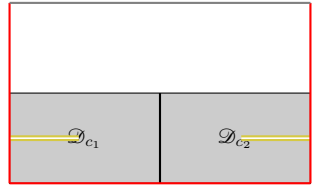
\[
\vcenter{\hbox{\begin{tikzpicture}[scale=0.4, every node/.style={scale=0.8},decoration={
    markings,
    mark=at position 0.5 with {\arrow{>}}}
    ]       
\coordinate (w0g) at (-5,1) {};
\coordinate (w0d) at (5,1) {};
\coordinate (h1) at (5,-2.5);
\coordinate (h2) at (5,-1.5);
\coordinate (h3) at (5,-1);
\draw[fill = black!20!white] (-5,-3) rectangle ++(10,3);
\node at (-2.5,-1.5) {$\mathcal{D}_{c_1}$};
\node at (2.5,-1.5) {$\mathcal{D}_{c_2}$};
\draw[\hyellow, double, thick] (-2.7,-1.5) -- (-5,-1.5);
\draw[\hyellow, double, thick] (2.7,-1.5) -- (5,-1.5);
\draw[thick] (0,-3)--(0,0);
\draw[red, thick] (-5,-3) -- (-5,3);
\draw[red, thick] (5,-3) -- (5,3);
\draw[red, thick] (-5,-3) -- (5,-3);
\draw[gray, thick] (-5,3) -- (5,3);
\end{tikzpicture}}}
\]
up to the application of $\str$, using earlier conventions. In what follows we need to cut in between punctures and so to input handles in between their real parts, this is why we introduce this slightly more complicated convention for diagrams. 

Now, there is an isomorphism of $\Laurent_{\boldsymbol{L}}$-modules:
\[
\splits_N: \begin{array}{ccc}
\CHbm(D_n) & \to & \CHbm(D_{N+1,n}) \otimes \CHbm(D_{1,N})\\
\mathcal{D} & \mapsto &  \mathcal{D_{c_2}} \otimes \mathcal{D_{c_1}}. 
\end{array}
\]
where $\mathcal{D}$ is a diagram in a split position as the one above. Notice the inversion in the order of tensor folds. 

We recall the expression of the coproduct on generators of $\Uqg^E$:
\begin{align*}
& \Delta(E_\alpha) = E_\alpha \otimes K_\alpha + 1 \otimes E_\alpha, & \Delta(F^{(k)}_\alpha) = \sum_{i+j=k} q_\alpha^{-\frac{ij}{2}} K_\alpha^{-i} F_\alpha^{(j)} \otimes F_\alpha^{(i)} \hspace{1cm} \Delta(K^{\pm 1}_\alpha) = K^{\pm 1}_\alpha \otimes K^{\pm 1}_\alpha
\end{align*}

\begin{theorem}\label{T:monoidality}
For any $N$ such that $1\le N < n$, the isomorphism $$\splits_N: \CHbm(D_n) \simeq \CHbm(D_{1,N}) \otimes \CHbm(D_{N+1,n})$$ is one of $\Uqg$-modules, where the action on the tensor product is given by the coproduct. 
\end{theorem}
\begin{proof}
The result is obvious for actions of $\CK_\alpha^{\pm 1}$. Now for the action of $\CF^{(1)}_\alpha$ on a typical basis vector:
\begin{align*}
q^{d_\alpha \frac{k(k-1)}{4}} \CF_\alpha^{(k)} (\mathcal{D})  & = \vcenter{\hbox{\begin{tikzpicture}[scale=0.4, every node/.style={scale=0.8},decoration={
    markings,
    mark=at position 0.5 with {\arrow{>}}}
    ]       
\coordinate (w0g) at (-5,1) {};
\coordinate (w0d) at (5,1) {};
\coordinate (h1) at (5,-2.5);
\coordinate (h2) at (5,-1.5);
\coordinate (h3) at (5,-1);
\draw[dashed, thick,postaction={decorate}] (w0g) -- (w0d) node[pos=0.15,black] (a1) {} node[pos=0.3,black,above]  {$k$} node[pos=0.55,black] (a2) {} coordinate[pos=0.8] (a3) {};
\coordinate (a33) at (5,2);
\coordinate (a22) at (5,2.5);
\draw[fill = black!20!white] (-5,-3) rectangle ++(10,3);
\node at (-2.5,-1.5) {$\mathcal{D}_{c_1}$};
\node at (2.5,-1.5) {$\mathcal{D}_{c_2}$};
\draw[\hyellow, double, thick] (-2.5,-1.7) -- (-2.5,-3);
\draw[\hyellow, double, thick] (2.5,-1.7) -- (2.5,-3);
\draw[thick] (0,-3)--(0,0);
\draw[red, thick] (-5,-3) -- (-5,3);
\draw[red, thick] (5,-3) -- (5,3);
\draw[red, thick] (-5,-3) -- (5,-3);
\draw[gray, thick] (-5,3) -- (5,3);
\draw[yellow!80!black,double,thick] (5,-3)--(a33)--(a33-|a3)--(a3);
\end{tikzpicture}}} \\
& =
\sum_{i+j= k} \left( q_{\alpha}^{\frac{ij}{2}} \prod_{p = N+1,\ldots,n} (s_{p,\alpha})^{i} \prod_{\beta \in \Pi} q_{\alpha,\beta}^{c_2(\beta)}  \right)
\vcenter{\hbox{\begin{tikzpicture}[scale=0.4, every node/.style={scale=0.8},decoration={
    markings,
    mark=at position 0.5 with {\arrow{>}}}
    ]       
\coordinate (w0g) at (-5,1) {};
\coordinate (w0d) at (5,1) {};
\coordinate (h1) at (5,-2.5);
\coordinate (h2) at (5,-1.5);
\coordinate (h3) at (5,-1);
\draw[fill = black!20!white] (-5,-3) rectangle ++(10,3);
\node at (-2.5,-1.5) {$\mathcal{D}_{c_1}$};
\node at (2.5,-1.5) {$\mathcal{D}_{c_2}$};
\draw[\hyellow, double, thick] (-2.5,-1.7) -- (-2.5,-3);
\draw[\hyellow, double, thick] (2.5,-1.7) -- (2.5,-3);
\draw[thick] (0,-3)--(0,0);
\draw[dashed, thick, postaction={decorate}] (w0g) to[bend left] node[pos=0.25,black,above]  {$i$} node[pos=0.4] (a3p) {} (-0.5,-3)  ;
\draw[dashed, thick, postaction={decorate}] (0.5,-3) to[bend left] node[pos=0.6,black,above]  {$j$} node[pos=0.8] (a3) {} (w0d);
\coordinate (a33p) at (0,-3);
\coordinate (a33) at (5,2);
\coordinate (a22) at (5,2.5);
\draw[red, thick] (-5,-3) -- (-5,3);
\draw[red, thick] (5,-3) -- (5,3);
\draw[red, thick] (-5,-3) -- (5,-3);
\draw[gray, thick] (-5,3) -- (5,3);
\draw[yellow!80!black,double,thick] (a33p)--(a33p|-a3p)--(a3p);
\draw[yellow!80!black,double,thick] (5,-3)--(a33)--(a33-|a3)--(a3);
\end{tikzpicture}}}
\end{align*}
where we have successively applied a cut of the dashed arc on $\partial^- D_n$ and a handle rule for the handle of the indexed $i$ new dashed arc to reach the left part of $\partial D_n$. More precisely, this $i$-indexed handle reaches a configuration with real parts strictly in between those of $w_N$ and $w_{N+1}$. Along this the handle to the indexed $i$ dashed arc winds half around punctures $w_{N+1}, \ldots , w_n$, around the configuration in $\mathcal{D}_{c_2}$, and around the handle to the indexed $j$ dashed arc, each time following the counterclockwise orientation. It justifies the coefficient. This turns the last diagram into a split position. Thus:
\begin{align*}
\splits_N \left( q^{d_\alpha \frac{k(k-1)}{4}} \CF_\alpha^{(k)} (\mathcal{D}) \right) & = \sum_{i+j= k} \left(  q_{\alpha}^{\frac{ij}{2}} \prod_{p = N+1,\ldots,n} (s_{p,\alpha})^{i} \prod_{\beta \in \Pi} q_{\alpha,\beta}^{c_2(\beta)} \right)
\vcenter{\hbox{\begin{tikzpicture}[scale=0.4, every node/.style={scale=0.8},decoration={
    markings,
    mark=at position 0.5 with {\arrow{>}}}
    ]       
\coordinate (w0g) at (-5,1) {};
\coordinate (w0d) at (5,1) {};
\coordinate (h1) at (5,-2.5);
\coordinate (h2) at (5,-1.5);
\coordinate (h3) at (5,-1);
\draw[fill = black!20!white] (0,-3) rectangle ++(5,3);
\node at (2.5,-1.5) {$\mathcal{D}_{c_2}$};
\draw[\hyellow, double, thick] (2.5,-1.7) -- (2.5,-3);
\draw[dashed, thick] (0.5,-3) to[bend left] node[pos=0.6,black,above]  {$j$} node[pos=0.8] (a3) {} (w0d);
\coordinate (a33p) at (-5,2);
\coordinate (a33) at (5,2);
\coordinate (a22) at (5,2.5);
\draw[red, thick] (0,-3) -- (0,3);
\draw[red, thick] (5,-3) -- (5,3);
\draw[red, thick] (0,-3) -- (5,-3);
\draw[gray, thick] (0,3) -- (5,3);
\draw[yellow!80!black,double,thick] (5,-3)--(a33)--(a33-|a3)--(a3);
\end{tikzpicture}}}
\otimes
\vcenter{\hbox{\begin{tikzpicture}[scale=0.4, every node/.style={scale=0.8},decoration={
    markings,
    mark=at position 0.5 with {\arrow{>}}}
    ]       
\coordinate (w0g) at (-5,1) {};
\coordinate (w0d) at (5,1) {};
\coordinate (h1) at (5,-2.5);
\coordinate (h2) at (5,-1.5);
\coordinate (h3) at (5,-1);
\draw[fill = black!20!white] (-5,-3) rectangle ++(5,3);
\node at (-2.5,-1.5) {$\mathcal{D}_{c_1}$};
\draw[\hyellow, double, thick] (-2.5,-1.7) -- (-2.5,-3);
\draw[dashed, thick] (w0g) to[bend left] node[pos=0.4,black,above]  {$i$} node[pos=0.2] (a3p) {} (-0.5,-3)  ;
\coordinate (a33p) at (0,-3);
\coordinate (a33) at (5,2);
\coordinate (a22) at (5,2.5);
\draw[red, thick] (-5,-3) -- (-5,3);
\draw[red, thick] (0,-3) -- (0,3);
\draw[red, thick] (-5,-3) -- (0,-3);
\draw[gray, thick] (-5,3) -- (0,3);
\draw[yellow!80!black,double, thick] (a33p)--(a33p|-a3p)--(a3p);
\end{tikzpicture}}}
\\
& =  \sum_{i+j= k} q_{\alpha}^{\frac{ij}{2}} \left( \prod_{p = N+1,\ldots,n} (s_{p,\alpha})^{i} \prod_{\beta \in \Pi} q_{\alpha,\beta}^{c_2(\beta)} \right)
\vcenter{\hbox{\begin{tikzpicture}[scale=0.4, every node/.style={scale=0.8},decoration={
    markings,
    mark=at position 0.5 with {\arrow{>}}}
    ]       
\coordinate (w0g) at (-5,1) {};
\coordinate (w0d) at (5,1) {};
\coordinate (h1) at (5,-2.5);
\coordinate (h2) at (5,-1.5);
\coordinate (h3) at (5,-1);
\draw[fill = black!20!white] (0,-3) rectangle ++(5,3);
\node at (2.5,-1.5) {$\mathcal{D}_{c_2}$};
\draw[\hyellow, double, thick] (2.5,-1.7) -- (2.5,-3);
\draw[dashed, thick] (0,1) to node[pos=0.5,black,above]  {$j$} node[pos=0.8,below] (a3) {} (w0d);
\coordinate (a33p) at (-5,2);
\coordinate (a33) at (5,2);
\coordinate (a22) at (5,2.5);
\draw[red, thick] (0,-3) -- (0,3);
\draw[red, thick] (5,-3) -- (5,3);
\draw[red, thick] (0,-3) -- (5,-3);
\draw[gray, thick] (0,3) -- (5,3);
\draw[yellow!80!black,double,thick] (5,-3)--(a33)--(a33-|a3)--(a3);
\end{tikzpicture}}} 
\otimes
\vcenter{\hbox{\begin{tikzpicture}[scale=0.4, every node/.style={scale=0.8},decoration={
    markings,
    mark=at position 0.5 with {\arrow{>}}}
    ]       
\coordinate (w0g) at (-5,1) {};
\coordinate (w0d) at (5,1) {};
\coordinate (h1) at (5,-2.5);
\coordinate (h2) at (5,-1.5);
\coordinate (h3) at (5,-1);
\draw[fill = black!20!white] (-5,-3) rectangle ++(5,3);
\node at (-2.5,-1.5) {$\mathcal{D}_{c_1}$};
\draw[\hyellow, double, thick] (-2.5,-1.7) -- (-2.5,-3);
\draw[dashed, thick] (w0g) to node[pos=0.4,black,above]  {$i$} node[pos=0.8,below] (a3p) {} (0,1)  ;
\coordinate (a33p) at (0,2);
\coordinate (a33) at (5,2);
\coordinate (a22) at (5,2.5);
\draw[red, thick] (-5,-3) -- (-5,3);
\draw[red, thick] (0,-3) -- (0,3);
\draw[red, thick] (-5,-3) -- (0,-3);
\draw[gray, thick] (-5,3) -- (0,3);
\draw[yellow!80!black,double, thick] (0,-3)--(a33p)--(a33p-|a3p)--(a3p);
\end{tikzpicture}}}
\end{align*}
where the second line is obtained by simple isotopies. One recognizes:
\begin{align*}
\splits_N \left( q^{d_\alpha \frac{k(k-1)}{4}} \CF_\alpha^{(k)} (\mathcal{D}) \right) & =  \sum_{i+j= k} \left(q_{\alpha}^{\frac{ij}{2}} \prod_{p = N+1,\ldots,n} (s_{p,\alpha})^{i} \prod_{\beta \in \Pi} q_{\alpha,\beta}^{c_2(\beta)} \right) q_\alpha^{\frac{i(i-1)}{4}} \CF_\alpha^{(i)} (\mathcal{D}_{c_1}) \otimes q_\alpha^{\frac{j(j-1)}{4}} \CF_\alpha^{(j)} (\mathcal{D}_{c_2}) \\
& = 
q_{\alpha}^{\frac{k(k-1)}{2}} \sum_{i+j= k} \left( q_\alpha^{-\frac{ij}{2}} \CK_{\alpha}^{-i} \CF_\alpha^{(j)} (\mathcal{D}_{c_2})  \otimes \CF_\alpha^{(i)} (\mathcal{D}_{c_1}) \right)
\end{align*}
It concludes the theorem for generators $\CF_\alpha$'s. The computation would go exactly the same for the action of $\CF_\alpha^{[1]}$ and left to the reader, and as it is the adjoint of $\CE_\alpha$ it concludes the proof. 
\end{proof}

\begin{coro}
By iteration of $\splits$, one obtains:
\[
\splits_{1,\ldots,n}: \CHbm(D_n) \simeq \CHbm(D_1)^{\otimes n},
\]
and hence we recover all tensor products of Verma modules. 
\end{coro}

\subsubsection{Braiding}\label{S:braiding}

We recall that the braid group on $n$ strands has the following definition as an Artin group:
$$\Bn = \left\langle \sigma_1,\ldots,\sigma_{n-1} \Big| \begin{array}{ll} \sigma_i \sigma_j = \sigma_j \sigma_i & \text{ if } |i-j| \ge 2 \\ 
\sigma_i \sigma_{i+1} \sigma_i = \sigma_{i+1} \sigma_i \sigma_{i+1} & \text{ for } i=1,\ldots, n-2 \end{array} \right\rangle$$
and that it has another definition as the mapping class group of $D_n$:
\[
\Bn = \Mod(D_n) 
\]
namely the group of isotopy classes of orientation preserving homeomorphisms of $D_n$ that restrict to the identity on $\partial D_n$ and preserves the set of punctures setwise. Then the generator $\sigma_i$ is the isotopy class of the half Dehn twist that swaps punctures $w_i$ and $w_{i+1}$. For $c \in \Pi$, and $[f] \in \Bn$ (the isotopy class of $f$), we have:
\[
f^{m_c} : \begin{array}{ccc}
\Conf_c & \to & \Conf_c \\
(\boldsymbol{z}^{\alpha_1}, \ldots , \boldsymbol{z}^{\alpha_l}) & \mapsto & (f^{c(\alpha_1)}(\boldsymbol{z}^{\alpha_1}), \ldots , f^{c(\alpha_l)}(\boldsymbol{z}^{\alpha_l}))
\end{array}
\] 
is the Cartesian product passed to the unordered configuration space, namely it applies $f$ diagonally to each coordinate. Just like Lemmas~6.33, and 6.35 in \cite{Jules_Verma}, $f^{m_c}$ has a unique lift at $\CH^e_c(D_n)$ and acts on it. If $[f]$ is a pure braid (that permutes the punctures pointwise) it promotes to a $\Laurent_{\boldsymbol{L}}$-module transformation and hence to a representation:
\[
R_c : \PBn \to \Aut_{\Laurent_{\boldsymbol{L}}}(\CH^e_c(D_n))
\]
of the pure braid group in $n$ strands. If moreover $L_1=\ldots = L_n$, it yields representations of the braid groups:
\[
R_c: \Bn \to \Aut_{\Laurent_{\boldsymbol{L}}}(\CH^e_c(D_n)).
\]
We refer the reader to \cite{Jules_Verma} since the well definedness of these representations adapts to the colored configuration spaces cases. We call them the family of Lawrence representations for the $\slt$ case, we call the present ones $\fg$-decorated Lawrence representations. Notice that the representations preserve colorings. Hence for any $\alpha$, $R_{2\alpha}$ is faithful, since it is trivially isomorphic to the $\slt$-case, and it was shown to be faithful by independent work of the first author \cite{Stephen_linearity} and Krammer \cite{Kra}. It is also true for $\rho_{k\alpha}$ when $k>2$, while the case $\rho_\alpha$ is not faithful for $n>4$, the case $n=4$ remaining open, it is the Burau representations, for the case $L_1 = \ldots = L_n=L$. In the case of arbitrary $\boldsymbol{L}$, and $n=1$ it is the Gassner representations of pure braid group and it is still not known whether it is faithful. 

\begin{prop}
The action of a Dehn twist by $R_c$ defines an $R$-matrix on $\CH^e(D_n) \simeq (\CH^e(D_1))^{\otimes n}$. Namely:
\[
R_c: \Bn \to \Aut_{\Laurent_{\boldsymbol{L}},\Uqg^E}(\CH^e_c(D_n)),
\]
and:
\[
\splits_{1,\ldots,n} (\rho_c(\sigma_i)) = \Id^{i-1} \otimes \Lambda(q,L) \mathrm{R} \circ T \otimes \Id^{n-i+1}
\]
where $R$ is in $\Aut(\CH^e(D_1) \otimes \CH^e(D_1))$, it is computable from $R_c(\sigma_1)$ and is called the $R$-matrix, $\Lambda(q,L) \in \Laurent_L$  and $T$ permutes trivially the tensor copies. 
\end{prop}
\begin{proof}
It is straightforward to check that the operators defining the action of $\Uqg^E$ come from chain maps, commuting with actions of continuous maps at homology. Now for a diagram in a position to split, it is clear that the action of a Dehn twist $\sigma_i$ is supported on a disk that contains only the part of the diagram concerned with the two punctures $w_i$ and $w_{i+1}$, one can think of the bases elements to be convinced. We leave the computation of the $R$-matrix to the reader.  
\end{proof}

We will further investigate this $R$-matrix and use it to study knot invariants associated with quantum groups from a homological perspective. 

\appendix

\section{Homologies and dualities}\label{A:the_pairing}

We refer the reader to \cite[Appendix~A]{JulesMarco} for detailed definitions of twisted Borel--Moore homology of configuration spaces in the relative case also. In there the twisted structure is defined from a representation of the fundamental group of configuration spaces while here we first use a map to a product of circles (Sec.~\ref{S:local_system_empty_disk}) and we choose a representation of the fundamental group of this product of circles instead. It results in a construction not depending on a choice of a (single) base point for configuration spaces but rather only one for the product of circles. In configuration spaces then, any vertically aligned configuration plays the role of base point (in comparison with \cite{JulesMarco} for instance where the base point is always the same on diagrams) and it simplifies a lot the computations in these colored configuration spaces. This trick is explained in Julien Marché's lecture notes (in french) \cite[Sec~2.2]{Mar} who also taught it to the second author. 

Notice that Configuration spaces (even unordered) admit a homotopy equivalent compactification constructed by Dev Sinha \cite{sinha}. As a consequence the Borel--Moore homology can be think as the homology of that compactification relative to its boundary, even in the twisted case thanks to the homotopy equivalence. We refer the reader to the proof of \cite[Coro.~3.23]{Jules_Lefschetz} for precise notation, since we won't use it in the present paper, while thinking of the Borel--More homology as a relative homology of a compact manifold could help.

Now let $\mathbb{D}$ denote either an empty disk, a disk with punctures or a disk with holes namely $D, D_n$ or $D_n^\circ$ following the notation of the paper. They are the three surfaces studied, and they are delivered with a marked part of their boundary, denoted respectively $\partial^-D, \partial^-D_n$ and $\partial^-D_n$, we use $\partial^-\mathbb{D}$ to designate either of them. We let $\CHbm_c(\mathbb{D})$ designate $\Hlf_{m_c}(\Conf_c(\mathbb{D}), S(\mathbb{D}); \Laurent_c(\mathbb{D}))$, where $\Laurent_c(\mathbb{D})$ is the appropriate local system, hence homologies have coefficients in an integral ring of Laurent polynomials in several variables. In Propositions~\ref{structure_result},~\ref{P:structure_result_punctures} we show:
\begin{enumerate}
\item $\CHbm_c(\mathbb{D})$ is free on $\Laurent_c(\mathbb{D})$ (exhibiting a diagrammatic basis),
\item It is the only non vanishing module of the whole homology.
\end{enumerate}

We let $\CH^\dagger_c(\mathbb{D}) := \Hnot_c(\Conf_c(\mathbb{D}, T(\mathbb{D}); \Laurent_c(\mathbb{D}))$, where $T(\mathbb{D})$ is made of configurations with at least one point in $\partial^+\mathbb{D} := \partial\mathbb{D} \setminus \partial^-\mathbb{D}$. Poincaré duality systematically provides isomorphisms:
\begin{align}
\label{E:Poincaré_duality_BM} & \CHbm_c(\mathbb(D)) \simeq \Hnot^{m_c}(\Conf_c(\mathbb{D}),T(\mathbb{D}); \Laurent_c(\mathbb{D})) \\
\label{E:Poincaré_duality} & \CH_c^\dagger(\mathbb{D}) \simeq \Hnot_{\mathrm{compact}}^{m_c}(\Conf_c(\mathbb{D}),S(\mathbb{D}); \Laurent_c(\mathbb{D}))
\end{align}
where a shift must be considered if one wants to work elsewhere than the middle dimension. It is standard Poincaré duality for manifold with boundary, where here it is useful to think as Borel--Moore to be relative to part of the boundary, and $\Hnot_{\mathrm{compact}}$ is the compactly supported cohomology that is the one obtained from the dual complex as that of the Borel--Moore homology. 

\begin{prop}\label{P:standard_homology_is_free}
The module $\CH_c^\dagger(\mathbb{D})$ is free on $\Laurent_c(\mathbb{D})$. 
\end{prop}
\begin{proof}
Since the Borel--Moore homology is concentrated in this middle dimension, all the lower terms vanish. The universal coefficient theorem in that case simply provides an isomorphism:
\[
\Hnot_{\mathrm{compact}}^{m_c}(\Conf_c(\mathbb{D}),S(\mathbb{D}); \Laurent_c(\mathbb{D})) \simeq \Hom_{\Laurent_c}(\CHbm_c(\mathbb{D}), \Laurent_c).
\]
Notice that while working with a ring of integral Laurent polynomials, which is not a PID, one must rather use the universal coefficients spectral sequence. However, in our case, since differentials have mixed degree, the spectral sequence converges immediately and provides the above. 

Since $\CHbm_c(\mathbb{D})$ is free, so is $\Hom_{\Laurent_c}(\CHbm_c(\mathbb{D}), \Laurent_c)$, so is $\Hnot_{\mathrm{compact}}^{m_c}(\Conf_c(\mathbb{D}),S(\mathbb{D}); \Laurent_c(\mathbb{D}))$. By Poincaré duality \eqref{E:Poincaré_duality}, $\CH_c^\dagger(\mathbb{D})$ is free too. 
\end{proof}

Now we introduce a pairing that finds a basis of $\CH_c^\dagger(\mathbb{D})$ dual to that of $\CHbm_c(\mathbb{D})$. Namely, there exists an intersection pairing which is a bilinear form:
\begin{equation}\label{E:the_pairing}
\langle \cdot ,\cdot \rangle_{\mathbb{D}}: \CH_c^\dagger(\mathbb{D}) \otimes \CHbm_c(\mathbb{D}) \to \Laurent_c(\mathbb{D}).
\end{equation}
It has been widely studied for these type of homologies, and its most famous strength is that it detects the geometric intersection of isotopy classes of curves which is the heart of the proof of the linearity of braid groups by the first author \cite{Stephen_linearity}. The way to compute it from diagrammatic classes has also been described many times, in \cite{Stephen_linearity} for instance and also in \cite[Sec.4.2]{Jules_Lefschetz}. We give a short summary of that protocole. For two given diagrams $\mathcal{D}^\dagger \in \CH_c^\dagger(\mathbb{D})$ and $\mathcal{D} \in \CHbm_c(\mathbb{D})$, an intersection configuration $p$ is a configuration in $\Conf_c(\mathbb{D})$ that lies on both manifolds described by the diagram. Then: 
\[
\langle \mathcal{D}^\dagger ,\mathcal{D} \rangle_{\mathbb{D}} = \sum_p \varepsilon_p \kappa_p
\]
where $p$ ranges on intersection configurations, $\varepsilon_p \in \{\pm 1\}$ and $\kappa_p \in \Laurent_c$. The sign is computed from the orientation of the manifolds and thus from that of arcs composing the diagram. More interesting is how to compute $\kappa_p$. It uses a loop denoted $\gamma_p$ of $\Conf_c(\mathbb{D})$. It is defined by a composition of paths:
\begin{itemize}
\item First a path going from a vertically aligned configuration to the manifold that $\mathcal{D}^\dagger$ describes following the handle of $\mathcal{D}^\dagger$. 
\item Then reaching $p$ following $\mathcal{D}^\dagger$, noticing that there is only one possible choice,
\item Then reaching the handle of $\mathcal{D}$ traveling along $\mathcal{D}$ again with one possible choice,
\item Finally reaching a vertically aligned configuration following the handle of $\mathcal{D}$. 
\end{itemize}
Then $\kappa_p$ is the image of this loop by the local system. We denote by $\rot^+$ the positive $\pi/2$ rotation of $\mathbb{D}$ and $\flip$ the flip map along the middle horizontal axis of $\mathbb{D}$. These transformations help passing from homologies relative to $S$, to those relative to $T$ without drawing too many diagrams, since $\rot^+(\partial^-D)=\partial^+D$ and $\flip(\partial^-D_n) = \partial^+D_n$ (sometimes with a little stretch of the boundary transparent in homology). We have:
\begin{enumerate}
\item For $(r'_1,\ldots,r'_k), (r_1,\ldots,r_k)$ both lists in $\CP_c$:
\[
\langle \rot^+(\CF^{[r'_1,\ldots,r'_k]}) , \CF_{(r_1,\ldots,r_k)} \rangle_D = \delta_{\left( (r'_1,\ldots,r'_k), (r_1,\ldots,r_k)  \right)}
\]
following notations for diagrams in Sec.~\ref{S:twisted_homology_bases_etc}. 
\item And for $(\br'^1, \ldots, \br'^n),(\br^1, \ldots, \br^n)$ both lists in $\CP_c$:
\[
\langle \flip\left( \CF^{[\br'^1, \ldots, \br'^n]} \right) , \CF_{(\br^1, \ldots, \br^n)} \rangle_{D_n} = \delta_{\left( (\br'^1, \ldots, \br'^n), (\br^1, \ldots, \br^n)  \right)}
\]
following notation from Sec.~\ref{S:twisted_homology_punctures_bases_etc}.
\end{enumerate}
In the above, $\delta$ denotes the Kronecker symbol for lists. In both cases the computation is easy, first if the two lists are different, then there is not any intersection configuration, and if the two lists are the same there is only one. Computing the sign is a matter of convention, with the one we chose for diagrams it is one. The monomial in the ring of coefficients is easy to compute, it is one in both cases. For the second one the computation is exactly that of \cite[Fig.~5.8]{Pierre} that was also done in \cite[Prop.~66]{JulesSonny}. The first case is even more trivial to compute. 

\begin{prop}\label{P:perfect_pairing}
The intersection pairing $\langle \cdot , \cdot \rangle_{\mathbb{D}}$ is a perfect pairing. 
\end{prop}
\begin{proof}
Since the classes involved on the Borel--Moore side of those pairings are the bases of $\CHbm(D)$ resp $\CHbm(D_n)$ according to Propositions~\ref{structure_result} resp. \ref{P:structure_result_punctures}, and since both modules are free, it shows that the pairing is perfect, and it provides bases for the standard homology.
\end{proof}
 The bases it so furnished for the standard homologies are the ones involved in Propositions~\ref{T:structure_standard_homology} resp. \ref{P:structure_result_punctures_standard}.

\end{document}